\documentclass[a4paper]{article}%
\usepackage{hyperref}
\usepackage[authoryear,comma]{natbib}
\usepackage{amsmath}
\usepackage{amsfonts}
\usepackage{subfig,caption}
\usepackage{amssymb}
\usepackage{color}
\usepackage{graphicx}
\usepackage{caption}
\usepackage{subfloat}%
\providecommand{\U}[1]{\protect\rule{.1in}{.1in}}
\hypersetup{citecolor=blue,colorlinks=true}
\hypersetup{citecolor=blue,colorlinks=true}
\newtheorem{theorem}{Theorem}

\newtheorem{corollary}[theorem]{Corollary}

\newtheorem{definition}[theorem]{Definition}

\newtheorem{remark}[theorem]{Remark}

\newenvironment{proof}[1][Proof]{\noindent\textbf{#1.} }{\ \rule{0.5em}{0.5em}}
\begin{document}

\date{\today}
\title{Influence Analysis of Robust Wald-type Tests\thanks{This paper was supported
by Ministerio de Econom\'{\i}a y Competitividad of Spain, Grant
MTM-2012-33740.}}
\author{Abhik Ghosh$^{1}$, Abhijit Mandal$^{2}$, Nirian Mart\'{\i}n$^{3}$, Leandro
Pardo$^{3}$
\and $^{1}${\small Indian Statistical Institute, Kolkata, India}
\and $^{2}${\small Univesity of Minnesota, Minneapolis, USA}
\and $^{3}${\small Complutense University of Madrid, Madrid, Spain}}
\maketitle

\begin{abstract}
We consider a robust version of the classical Wald test statistics for testing
simple and composite null hypotheses for general parametric models. These test
statistics are based on the minimum density power divergence estimators
instead of the maximum likelihood estimators. An extensive study of their
robustness properties is given though the influence functions as well as the
chi-square inflation factors. It is theoretically established that the level
and power of these robust tests are stable against outliers, whereas the
classical Wald test breaks down. Some numerical examples confirm the validity
of the theoretical results.

\end{abstract}

\bigskip



\section{Introduction\label{sec1}}

Testing statistical hypothesis is an important area within the class of
statistical inference procedures. Most widely used and popular classical tests
are based on the likelihood ratio, score and Wald test statistics. Although
they enjoy several optimum asymptotic properties, they are highly non-robust
in case of model misspecification and presence of outlying observations. It is
well-known that a small deviation from the underlying assumptions on the model
can have drastic effect on the performance of these classical tests.
%
So, the practical importance of a robust test procedure is beyond doubt; and
it is helpful for solving several real life problems containing some outliers
in the observed sample.

The purpose in robust testing of hypothesis is two-fold. A good robust test
should exhibit stability under small, arbitrary departures from the null
hypothesis (robustness of validity), and should have good power under small,
arbitrary departures from specified alternatives (robustness of efficiency).
However, these robustness aspects of a test are not widely explored as
compared to the robustness of the estimators. Hample's influence function
(\citealp{Hampel74}) gives an important measure of robustness to investigate
the local stability along with the global reliability of an estimator.
\citet{Ron79,Ron82a,Ron82b}
and
\citet{RouRon79,RouRon81}
have extended the concept of an influence function in testing a null
hypothesis about a scalar parameter (see \citealp[Chapter 3]{Hampel86}).
Besides considering the influence function of the test statistic, they have
also proposed to study the behavior of the level and power of the test as
functions of an additional observation at any point $\mathbf{x}$ -- it
reflects the influence of the additional infinitesimal contamination on the
level and power of the test. An essential result of this approach is the
approximation of the asymptotic level and power under a contaminated
distribution in a neighborhood of the null hypothesis. A very nice review
about the influence function in the study of robustness of a test statistic is
given in
\citet{MarRon97}%
. The idea of influence function analysis has been studied extensively in
different tests by
\citet{Cantoni}%
,
\citet{Trojani}%
,
\citet{Wang}
and
\citet{VanAelst}
Recently,
\citet{Toma2010}%
,
\citet{Toma}%
,
\citet{2014arXiv1404.5126G}
derived some important results for the tests based on the divergence measures.

In this paper we explore the theoretical robustness properties for a class of
Wald-type tests recently proposed by
\citet{2014arXiv1403.7616B}%
. The family of tests is based on the minimum density power divergence
estimators (MDPDE); and it has been developed for testing both simple and
composite null hypotheses.
\citet{2014arXiv1403.7616B}
have empirically demonstrated that the Wald-type test exhibits strong
robustness properties, but relevant theoretical results supporting the
empirical findings are not derived. Here, we will fill that gap by developing
some theoretical results on robustness for the general Wald-type tests based
on the influence function analysis. In comparison with the paper by
\citet{Heritier}%
, where robustness of some Wald-type tests with M-estimators are studied, our
paper covers more general composite hypothesis testing, since it is not
restricted only on linear transformations. Moreover, other than level and
power influence functions we have also studied the chi-square inflation factor
which measures an overall departure of the test statistic from the null
distribution due to contamination.

The rest of the paper is organized as follows. In Section \ref{sec2} we have
presented some notations and results from
\citet{2014arXiv1403.7616B}
which are necessary to develop further theoretical results for this paper.
Section \ref{sec3} presents the influence functions of the Wald-type test
statistics. The power and level influence functions for testing simple and
composite null hypotheses are derived in Section \ref{Sec4}. The chi-square
inflation factors for Wald-type test statistics are calculated in Section
\ref{sec5}. In Section \ref{sec6} we have presented some examples to justify
the theoretical results developed in this paper. A discussion on choosing the
tuning parameter for the density power divergence measure is given in Section
\ref{sec7}, and finally, some concluding remarks are provided in Section
\ref{sec8}.

\section{Preliminaries\label{sec2}}

Let $\mathcal{G}$ denote the set of all distributions having densities with
respect to a dominating measure (generally the Lebesgue measure or the
counting measure). Given any two densities $g$ and $f$ in $\mathcal{G}$, the
density power divergence with a nonnegative tuning parameter $\beta$, is
defined as
\begin{equation}
d_{\beta}(g,f)=\left\{
\begin{array}
[c]{ll}%
\int\left\{  f^{1+\beta}(\boldsymbol{x})-\left(  1+\frac{1}{\beta}\right)
f^{\beta}(\boldsymbol{x})g(\boldsymbol{x})+\frac{1}{\beta}g^{1+\beta
}(\boldsymbol{x})\right\}  d\boldsymbol{x}, & \text{for}\mathrm{~}%
\beta>0,\\[2ex]%
\int g(\boldsymbol{x})\log\left(  \frac{g(\boldsymbol{x})}{f(\boldsymbol{x}%
)}\right)  d\boldsymbol{x}, & \text{for}\mathrm{~}\beta=0.
\end{array}
\right.  \label{2.1}%
\end{equation}
The divergence corresponding to $\beta=0$ may be derived from the general case
by taking the continuous limit as $\beta\rightarrow0$, and in this case
$d_{0}(g,f)$ turns out to be the Kullback-Leibler divergence.
Details about the inference based on divergence measures can be found in
\citet{AyanBook}
and
\citet{LeandroBook}%
.

We consider a parametric model of densities $\{f_{\boldsymbol{\theta}%
}:{\boldsymbol{\theta}}\in\Theta\subset%
\mathbb{R}
^{p}\}$, and we are interested in the estimation of $\boldsymbol{\theta}$. Let
$G$ represent the distribution function corresponding to the density $g$ that
generates the data. The minimum density power divergence functional at $G$,
denoted by $\boldsymbol{T}_{\beta}(G)$, is defined as
\begin{equation}
d_{\beta}(g,f_{\boldsymbol{T}_{\beta}(G)})=\min_{\boldsymbol{\theta}\in\Theta
}d_{\beta}(g,f_{\boldsymbol{\theta}}). \label{22.2}%
\end{equation}
Therefore the MDPDE of $\boldsymbol{\theta}$ is given by%
\begin{equation}
\widehat{\boldsymbol{\theta}}_{\beta}=\boldsymbol{T}_{\beta}(G_{n}),
\label{22.3}%
\end{equation}
where $G_{n}$ is the empirical distribution function associated with a random
sample $\boldsymbol{X}_{1},\ldots,\boldsymbol{X}_{n}$ from the population with
density $g$ (having distribution function $G$). As the last term of equation
(\ref{2.1}) does not depend on $\boldsymbol{\theta}$,
$\widehat{\boldsymbol{\theta}}_{\beta}$ is given by%
\begin{equation}
\widehat{\boldsymbol{\theta}}_{\beta}=\arg\min_{\boldsymbol{\theta}\in\Theta
}\left\{  \int f_{\boldsymbol{\theta}}^{1+\beta}(\boldsymbol{x}%
)d\boldsymbol{x}-\left(  1+\frac{1}{\beta}\right)  \frac{1}{n}\sum_{i=1}%
^{n}f_{\boldsymbol{\theta}}^{\beta}(\boldsymbol{X}_{i})\right\}  ,
\label{22.5}%
\end{equation}
if $\beta>0$ and
\begin{equation}
\widehat{\boldsymbol{\theta}}_{\beta}=\arg\min_{\boldsymbol{\theta}\in\Theta
}\left\{  -\frac{1}{n}\sum_{i=1}^{n}\log f_{\boldsymbol{\theta}}%
(\boldsymbol{X}_{i})\right\}  , \label{22.6}%
\end{equation}
when $\beta=0.$ Notice that $\widehat{\boldsymbol{\theta}}_{\beta}$ for
$\beta=0$ coincides with the maximum likelihood estimator (MLE). In
\citet{MR1665873}, it was established that the MDPDE is an M-estimator.


The functional $\boldsymbol{T}_{\beta}(G)$ is Fisher consistent; it takes the
value $\boldsymbol{\theta}$$_{0}$, the true value of the parameter, when the
true density is a member of the model, i.e. $g=f_{\boldsymbol{\theta}_{0}}$.
Let us assume $g=f_{\boldsymbol{\theta}_{0}}$, and define the quantities%
\begin{equation}
\boldsymbol{J}_{\beta}\left(  \boldsymbol{\theta}\right)  =%
{\displaystyle\int}
\boldsymbol{u}_{\boldsymbol{\theta}}(\boldsymbol{x})\boldsymbol{u}%
_{\boldsymbol{\theta}}^{T}(\boldsymbol{x})f_{\boldsymbol{\theta}}^{1+\beta
}(\boldsymbol{x})d\boldsymbol{x},\text{\quad}\boldsymbol{K}_{\beta}\left(
\boldsymbol{\theta}\right)  =%
{\displaystyle\int}
\boldsymbol{u}_{\boldsymbol{\theta}}(\boldsymbol{x})\boldsymbol{u}%
_{\boldsymbol{\theta}}^{T}(\boldsymbol{x})f_{\boldsymbol{\theta}}^{1+2\beta
}(\boldsymbol{x})d\boldsymbol{x}-\boldsymbol{\xi}_{\beta}\left(
\boldsymbol{\theta}\right)  \boldsymbol{\xi}_{\beta}^{T}\left(
\boldsymbol{\theta}\right)  , \label{22.9}%
\end{equation}
where
\[
\boldsymbol{\xi}_{\beta}\left(  \boldsymbol{\theta}\right)  =%
{\displaystyle\int}
\boldsymbol{u}_{\boldsymbol{\theta}}(\boldsymbol{x})f_{\boldsymbol{\theta}%
}^{1+\beta}(\boldsymbol{x})d\boldsymbol{x}~~\text{and}~~\boldsymbol{u}%
_{\boldsymbol{\theta}}(\boldsymbol{x})=\frac{\partial}{\partial
\boldsymbol{\theta}}\log f_{\boldsymbol{\theta}}(\boldsymbol{x}).
\]
Then, following
\citet{MR1665873}
and
\citet{AyanBook}
, it can be shown that
\begin{equation}
n^{1/2}({\boldsymbol{\widehat{\boldsymbol{\theta}}}}_{\beta}%
-\boldsymbol{\theta}_{0})\underset{n\rightarrow\infty}{\overset{\mathcal{L}%
}{\longrightarrow}}\mathcal{N}(\boldsymbol{0}_{p},\boldsymbol{\Sigma}_{\beta
}(\boldsymbol{\theta}_{0})), \label{asymp_dist}%
\end{equation}
where
\begin{equation}
\boldsymbol{\Sigma}_{\beta}(\boldsymbol{\theta}_{0})=\boldsymbol{J}_{\beta
}^{-1}(\boldsymbol{\theta}_{0})\boldsymbol{K}_{\beta}(\boldsymbol{\theta}%
_{0})\boldsymbol{J}_{\beta}^{-1}(\boldsymbol{\theta}_{0}). \label{var}%
\end{equation}

\subsection{Wald-type Test Statistics for the Simple Null
Hypothesis\label{sec:simple_null}}

In
\citet{2014arXiv1403.7616B}
the family of Wald-type test statistics
\begin{equation}
W_{n}^{0}({\boldsymbol{\widehat{\boldsymbol{\theta}}}}_{\beta}%
)=n({\boldsymbol{\widehat{\boldsymbol{\theta}}}}_{\beta}-\boldsymbol{\theta
}_{0})^{T}\boldsymbol{\Sigma}_{\beta}^{-1}(\boldsymbol{\theta}_{0}%
)({\boldsymbol{\widehat{\boldsymbol{\theta}}}}_{\beta}-\boldsymbol{\theta}%
_{0}) \label{22.11}%
\end{equation}
was considered for testing the simple null hypothesis%
\begin{equation}
H_{0}:\boldsymbol{\theta}=\boldsymbol{\theta}_{0}\ \text{against}%
\ H_{1}:\boldsymbol{\theta}\neq\boldsymbol{\theta}_{0}, \label{22.10}%
\end{equation}
where $\boldsymbol{\theta}_{0}\in\Theta\subset%
\mathbb{R}
^{p}$. The asymptotic distribution of $W_{n}^{0}%
({\boldsymbol{\widehat{\boldsymbol{\theta}}}}_{\beta})$, defined in
(\ref{22.11}), is a chi-square with $p$ degrees of freedom. In the particular
case when $\beta=0$, i.e., the MDPDE coincides with the MLE, the
variance-covariance matrix, (\ref{var}), coincides with the inverse of the
Fisher information matrix of the model and then we get the classical Wald test
statistic for testing (\ref{22.10}). The power function $\beta_{W_{n}^{0}}$ of
the Wald-type test statistics at {$\boldsymbol{\theta}$}$^{\ast}\in
\Theta-\{\boldsymbol{\theta}_{0}\}$, is given by
\begin{equation}
\beta_{W_{n}^{0}}\left(  {\boldsymbol{\theta}}^{\ast}\right)  \cong%
1-\Phi\left(  \frac{\sqrt{n}}{\sigma_{W_{n}^{0}}\left(  {\boldsymbol{\theta}%
}^{\ast}\right)  }\left(  \frac{\chi_{p,\alpha}^{2}}{n}-\ell\left(
{\boldsymbol{\theta}}^{\ast}\right)  \right)  \right)  , \label{22.10.1}%
\end{equation}
where%
\[
\ell({\boldsymbol{\theta}}^{\ast})=\left(  {\boldsymbol{\theta}}^{\ast
}-\boldsymbol{\theta}_{0}\right)  ^{T}\boldsymbol{\Sigma}_{\beta}%
^{-1}(\boldsymbol{\theta}_{0})\left(  {\boldsymbol{\theta}}^{\ast
}-{\boldsymbol{\theta}}_{0}\right)  ,
\]%
\[
\sigma_{W_{n}^{0}}^{2}\left(  {\boldsymbol{\theta}}^{\ast}\right)  =4\left(
{\boldsymbol{\theta}}^{\ast}-{\boldsymbol{\theta}}_{0}\right)  ^{T}%
\boldsymbol{\Sigma}_{\beta}^{-1}({\boldsymbol{\theta}}^{\ast})\left(
{\boldsymbol{\theta}}^{\ast}-{\boldsymbol{\theta}}_{0}\right)  .
\]
Here $\alpha$ is the level of the test, $\chi_{p,\alpha}^{2}$ is the
$100(1-\alpha)$-th percentile of a chi-square distribution with $p$ degrees of
freedom and $\Phi(\cdot)$ is the standard normal distribution function. It is
clear that
\[
\lim_{n\rightarrow\infty}\beta_{W_{n}^{0}}({\boldsymbol{\theta}}^{\ast})=1,
\]
for all $\alpha\in\left(  0,1\right)  .$ Therefore the test is consistent in
the sense of
\citet{MR0083868}%
.

In order to produce a nontrivial asymptotic power, we can consider contiguous
alternative hypotheses. Consider the contiguous alternative hypotheses
described by%
\begin{equation}
H_{1,n}:\boldsymbol{\theta}_{n}=\boldsymbol{\theta}_{0}+n^{-1/2}%
\boldsymbol{d}, \label{22.12}%
\end{equation}
where ${\boldsymbol{d}}$ is a fixed vector in $\mathbb{R}^{p}$ such that
$\boldsymbol{\theta}$$_{n}\in\Theta\subset\mathbb{R}^{p}$. It can be shown
that the asymptotic distribution of the Wald-type test statistic $W_{n}%
^{0}({\boldsymbol{\widehat{\boldsymbol{\theta}}}}_{\beta})$ under the
alternative $H_{1,n}$\ is a non-central chi-square with $p$ degrees of freedom
and non-centrality parameter%
\begin{equation}
\delta=\boldsymbol{d}^{T}\boldsymbol{\Sigma}_{\beta}(\boldsymbol{\theta}%
_{0})\boldsymbol{d}. \label{22.13}%
\end{equation}
Based on this result, under (\ref{22.12}) we have the following approximation
to the power function%
\begin{equation}
\beta_{W_{n}^{0}}\left(  \boldsymbol{\theta}_{n}\right)  =1-F_{\chi_{p}%
^{2}(\delta)}\left(  \chi_{p,\alpha}^{2}\right)  , \label{22.13.1}%
\end{equation}
where $F_{\chi_{p}^{2}(\delta)}\left(  \cdot\right)  $ is the distribution
function of a non-central chi-square random variable with $p$ degrees of
freedom and non-centrality parameter $\delta.$

\subsection{Wald-type Test Statistics for the Composite Null
Hypothesis\label{Sec2.2}}

We shall now consider the problem of testing the composite null hypothesis
given by%
\begin{equation}
H_{0}:\boldsymbol{\theta}\in\Theta_{0}~\text{against}~H_{1}:\boldsymbol{\theta
}\notin\Theta_{0}, \label{22.16}%
\end{equation}
where $\Theta_{0}$ is a subset of the parameter space $\Theta\in\mathbb{R}%
^{p}$. The restricted parameter space $\Theta_{0}$ is often defined by a set
of $r$ restrictions of the form
\begin{equation}
\boldsymbol{m}(\boldsymbol{\theta})=\boldsymbol{0}_{r}, \label{22.14}%
\end{equation}
where $\boldsymbol{m}:\mathbb{R}^{p}\rightarrow\mathbb{R}^{r}$ with $r\leq p$
(see
\citealp{MR595165}%
). So $\Theta_{0}=\{$$\boldsymbol{\theta}\in\Theta:$ $\boldsymbol{m}%
($$\boldsymbol{\theta}$$)=\boldsymbol{0}_{r}\}$. Assume that the $p\times r$
matrix
\begin{equation}
\boldsymbol{M}(\boldsymbol{\theta})=\frac{\partial\boldsymbol{m}%
^{T}(\boldsymbol{\theta})}{\partial\boldsymbol{\theta}} \label{22.15}%
\end{equation}
exists and is continuous in all $\boldsymbol{\theta}$\ belonging to a
neighbourhood of the true value of $\boldsymbol{\theta}$, $\boldsymbol{\theta
}_{0}$, and $\mathrm{rank}\left(  \boldsymbol{M}(\boldsymbol{\theta}%
_{0})\right)  =r$.%

\citet{2014arXiv1403.7616B}
have considered the following family of Wald-type test statistics
\begin{equation}
W_{n}({\boldsymbol{\widehat{\boldsymbol{\theta}}}}_{\beta})=n\boldsymbol{m}%
^{T}({\boldsymbol{\widehat{\boldsymbol{\theta}}}}_{\beta})\left(
\boldsymbol{M}^{T}({\boldsymbol{\widehat{\boldsymbol{\theta}}}}_{\beta
})\boldsymbol{\Sigma}_{\beta}({\boldsymbol{\widehat{\boldsymbol{\theta}}}%
}_{\beta})\boldsymbol{M}({\boldsymbol{\widehat{\boldsymbol{\theta}}}}_{\beta
})\right)  ^{-1}\boldsymbol{m}({\boldsymbol{\widehat{\boldsymbol{\theta}}}%
}_{\beta}), \label{22.17}%
\end{equation}
where the matrix $\boldsymbol{\Sigma}_{\beta}({\boldsymbol{\cdot}})$ is
defined in (\ref{var}). The asymptotic distribution of the Wald-type test
statistic $W_{n}({\boldsymbol{\widehat{\boldsymbol{\theta}}}}_{\beta})$ under
the composite null hypothesis (\ref{22.16}) is a chi-square with $r$ degrees
of freedom.

In the special case when $\beta=0$, ${\boldsymbol{\widehat{\boldsymbol{\theta
}}}}_{\beta}$ coincides with the maximum likelihood estimator of
$\boldsymbol{\theta}$, and $\boldsymbol{\Sigma}$$_{\beta}(\boldsymbol{\cdot})$
becomes the inverse of the Fisher information matrix. Thus, the statistic in
(\ref{22.17}) reduces to the classical Wald test statistic.

The power function $\beta_{W_{n}}({\boldsymbol{\theta}}^{\ast})$ of the
Wald-type test statistic at {$\boldsymbol{\theta}$}$^{\ast}\in\Theta
-\Theta_{0}$, is given by
\begin{equation}
\beta_{W_{n}}\left(  {\boldsymbol{\theta}}^{\ast}\right)  \cong1-\Phi\left(
\frac{\sqrt{n}}{\sigma_{W_{n}}\left(  {\boldsymbol{\theta}}^{\ast}\right)
}\left(  \frac{\chi_{r,\alpha}^{2}}{n}-\ell^{\ast}\left(  {\boldsymbol{\theta
}}^{\ast},{\boldsymbol{\theta}}^{\ast}\right)  \right)  \right)  ,
\label{22.17.1}%
\end{equation}
where%
\[
\ell^{\ast}({\boldsymbol{\theta}}_{1},{\boldsymbol{\theta}}_{2}%
)=n\boldsymbol{m}^{T}\left(  {\boldsymbol{\theta}}_{1}\right)  \left(
\boldsymbol{M}^{T}({\boldsymbol{\theta}}_{2})\boldsymbol{\Sigma}_{\beta
}(\boldsymbol{\theta}_{2})\boldsymbol{M}({\boldsymbol{\theta}}_{2})\right)
^{-1}\boldsymbol{m}\left(  {\boldsymbol{\theta}}_{1}\right)  ,
\]
and
\begin{equation}
\sigma_{W_{n}}^{2}\left(  {\boldsymbol{\theta}}^{\ast}\right)  =\left.
\frac{\partial\ell^{\ast}({\boldsymbol{\theta}},{\boldsymbol{\theta}}^{\ast}%
)}{\partial{\boldsymbol{\theta}}^{T}}\right\vert _{{\boldsymbol{\theta}%
}={\boldsymbol{\theta}}^{\ast}}\boldsymbol{\Sigma}_{\beta}({\boldsymbol{\theta
}}^{\ast})\left.  \frac{\partial\ell^{\ast}({\boldsymbol{\theta}%
},{\boldsymbol{\theta}}^{\ast})}{\partial{\boldsymbol{\theta}}}\right\vert
_{{\boldsymbol{\theta}}={\boldsymbol{\theta}}^{\ast}}. \label{4.3}%
\end{equation}
%

\citet{2014arXiv1403.7616B}
proposed an approximation of the power of $W_{n}%
({\boldsymbol{\widehat{\boldsymbol{\theta}}}}_{\beta})$ at an alternative
hypothesis close to the null hypothesis. Let $\boldsymbol{\theta}$$_{n}%
\in\Theta-\Theta_{0}$ be a given alternative, and let $\boldsymbol{\theta}%
$$_{0}$ be the element in $\Theta_{0}$ closest to $\boldsymbol{\theta}$$_{n}$
in terms of the Euclidean distance. One possibility to introduce contiguous
alternative hypotheses, in this context, is to consider a fixed vector
$\boldsymbol{d}\in\mathbb{R}^{p}$ and permit $\boldsymbol{\theta}$$_{n}$ to
move towards $\boldsymbol{\theta}$$_{0}$ as $n$ increases through the relation
$H_{1,n}$ given in (\ref{22.12}). A second approach is to relax the condition
$\boldsymbol{m}\left(  \boldsymbol{\theta}\right)  =\boldsymbol{0}_{r}$ that
defines $\Theta_{0}$. Let $\boldsymbol{\delta}\in\mathbb{R}^{r}$ and consider
the following sequence of parameters $\{$$\boldsymbol{\theta}$$_{n}\}$ moving
towards $\boldsymbol{\theta}$$_{0}$ according to the set up
\begin{equation}
H_{1,n}^{\ast}:\boldsymbol{m}\left(  \boldsymbol{\theta}_{n}\right)
=n^{-1/2}\boldsymbol{\delta}. \label{22.19}%
\end{equation}
Note that a Taylor series expansion of $\boldsymbol{m}\left(
\boldsymbol{\theta}_{n}\right)  $ around $\boldsymbol{\theta}$$_{0}$ yields
\begin{equation}
\boldsymbol{m}\left(  \boldsymbol{\theta}_{n}\right)  =\boldsymbol{m}\left(
\boldsymbol{\theta}_{0}\right)  +\boldsymbol{M}^{T}(\boldsymbol{\theta}%
_{0})\left(  \boldsymbol{\theta}_{n}-\boldsymbol{\theta}_{0}\right)  +o\left(
\left\Vert \boldsymbol{\theta}_{n}-\boldsymbol{\theta}_{0}\right\Vert \right)
. \label{22.20}%
\end{equation}
By substituting $\boldsymbol{\theta}$$_{n}=\boldsymbol{\theta}$$_{0}%
+n^{-1/2}\boldsymbol{d}$ in (\ref{22.20}) and taking into account
that\textbf{\ }$\boldsymbol{m}($$\boldsymbol{\theta}$$_{0})=\boldsymbol{0}%
_{r}$, we get
\begin{equation}
\boldsymbol{m}\left(  \boldsymbol{\theta}_{n}\right)  =n^{-1/2}\boldsymbol{M}%
^{T}(\boldsymbol{\theta}_{0})\boldsymbol{d}+o\left(  \left\Vert
\boldsymbol{\theta}_{n}-\boldsymbol{\theta}_{0}\right\Vert \right)  .
\label{22.21}%
\end{equation}
So, the equivalence relationship between the hypotheses $H_{1,n}$ and
$H_{1,n}^{\ast}$ is
\begin{equation}
\boldsymbol{\delta}=\boldsymbol{M}^{T}(\boldsymbol{\theta}_{0})\boldsymbol{d}%
\text{ as }n\rightarrow\infty. \label{22.22}%
\end{equation}
The asymptotic distribution of $W_{n}({\boldsymbol{\widehat{\boldsymbol{\theta
}}}}_{\beta})$ is given by%
\begin{equation}
W_{n}({\boldsymbol{\widehat{\boldsymbol{\theta}}}}_{\beta}%
)\underset{n\rightarrow\infty}{\overset{\mathcal{L}}{\longrightarrow}}\chi
_{r}^{2}\left(  \boldsymbol{d}^{T}\boldsymbol{M}(\boldsymbol{\theta}%
_{0})\left(  \boldsymbol{M}^{T}(\boldsymbol{\theta}_{0})\boldsymbol{\Sigma
}_{\beta}(\boldsymbol{\theta}_{0})\boldsymbol{M}(\boldsymbol{\theta}%
_{0})\right)  ^{-1}\boldsymbol{M}^{T}(\boldsymbol{\theta}_{0})\boldsymbol{d}%
\right)  \label{22.23}%
\end{equation}
under $H_{1,n}$ given in (\ref{22.12}) and by
\begin{equation}
W_{n}({\boldsymbol{\widehat{\boldsymbol{\theta}}}}_{\beta}%
)\underset{n\rightarrow\infty}{\overset{\mathcal{L}}{\longrightarrow}}\chi
_{r}^{2}\left(  \boldsymbol{\delta}^{T}\left(  \boldsymbol{M}^{T}%
(\boldsymbol{\theta}_{0})\boldsymbol{\Sigma}_{\beta}(\boldsymbol{\theta}%
_{0})\boldsymbol{M}(\boldsymbol{\theta}_{0})\right)  ^{-1}\boldsymbol{\delta
}\right)  \label{22.24}%
\end{equation}
under $H_{1,n}^{\ast}$ given in (\ref{22.19}). These asymptotic distributions
may be used to calculate the power functions of the Wald-type test statistics
under the contiguous alternatives.

\section{Influence functions of the Wald-type test statistics\label{sec3}}

The influence function was introduced by
\citet{Hampel74}
and it plays a crucial role for important applications in robustness
analysis.
\citet{huber1981}
interpreted the influence function as the limiting influence of an
infinitesimal observation on the value of an estimator or a statistic that
characterizes a distribution in a large sample. If the influence function is
bounded, the corresponding estimator or the statistic is said to have
infinitesimal robustness. Therefore, the influence function particularly can
be used to quantify infinitesimal robustness of an estimator or a statistic by
measuring the approximate impact on an additional observation to the
underlying data. More simply, the influence function $\mathcal{IF}\left(
\boldsymbol{x},\boldsymbol{T}_{\beta},F_{\boldsymbol{\theta}_{0}}\right)  $ is
the first derivative of an estimator or statistic viewed as a functional
$\boldsymbol{T}_{\beta}$ and it describes the normalized influence on the
estimate or statistic of an infinitesimal observation $\boldsymbol{x}$.

In this Section we study the influence function of the Wald-type test
statistics defined in (\ref{22.11}) and (\ref{22.17}).
In
\citet{MR1665873}
it was established that the influence function of the density power divergence
functional is
\begin{equation}
\mathcal{IF}\left(  \boldsymbol{x},\boldsymbol{T}_{\beta}%
,F_{\boldsymbol{\theta}_{0}}\right)  =\lim_{\varepsilon\rightarrow0}%
\frac{\boldsymbol{T}_{\beta}\left(  F_{\varepsilon}\right)  -\boldsymbol{T}%
_{\beta}\left(  F_{\boldsymbol{\theta}_{0}}\right)  }{\varepsilon
}=\boldsymbol{J}_{\beta}^{-1}(\boldsymbol{\theta}_{0})\left(  \boldsymbol{u}%
_{\boldsymbol{\theta}}\left(  \boldsymbol{x}\right)  f_{\boldsymbol{\theta
}_{0}}^{\beta}(\boldsymbol{x})-\boldsymbol{\xi}\left(  \boldsymbol{\theta}%
_{0}\right)  \right)  , \label{IF}%
\end{equation}
where $F_{\varepsilon}=(1-\varepsilon)F_{\boldsymbol{\theta}_{0}}%
+\varepsilon\Delta_{\boldsymbol{x}}$ is the $\varepsilon$-contaminated
distribution of $F_{\boldsymbol{\theta}_{0}}$ with respect to $\Delta
_{\boldsymbol{x}}$, the point mass distribution at $\boldsymbol{x}$. If we
assume that $\boldsymbol{J}_{\beta}($$\boldsymbol{\theta}$$_{0})$
and\ $\boldsymbol{\xi}\left(  \boldsymbol{\theta}_{0}\right)  $ are finite,
the influence function is a bounded function of $\boldsymbol{x}$ whenever
$\boldsymbol{u}_{\boldsymbol{\theta}}\left(  \boldsymbol{x}\right)
f_{\boldsymbol{\theta}_{0}}^{\beta}(\boldsymbol{x})$ is bounded. This is true,
for example in the normal location-scale problem for $\beta>0$, unlike other
density based minimum divergence procedures such as those based on the
Hellinger distance. In the case of the normal model with known variance
$\sigma^{2}$ and unknown mean $\theta_{0}$, we have
\[
\mathcal{IF}\left(  x,\boldsymbol{T}_{\beta},F_{\theta_{0}}\right)
=\frac{x-\theta_{0}}{\sigma^{\beta+2}(\sqrt{2\pi})^{\beta}}\exp\left\{
-\frac{1}{2}\left(  \frac{x-\theta_{0}}{\sigma}\right)  ^{2}\beta\right\}  .
\]
For any $\beta>0$, the above mentioned influence function is bounded, but for
$\beta=0$ it is not bounded.

Let us consider the test statistic $W_{n}^{0}%
({\boldsymbol{\widehat{\boldsymbol{\theta}}}}_{\beta})$ for testing the simple
null hypothesis given in (\ref{22.10}). The functional associated with the
test statistic $W_{n}^{0}({\boldsymbol{\widehat{\boldsymbol{\theta}}}}_{\beta
})$, evaluated at $G$, is given by (ignoring the multiplier $n$)%
\begin{equation}
W_{\beta}^{0}(G)=(\boldsymbol{T}_{\beta}(G)-\boldsymbol{\theta}_{0}%
)^{T}\boldsymbol{\Sigma}_{\beta}^{-1}(\boldsymbol{\theta}_{0})(\boldsymbol{T}%
_{\beta}(G)-\boldsymbol{\theta}_{0}). \label{WoF}%
\end{equation}
Let $G_{\varepsilon}=(1-\varepsilon)G+\varepsilon\Delta_{x}$ be the
$\varepsilon$-contaminated distribution of $G$ with respect to the point mass
distribution $\Delta_{\boldsymbol{x}}$ at $\boldsymbol{x}$. The influence
function of $W_{\beta}^{0}(\cdot)$ is defined as%
\[
\mathcal{IF}(\boldsymbol{x},W_{\beta}^{0},G)=\left.  \frac{\partial W_{\beta
}^{0}(G_{\varepsilon})}{\partial\varepsilon}\right\vert _{\varepsilon=0},
\]
where%
\[
\left.  \frac{\partial W_{\beta}^{0}(G_{\varepsilon})}{\partial\varepsilon
}\right\vert _{\varepsilon=0}=2(\boldsymbol{T}_{\beta}(G)-\boldsymbol{\theta
}_{0})^{T}\boldsymbol{\Sigma}_{\beta}^{-1}(\boldsymbol{\theta}_{0}%
)\mathcal{IF}(\boldsymbol{x},\boldsymbol{T}_{\beta},G).
\]
Under the simple null hypothesis given in (\ref{22.10}),
$G=F_{\boldsymbol{\theta}_{0}}$ and $\boldsymbol{T}_{\beta}(G)=$
$\boldsymbol{\theta}_{0}$. So $\mathcal{IF}(\boldsymbol{x},W_{\beta}%
^{0},F_{\boldsymbol{\theta}_{0}})=0$, which shows that the influence function
analysis based on the first derivative of $W_{\beta}^{0}(G_{\varepsilon})$ is
not adequate to quantify the robustness of these estimators. This influence
function is bounded in $x$ for all $\beta\geq0$, but it does not imply that
the test is necessarily robust since we know the non-robust nature of the
usual MLE based Wald-test at $\beta=0$. So other type of analysis should be applied.

The functional associated with the test statistic $W_{n}%
({\boldsymbol{\widehat{\boldsymbol{\theta}}}}_{\beta})$, given in
(\ref{22.17}), evaluated at $G$, is given by (ignoring the multiplier $n$)%
\begin{equation}
W_{\beta}(G)=\boldsymbol{m}^{T}(\boldsymbol{T}_{\beta}(G))\left(
\boldsymbol{M}^{T}(\boldsymbol{T}_{\beta}(G))\boldsymbol{\Sigma}_{\beta
}(\boldsymbol{T}_{\beta}(G))\boldsymbol{M}(\boldsymbol{T}_{\beta}(G))\right)
^{-1}\boldsymbol{m}(\boldsymbol{T}_{\beta}(G)). \label{WF}%
\end{equation}
The influence function of $W_{\beta}(\cdot)$ is defined as%
\[
\mathcal{IF}(\boldsymbol{x},W_{\beta},G)=\left.  \frac{\partial W_{\beta
}(G_{\varepsilon})}{\partial\varepsilon}\right\vert _{\varepsilon=0},
\]
where%
\[
\left.  \frac{\partial W_{\beta}(G_{\varepsilon})}{\partial\varepsilon
}\right\vert _{\varepsilon=0}=2\boldsymbol{m}^{T}(\boldsymbol{T}_{\beta
}(G))\left(  \boldsymbol{M}^{T}(\boldsymbol{T}_{\beta}(G))\boldsymbol{\Sigma
}_{\beta}(\boldsymbol{T}_{\beta}(G))\boldsymbol{M}(\boldsymbol{T}_{\beta
}(G))\right)  ^{-1}\boldsymbol{M}^{T}(\boldsymbol{T}_{\beta}(G))\mathcal{IF}%
(\boldsymbol{x},\boldsymbol{T}_{\beta},G).
\]
Let $\boldsymbol{\theta}$$_{0}\in\Theta_{0}$ be the true value of the
parameter under the composite hypothesis given in (\ref{22.16}). So
$G=F_{\boldsymbol{\theta}_{0}}$ and $\boldsymbol{m}(\boldsymbol{T}_{\beta
}(G))=\mathbf{0}_{r}$, and finally it turns out that $\mathcal{IF}%
(\boldsymbol{x},W_{\beta},G)=0$, which indicates that the derivation of second
order influence function is necessary.

The following theorem present the second order influence function for the
Wald-type test statistics $W_{n}^{0}({\boldsymbol{\widehat{\boldsymbol{\theta
}}}}_{\beta})$ and $W_{n}({\boldsymbol{\widehat{\boldsymbol{\theta}}}}_{\beta
})$.

\begin{theorem}
\label{Theorem1}The second order influence functions of the Wald-type test
statistics $W_{n}^{0}({\boldsymbol{\widehat{\boldsymbol{\theta}}}}_{\beta})$,
given in (\ref{22.11}), and $W_{n}({\boldsymbol{\widehat{\boldsymbol{\theta}%
}_{\beta}}})$, given in (\ref{22.17}), are respectively%
\begin{align}
\mathcal{IF}_{2}(\boldsymbol{x},W_{\beta}^{0},F_{\boldsymbol{\theta}_{0}})  &
=2\left(  \boldsymbol{u}_{\boldsymbol{\theta}}\left(  \boldsymbol{x}\right)
f_{\boldsymbol{\theta}_{0}}^{\beta}(\boldsymbol{x})-\boldsymbol{\xi}\left(
\boldsymbol{\theta}_{0}\right)  \right)  ^{T}\boldsymbol{J}_{\beta}%
^{-1}(\boldsymbol{\theta}_{0})\boldsymbol{\Sigma}_{\beta}^{-1}%
(\boldsymbol{\theta}_{0})\boldsymbol{J}_{\beta}^{-1}(\boldsymbol{\theta}%
_{0})\left(  \boldsymbol{u}_{\boldsymbol{\theta}}\left(  \boldsymbol{x}%
\right)  f_{\boldsymbol{\theta}_{0}}^{\beta}(\boldsymbol{x})-\boldsymbol{\xi
}\left(  \boldsymbol{\theta}_{0}\right)  \right)  ,\label{IF1}\\
\mathcal{IF}_{2}(\boldsymbol{x},W_{\beta},F_{\boldsymbol{\theta}_{0}})  &
=2\left(  \boldsymbol{u}_{\boldsymbol{\theta}}\left(  \boldsymbol{x}\right)
f_{\boldsymbol{\theta}_{0}}^{\beta}(\boldsymbol{x})-\boldsymbol{\xi}\left(
\boldsymbol{\theta}_{0}\right)  \right)  ^{T}\boldsymbol{J}_{\beta}%
^{-1}(\boldsymbol{\theta}_{0})\boldsymbol{M}(\boldsymbol{\theta}%
_{0})\boldsymbol{\Sigma}_{\beta}^{\ast-1}(\boldsymbol{\theta}_{0}%
)\boldsymbol{M}^{T}(\boldsymbol{\theta}_{0})\boldsymbol{J}_{\beta}%
^{-1}(\boldsymbol{\theta}_{0})\left(  \boldsymbol{u}_{\boldsymbol{\theta}%
}\left(  \boldsymbol{x}\right)  f_{\boldsymbol{\theta}_{0}}^{\beta
}(\boldsymbol{x})-\boldsymbol{\xi}\left(  \boldsymbol{\theta}_{0}\right)
\right)  , \label{IF2}%
\end{align}
where%
\begin{equation}
\boldsymbol{\Sigma}_{\beta}^{\ast}(\boldsymbol{\theta}_{0})=\boldsymbol{M}%
^{T}(\boldsymbol{\theta}_{0})\boldsymbol{\Sigma}_{\beta}(\boldsymbol{\theta
}_{0})\boldsymbol{M}(\boldsymbol{\theta}_{0}). \label{sigStar}%
\end{equation}

\end{theorem}

\begin{proof}
See Appendix \ref{A1}.
\end{proof}

It is interesting to note that in most of the cases $\boldsymbol{K}_{\beta}%
($$\boldsymbol{\theta}$$_{0})$, as defined in (\ref{22.9}), is a full rank
matrix and so%
\begin{equation}
\mathcal{IF}_{2}(\boldsymbol{x},W_{\beta}^{0},F_{\boldsymbol{\theta}_{0}%
})=2\left(  \boldsymbol{u}_{\boldsymbol{\theta}}\left(  \boldsymbol{x}\right)
f_{\boldsymbol{\theta}_{0}}^{\beta}(\boldsymbol{x})-\boldsymbol{\xi}\left(
\boldsymbol{\theta}_{0}\right)  \right)  ^{T}\boldsymbol{K}_{\beta}%
^{-1}(\boldsymbol{\theta}_{0})\left(  \boldsymbol{u}_{\boldsymbol{\theta}%
}\left(  \boldsymbol{x}\right)  f_{\boldsymbol{\theta}_{0}}^{\beta
}(\boldsymbol{x})-\boldsymbol{\xi}\left(  \boldsymbol{\theta}_{0}\right)
\right)  . \label{IFb}%
\end{equation}
The above theorem yields the possibility of studying the robustness of the
Wald-type tests through its non-zero (in general) second order influence functions.

In particular, for the simple hypothesis testing, the second order influence
function of the corresponding Wald-type test turns out to be bounded in
$\boldsymbol{x}$ for most parametric models if $\beta>0$; it becomes unbounded
at $\beta=0$ hence, this test is expected to be robust for most common
parametric models whenever $\beta>0$, but non-robust at $\beta=0$ (the
ordinary Wald-type test). In case of composite hypothesis also, the second
order influence functions of the general Wald-type tests with $\beta>0$ are
bounded in the contamination point $\boldsymbol{x}$ in most parametric models
implying their robustness. Some illustrative examples are provided later in
Section \ref{sec6}.

\section{Level and Power Influence Functions\label{Sec4}}

In this section, we investigate the local stability of the Wald-type test
statistic by means of the influence function when the simple null hypothesis
is considered. For a finite sample size, in general, it is difficult to
calculate the level and power, and therefore, we shall use asymptotic
approximations. At a fixed alternative the power function of the Wald-type
test statistic was given in equation (\ref{22.10.1}). This power function
tends to one as $n$ increases, so the test is consistent in the Fraser's
sense. Therefore, it is important to calculate power functions at the
contiguous alternatives as mentioned in (\ref{22.12}). In this case the
asymptotic power function can be approximated using (\ref{22.13.1}).

Now we shall consider the sequence of alternatives $\boldsymbol{\theta}%
_{n}=\boldsymbol{\theta}_{0}+n^{-1/2}\boldsymbol{d}$ as given in
(\ref{22.12}). When $\boldsymbol{\theta}$$_{n}$ tends to $\boldsymbol{\theta}%
$$_{0}$ the contamination proportion is also assumed to tend to zero at the
same rate. Therefore, we shall define the contaminated distributions for the
power as
\begin{equation}
F_{n,\varepsilon,\boldsymbol{x}}^{P}=(1-\tfrac{\varepsilon}{\sqrt{n}%
})F_{\boldsymbol{\theta}_{n}}+\tfrac{\varepsilon}{\sqrt{n}}\Delta
_{\boldsymbol{x}}, \label{power_F}%
\end{equation}
where $\Delta_{\boldsymbol{x}}$ denotes the degenerate distribution function
with all its mass concentrated at point $\boldsymbol{x}$, and $\varepsilon
/\sqrt{n}$ is the contamination proportion. Substituting $\boldsymbol{d}%
=\boldsymbol{0}_{p}$ in equation (\ref{power_F}) we get the contaminated
distributions for the level as
\[
F_{n,\varepsilon,\boldsymbol{x}}^{L}=(1-\tfrac{\varepsilon}{\sqrt{n}%
})F_{\boldsymbol{\theta}_{0}}+\tfrac{\varepsilon}{\sqrt{n}}\Delta
_{\boldsymbol{x}}.
\]
Let us consider the following notations
\[
\alpha_{W_{n}^{0}}(\varepsilon,\boldsymbol{x})=\lim\limits_{n\rightarrow
\infty}P_{F_{n,\varepsilon,\boldsymbol{x}}^{L}}(W_{n}^{0}%
(\widehat{\boldsymbol{\theta}}_{\beta})>\chi_{p,\alpha}^{2})\text{,}%
\;\alpha_{W_{n}}(\varepsilon,\boldsymbol{x})=\lim\limits_{n\rightarrow\infty
}P_{F_{n,\varepsilon,y}^{L},\varepsilon,x}(W_{n}(\widehat{\boldsymbol{\theta}%
}_{\beta})>\chi_{r,\alpha}^{2})
\]
and
\[
\beta_{W_{n}^{0}}(\boldsymbol{\theta}_{n},\varepsilon,\boldsymbol{x}%
)=\lim\limits_{n\rightarrow\infty}P_{F_{n,\varepsilon,\boldsymbol{x}}^{P}%
}(W_{n}^{0}(\widehat{\boldsymbol{\theta}}_{\beta})>\chi_{p,\alpha}%
^{2})\text{,}\;\beta_{W_{n}}(\boldsymbol{\theta}_{n},\varepsilon
,\boldsymbol{x})=\lim\limits_{n\rightarrow\infty}P_{F_{n,\varepsilon
,\boldsymbol{x}}^{P}}(W_{n}(\widehat{\boldsymbol{\theta}}_{\beta}%
)>\chi_{r,\alpha}^{2}).
\]
Using these quantities, we will now define the level and power influence
function for our proposed Wald-type test statistics.

\begin{definition}
The level influence functions associated with the Wald-type test statistics
for simple and composite null hypotheses are defined as
\[
\mathcal{LIF}(\boldsymbol{x};W_{\beta}^{0},F_{\boldsymbol{\theta}_{0}%
})=\left.  \dfrac{\partial}{\partial\varepsilon}\alpha_{W_{n}^{0}}%
(\varepsilon,\boldsymbol{x})\right\vert _{\varepsilon=0},\;\mathcal{LIF}%
(\mathbf{x};W_{\beta},F_{\boldsymbol{\theta}_{0}})=\left.  \dfrac{\partial
}{\partial\varepsilon}\alpha_{W_{n}}(\varepsilon,\boldsymbol{x})\right\vert
_{\varepsilon=0}.
\]
Similarly, we define the power influence functions as%
\[
\mathcal{PIF}(\mathbf{x};W_{\beta}^{0},F_{\boldsymbol{\theta}_{0}})=\left.
\dfrac{\partial}{\partial\varepsilon}\beta_{W_{n}^{0}}(\boldsymbol{\theta}%
_{n},\varepsilon,\boldsymbol{x})\right\vert _{\varepsilon=0},\;\mathcal{PIF}%
(\mathbf{x};W_{\beta},F_{\boldsymbol{\theta}_{0}})=\left.  \dfrac{\partial
}{\partial\varepsilon}\beta_{W_{n}}(\boldsymbol{\theta}_{n},\varepsilon
,\boldsymbol{x})\right\vert _{\varepsilon=0}.
\]

\end{definition}

The level and power influence functions indicate the limiting change in the
asymptotic level and power of the test respectively under the sequence of
corresponding contaminated distributions with infinitesimal contamination at
the limit. In simple term, they indicate how the asymptotic level and power of
the test change due to the contamination in data generating distributions.
Boundedness of these level and power influence functions imply the stability
of the level and power of the test respectively. For more details see
\citet[Section 3.2c]{Hampel86}%
.

The above definitions of the level and power influence functions are
completely general one and have no direct relation with the influence function
of the corresponding test statistics. However, in case of our Wald-type test
statistics, we have seen that the second order influence functions of the test
statistics at the null hypothesis are quadratic function of the influence
function of the parameters estimators used in constructing the test. Further,
we will see below that the level and power influence functions are also linear
function of the influence function of the corresponding estimators. In that
way, there is a indirect link of the level and power influence function with
the influence function of the test statistics (as derived in Section
\ref{sec3}). In particular, for any given testing problem, boundedness of one
would imply the same for others provided these influence functions are not
identically zero. However, it is also important to study these level and power
influence functions for all the testing problems to examine the extent of
robustness with respect to their level and power, which we cannot get only
studying the influence function of the test statistics alone.

\subsection{Simple null hypothesis\label{Sec4 .1}}

In the rest of the paper, we will frequently use the standard assumptions of
asymptotic inference as given by Assumptions A, B, C and D of
\citet[page 429]{MR702834}%
. We will refer to them as the Lehmann conditions. Some of the proofs will
also require the conditions D1--D5 of
\citet[page 311]{AyanBook}
which we will refer to as Basu et al. conditions. In order to avoid arresting
the flow of the paper, these conditions have been presented in the Appendix.

\begin{theorem}
\label{THM:7asymp_power_one}Assume that the Lehmann and Basu et al. conditions
hold for the model. Let us consider the contiguous alternatives in
(\ref{22.12}) against the simple null hypothesis, and the underlying
contaminated model as given in (\ref{power_F}). Then we have the following:

\begin{enumerate}
\item The asymptotic distribution of the test statistics $W_{n}^{0}%
(\widehat{\boldsymbol{\theta}}_{\beta})$ under $F_{n,\varepsilon
,\boldsymbol{x}}^{P}$ is non-central chi-square with $p$ degrees of freedom
and the non-centrality parameter
\[
\delta= \widetilde{\boldsymbol{d}}_{\varepsilon,\boldsymbol{x},\beta}%
^{T}(\boldsymbol{\theta}_{0}) \boldsymbol{\Sigma}_{\beta}^{-1}%
(\boldsymbol{\theta}_{0}) \widetilde{\boldsymbol{d}}_{\varepsilon
,\boldsymbol{x},\beta}(\boldsymbol{\theta}_{0}),
\]
where $\widetilde{\boldsymbol{d}}_{\varepsilon,\boldsymbol{x},\beta}
($$\boldsymbol{\theta}$$_{0})=\boldsymbol{d}+\varepsilon\mathcal{IF}
(\boldsymbol{x},\boldsymbol{T}_{\beta},F_{\boldsymbol{\theta}_{0}})$ and
$\mathcal{IF}(\boldsymbol{x},\boldsymbol{T}_{\beta},F_{\boldsymbol{\theta}%
_{0}})$ is given by (\ref{IF}).

\item The asymptotic power under contiguous alternative and contiguous
contamination can be approximated as
\begin{align}
\beta_{W_{n}^{0}}(\boldsymbol{\theta}_{n},\varepsilon,\boldsymbol{x})  &
\cong1-F_{\chi_{p}^{2}(\delta)}(\chi_{p,\alpha}^{2})\nonumber\\
&  \cong\sum\limits_{v=0}^{\infty}C_{v}\left(  \widetilde{\boldsymbol{d}%
}_{\varepsilon,\boldsymbol{x},\beta}(\boldsymbol{\theta}_{0}%
),\boldsymbol{\Sigma}_{\beta}^{-1}(\boldsymbol{\theta}_{0})\right)  P\left(
\chi_{p+2v}^{2}>\chi_{p,\alpha}^{2}\right)  ,
\end{align}
where
\[
C_{v}\left(  \boldsymbol{t},\boldsymbol{A}\right)  =\frac{\left(
\mathbf{t}^{T}\boldsymbol{A}\mathbf{t}\right)  ^{v}}{v!2^{v}}e^{-\frac{1}%
{2}\mathbf{t}^{T}\boldsymbol{A}\mathbf{t}},
\]
$F_{\chi_{p}^{2}(\delta)}$ is the distribution function of a $\chi_{p}%
^{2}(\delta)$ random variable having degrees of freedom $p$ and non-centrality
parameter $\delta$ and $\chi_{q}^{2}$ denotes a central chi-square random
variable with $q$ degrees of freedom.
\end{enumerate}

\begin{proof}
See Appendix \ref{A2}.
\end{proof}
\end{theorem}

Further, substituting $\boldsymbol{d}=\boldsymbol{0}_{p}$ or $\varepsilon=0$
in above theorem, we shall get several important cases; these are presented in
the following corollaries.

\begin{corollary}
\label{COR:7cont_power_one} Putting $\varepsilon=0$ in the above theorem, we
get the asymptotic power under the contiguous alternative hypotheses
(\ref{22.12}) as
\begin{equation}
\beta_{W_{n}^{0}}(\boldsymbol{\theta}_{n})=\beta_{W_{n}^{0}}%
(\boldsymbol{\theta}_{n},0,\boldsymbol{x})\cong\sum\limits_{v=0}^{\infty}%
C_{v}\left(  \boldsymbol{d},\boldsymbol{\Sigma}_{\beta}^{-1}%
(\boldsymbol{\theta}_{0})\right)  P\left(  \chi_{p+2v}^{2}>\chi_{p,\alpha}%
^{2}\right)  .\nonumber
\end{equation}

\end{corollary}

\noindent Notice that Corollary \ref{COR:7cont_power_one} is an alternative
approximation for the result given in (\ref{22.13.1}).

\begin{corollary}
\label{COR:7asymp_level_one}Putting $\boldsymbol{d}=\boldsymbol{0}_{p}$ in the
above theorem, we get the asymptotic distribution of $W_{n}^{0}%
(\widehat{\boldsymbol{\theta}}_{\beta})$ under the probability distribution
$F_{n,\varepsilon,\boldsymbol{x}}^{L}$ as the non-central chi-square
distribution with degrees of freedom $p$ and non-centrality parameter
$\varepsilon\mathcal{IF}(\boldsymbol{x};T_{\beta},F_{\boldsymbol{\theta}_{0}%
}).$Then, the corresponding asymptotic level is given by
\begin{align}
\alpha_{W_{n}^{0}}(\varepsilon,\boldsymbol{x})  &  =\beta_{W_{n}^{0}%
}(\boldsymbol{\theta}_{0},\varepsilon,\boldsymbol{x})\nonumber\\
&  \cong\sum\limits_{v=0}^{\infty}C_{v}\left(  \varepsilon\mathcal{IF}%
(\boldsymbol{x};T_{\beta},F_{\boldsymbol{\theta}_{0}}),\boldsymbol{\Sigma
}_{\beta}^{-1}(\boldsymbol{\theta}_{0})\right)  P\left(  \chi_{p+2v}^{2}%
>\chi_{p,\alpha}^{2}\right)  .\nonumber
\end{align}
In particular, as $\varepsilon\rightarrow0$, $\boldsymbol{\theta}_{n}^{\ast
}\rightarrow\boldsymbol{\theta}_{0}$ and the non-centrality parameter of the
above asymptotic distribution tends to zero. In this way we get the asymptotic
distribution of the test statistics under null, the central chi-square
distribution with $p$ degrees of freedom, which is the same as obtained
independently by \cite{MR3011625}.
\end{corollary}

\noindent This was the expected result according to the construction of the
test statistic and its critical value. Next we derive the power influence
function of the Wald-type test statistic.

\begin{theorem}
\label{Theorem10}Assume that the Lehmann and Basu et al. conditions hold for
the model. Then, the power influence function of the Wald-type test statistic
under the simple null hypothesis is given by
\begin{equation}
\mathcal{PIF}(\boldsymbol{x},W_{\beta}^{0},F_{\boldsymbol{\theta}_{0}})\cong
K_{p}^{\ast}\left(  \mathbf{d}^{T}\boldsymbol{\Sigma}_{\beta}^{-1}%
(\boldsymbol{\theta}_{0})\mathbf{d}\right)  \boldsymbol{d}^{T}%
\boldsymbol{\Sigma}_{\beta}^{-1}(\boldsymbol{\theta}_{0})\mathcal{IF}%
(\boldsymbol{x},\boldsymbol{T}_{\beta},F_{\boldsymbol{\theta}_{0}}),
\label{EQ:7PIF_simpleTest1}%
\end{equation}
where
\[
K_{p}^{\ast}(s)=e^{-\frac{s}{2}}\sum\limits_{v=0}^{\infty}\frac{s^{v-1}%
}{v!2^{v}}\left(  2v-s\right)  P\left(  \chi_{p+2v}^{2}>\chi_{p,\alpha}%
^{2}\right)  .
\]

\end{theorem}


\begin{proof}
See Appendix \ref{A3}.
\end{proof}

\label{Rem copy(1)} Clearly the above theorem shows that the power influence
function is bounded whenever the influence function of the MDPDE is bounded.

To calculate the level influence function, we can again start from the
expression of $\alpha_{W_{n}^{0}}(\varepsilon,\boldsymbol{x})$ as given in
Corollary \ref{COR:7asymp_level_one} and proceed as above. Alternatively, we
may also substitute $\boldsymbol{d}=\boldsymbol{0}_{p}$ in the expression of
the power influence function to get the level influence function as
\[
\mathcal{LIF}(\boldsymbol{x},W_{\beta}^{0},F_{\boldsymbol{\theta}_{0}})=0.
\]
Also, one can conclude that the derivative of $\alpha_{W_{n}^{0}}%
(\varepsilon,\boldsymbol{x})$ of any order will be zero at $\varepsilon=0$,
implying that the level influence function of any order will be zero. Thus,
asymptotically, the level of the Wald-type test statistic will be unaffected
by a contiguous contamination.

\subsection{Composite null Hypothesis\label{Sec4.2}}

We shall now calculate the level and power influence functions of the
Wald-type test statistic for the composite null hypothesis. We have considered
the same setting as mentioned in Section \ref{Sec4}.

\begin{theorem}
\label{THM:7asymp_power_composite}Assume that the Lehmann and Basu et al.
conditions hold for the model. Let us consider the contiguous alternatives in
(\ref{22.12}) against the composite null hypothesis, and the underlying
contaminated model as given in (\ref{power_F}). Then we have the following:

\begin{enumerate}
\item The asymptotic distribution of the test statistics $W_{n}
(\widehat{\boldsymbol{\theta}}_{\beta})$ under $F_{n,\varepsilon
,\boldsymbol{x}}^{P}$ is non-central chi-square with $r$ degrees of freedom
and the non-centrality parameter
\[
\delta= \widetilde{\boldsymbol{d}}_{\varepsilon,\boldsymbol{x},\beta}
^{T}(\boldsymbol{\theta}_{0}) \boldsymbol{M}(\boldsymbol{\theta}
_{0})\boldsymbol{\Sigma}_{\beta}^{\ast-1}(\boldsymbol{\theta}_{0})
\boldsymbol{M}^{T}(\boldsymbol{\theta}_{0}) \widetilde{\boldsymbol{d}
}_{\varepsilon,\boldsymbol{x},\beta}(\boldsymbol{\theta}_{0}),
\]
where $\boldsymbol{\Sigma}_{\beta}^{\ast}(\boldsymbol{\theta}_{0})
=\boldsymbol{M}^{T}(\boldsymbol{\theta}_{0}) \boldsymbol{\Sigma}_{\beta
}(\boldsymbol{\theta}_{0}) \boldsymbol{M}(\boldsymbol{\theta}_{0})$,
$\widetilde{\boldsymbol{d}}_{\varepsilon,\boldsymbol{x},\beta} ($
$\boldsymbol{\theta}$$_{0})=\boldsymbol{d}+\varepsilon\mathcal{IF}
(\boldsymbol{x},\boldsymbol{T}_{\beta},F_{\boldsymbol{\theta}_{0}})$ and
$\mathcal{IF}(\boldsymbol{x},\boldsymbol{T}_{\beta},F_{\boldsymbol{\theta}
_{0}})$ is given by (\ref{IF}).

\item The asymptotic power under contiguous alternative and contiguous
contamination can be approximated as
\begin{align}
\beta_{W_{n}^{0}}(\boldsymbol{\theta}_{n},\varepsilon,\boldsymbol{x})  &
\cong1-F_{\chi_{r}^{2}(\delta)}(\chi_{r,\alpha}^{2})\nonumber\\
&  \cong\sum\limits_{v=0}^{\infty}C_{v}\left(  \boldsymbol{M}^{T}%
(\boldsymbol{\theta}_{0})\widetilde{\boldsymbol{d}}_{\varepsilon
,\boldsymbol{x},\beta}(\boldsymbol{\theta}_{0}),\boldsymbol{\Sigma}_{\beta
}^{\ast-1}\right)  P\left(  \chi_{r+2v}^{2}>\chi_{r,\alpha}^{2}\right)  ,
\end{align}
where $C_{v}\left(  \boldsymbol{t},\boldsymbol{A}\right)  $ is as defined in
Theorem \ref{THM:PIF_composite}, $F_{\chi_{r}^{2}(\delta)}$ is the
distribution function of a $\chi_{r}^{2}(\delta)$ random variable having
degrees of freedom $r$ and non-centrality parameter $\delta$ and $\chi_{q}%
^{2}$ denotes a central chi-square random variable with $q$ degrees of freedom.
\end{enumerate}
\end{theorem}

\begin{proof}
See Appendix \ref{A4}.
\end{proof}

Putting $\varepsilon=0$ in the above theorem, we get the asymptotic power
under the contiguous alternatives as
\begin{equation}
\beta_{W_{n}}(\boldsymbol{\theta}_{n})=\beta_{W_{n}}(\boldsymbol{\theta}
_{n},0,\boldsymbol{x})\cong\sum\limits_{v=0}^{\infty}C_{v}\left(
\boldsymbol{M}^{T}(\boldsymbol{\theta}_{0}) \boldsymbol{d}, \boldsymbol{\Sigma
}_{\beta}^{\ast-1}(\boldsymbol{\theta}_{0})\right)  P\left(  \chi_{r+2v}^{2} >
\chi_{r,\alpha}^{2}\right)  .\nonumber
\end{equation}

\noindent Notice that this result is an alternative approximation of the power
function given in (\ref{22.17.1}).

Putting $\boldsymbol{d}$$=\boldsymbol{0}_{p}$ in the above theorem, we get the
asymptotic level under the probability distribution $F_{n,\varepsilon
,\boldsymbol{x}}^{L}$ as
\begin{align}
\alpha_{W_{n}}(\varepsilon,\boldsymbol{x})  &  =\beta_{W_{n}}
(\boldsymbol{\theta}_{0},\varepsilon,\boldsymbol{x}) \cong\sum\limits_{v=0}
^{\infty}C_{v}\left(  \varepsilon\boldsymbol{M}^{T}(\boldsymbol{\theta}_{0})
\mathcal{IF}(\boldsymbol{x},T_{\beta},F_{\boldsymbol{\theta}_{0}}),
\boldsymbol{\Sigma}_{\beta}^{\ast-1}(\boldsymbol{\theta}_{0})\right)  P\left(
\chi_{r+2v}^{2} > \chi_{r,\alpha}^{2}\right)  .\nonumber
\end{align}
In particular, taking $\varepsilon\rightarrow0$ in the above expression, we
get the asymptotic level of the test statistics as $\alpha_{W_{n}}(0,
\boldsymbol{x}) =\alpha. $

\noindent This was the expected result according to the construction of the
test statistic and its critical value. Next we derive the power influence
function of the proposed test statistic.

\begin{theorem}
\label{THM:PIF_composite}Assume that the Lehmann and Basu et al. conditions
hold for the model. Then, the power influence function of the proposed
Wald-type test statistic under the composite null hypothesis is given by
\[
\mathcal{PIF}(\boldsymbol{x},W_{\beta},F_{\boldsymbol{\theta}_{0}})\cong
K_{r}^{\ast}\left(  \boldsymbol{d}^{T}\boldsymbol{M}(\boldsymbol{\theta}%
_{0})\boldsymbol{\Sigma}_{\beta}^{\ast-1}(\boldsymbol{\theta}_{0}%
)\boldsymbol{M}^{T}(\boldsymbol{\theta}_{0})\boldsymbol{d}\right)
\boldsymbol{d}^{T}\boldsymbol{M}(\boldsymbol{\theta}_{0})\boldsymbol{\Sigma
}_{\beta}^{\ast-1}(\boldsymbol{\theta}_{0})\boldsymbol{M}^{T}%
(\boldsymbol{\theta}_{0})\mathcal{IF}(\boldsymbol{x},\boldsymbol{T}_{\beta
},F_{\boldsymbol{\theta}_{0}}),
\]
where the constant $K_{r}^{\ast}(s)$ is as defined in Theorem
\ref{THM:PIF_composite}.
\end{theorem}

\begin{proof}
See Appendix \ref{A5}.
\end{proof}

It is clear from the above expression that the power influence function of the
Wald-type test statistic under the composite null hypothesis is also bounded
whenever the influence function of the MDPDE is bounded.

To calculate the level influence function, we can start from the expression of
$\alpha_{W_{n}}(\varepsilon,\boldsymbol{x})$ as above. From this or
alternatively, by simply substituting $\mathbf{d}=0$ in the expression of the
power influence function, we obtain that%
\[
L\mathcal{IF}(\boldsymbol{x},W_{\beta},F_{\boldsymbol{\theta}_{0}}%
)=\frac{\partial}{\partial\varepsilon}\left.  \alpha_{W_{n}}(\varepsilon
,\boldsymbol{x})\right\vert _{\varepsilon=0}=0.
\]
Also, it is easy to see that the derivative of $\alpha_{W_{n}}(\varepsilon
,\boldsymbol{x})$ of any order will be zero at $\varepsilon=0$, implying that
the level influence function of any order will be zero. Thus, asymptotically,
the level of the proposed test statistics will be unaffected by a contiguous contamination.

\section{The Chi-Square Inflation Factor\label{sec5}}

Another important way of measuring the robustness of a test statistic is to
look at its asymptotic distribution for a general contaminated distribution,
in contrast to its null distribution under the model. Unlike the contiguous
contamination considered in the previous section, we shall now consider a
fixed departure from the model independent of the sample size. Under the
set-up of the previous sections, let us assume that the data come from a
general contaminated distribution $G$ having density $g$. The null hypothesis,
mentioned in (\ref{22.10}), can be written as%
\begin{equation}
H_{0}:\boldsymbol{T}_{\beta}(G)=\boldsymbol{\theta}_{0}.
\label{EQ:7null_hyp_func}%
\end{equation}
The asymptotic distribution of MDPDE under the model is given in
(\ref{asymp_dist}). We shall now derive the asymptotic null distribution of
the Wald-type test statistic under a general distribution $G$. Let us define%
\begin{equation}
\boldsymbol{J}_{\beta,g}(\boldsymbol{\theta})=\int\boldsymbol{u}%
_{{\boldsymbol{\theta}}}(\boldsymbol{x})\boldsymbol{u}_{{\boldsymbol{\theta}}%
}^{T}(\boldsymbol{x})f_{{\boldsymbol{\theta}}}^{1+\beta}(\boldsymbol{x}%
)d\boldsymbol{x}+\int\left(  \boldsymbol{I}_{\boldsymbol{\theta}%
}(\boldsymbol{x})-\beta\boldsymbol{u}_{{\boldsymbol{\theta}}}(\boldsymbol{x}%
)\boldsymbol{u}_{{\boldsymbol{\theta}}}^{T}(\boldsymbol{x})\right)  \left(
g(\boldsymbol{x})-f_{\boldsymbol{\theta}}(\boldsymbol{x})\right)
f_{{\boldsymbol{\theta}}}^{\beta}(\boldsymbol{x})d\boldsymbol{x},
\label{EQ:J-matrix}%
\end{equation}
and%
\begin{equation}
\boldsymbol{K}_{\beta,g}(\boldsymbol{\theta})=\int\boldsymbol{u}%
_{{\boldsymbol{\theta}}}(\boldsymbol{x})\boldsymbol{u}_{{\boldsymbol{\theta}}%
}^{T}(\boldsymbol{x})f_{{\boldsymbol{\theta}}}^{2\beta}(\boldsymbol{x}%
)g(\boldsymbol{x})d\boldsymbol{x}-\boldsymbol{\xi}^{g}({{\boldsymbol{\theta}}%
})\boldsymbol{\xi}^{gT}({{\boldsymbol{\theta}}}), \label{EQ:K-matrix}%
\end{equation}
where $\boldsymbol{\xi}_{\beta,g}({{\boldsymbol{\theta}}})=\int\boldsymbol{u}%
_{{\boldsymbol{\theta}}}(\boldsymbol{x})f_{{\boldsymbol{\theta}}}^{\beta
}(\boldsymbol{x})g(\boldsymbol{x})d\boldsymbol{x}$, and $\boldsymbol{I}%
_{\boldsymbol{\theta}}(\boldsymbol{x})=-\frac{\partial}{\partial
\boldsymbol{\theta}}\boldsymbol{u}_{\boldsymbol{\theta}}^{T}(\boldsymbol{x})$,
the so called information matrix of the model. Let
${\boldsymbol{\widehat{\boldsymbol{\theta}}}}_{\beta,g}=\boldsymbol{T}_{\beta
}(G_{n})$ be the MDPDE with tuning parameter $\beta$.\
\citet{MR1665873}
and
\citet{AyanBook}
established that
\begin{equation}
n^{1/2}({\boldsymbol{\widehat{\boldsymbol{\theta}}}}_{\beta,g}%
-\boldsymbol{\theta}_{0})\underset{n\rightarrow\infty}{\overset{\mathcal{L}%
}{\longrightarrow}}\mathcal{N}(\boldsymbol{0}_{p},\boldsymbol{\Sigma}%
_{\beta,g}(\boldsymbol{\theta}_{0})), \label{dist2}%
\end{equation}
where
\begin{equation}
\boldsymbol{\Sigma}_{\beta,g}(\boldsymbol{\theta}_{0})=\boldsymbol{J}%
_{\beta,g}^{-1}(\boldsymbol{\theta}_{0})\boldsymbol{K}_{\beta,g}%
(\boldsymbol{\theta}_{0})\boldsymbol{J}_{\beta,g}^{-1}(\boldsymbol{\theta}%
_{0}). \label{sigmaG}%
\end{equation}

In Section \ref{sec:simple_null} we presented the asymptotic distribution of
the Wald-type test statistic under the simple null hypothesis when
$G=F_{\boldsymbol{\theta}_{0}}$. Our next theorem will show the asymptotic
null distribution of the Wald-type test under the general set-up when the
underlying density may or may not belong to the model.

\begin{theorem}
\label{THM:7Chi_infl_fact} Let $\widehat{\boldsymbol{\theta}}_{\beta
,g}=\boldsymbol{T}_{\beta}(G_{n})$ be the MDPDE with tuning parameter $\beta$.
Then under the null hypothesis (\ref{EQ:7null_hyp_func}), the asymptotic
distribution of the Wald-type test statistic is given by
\begin{equation}
W_{n}^{0}({\boldsymbol{\widehat{\boldsymbol{\theta}}}}_{\beta,g}%
)=n({\boldsymbol{\widehat{\boldsymbol{\theta}}}}_{\beta,g}-\boldsymbol{\theta
}_{0})^{T}\boldsymbol{\Sigma}_{\beta}^{-1}(\boldsymbol{\theta}_{0}%
)({\boldsymbol{\widehat{\boldsymbol{\theta}}}}_{\beta,g}-\boldsymbol{\theta
}_{0})\underset{n\rightarrow\infty}{\overset{\mathcal{L}}{\longrightarrow}%
}\sum_{i=1}^{p}c_{i,\beta,g}(\boldsymbol{\theta}_{0})Z_{i}^{2},
\label{EQ:7Chi_infl_fact}%
\end{equation}
where $\{Z_{i}\}_{i=1}^{p}$ are i.i.d. standard normal random variables and
$\{c_{i,\beta,g}(\boldsymbol{\theta}_{0})\}_{i=1}^{p}$ the set of eigenvalues
of\linebreak$\boldsymbol{\Sigma}_{\beta}^{-1}(\boldsymbol{\theta}%
_{0})\boldsymbol{\Sigma}_{\beta,g}(\boldsymbol{\theta}_{0})$.
\end{theorem}

\begin{proof}
The result follows from (\ref{dist2}), using Corollary 2.2 of
\citet{DikGunst}%
.
\end{proof}

The above theorem shows that the asymptotic distribution of the Wald-type test
statistic, under null hypotheis with contamination, is a linear combination of
independent $\chi_{1}^{2}$ random variables. On the other hand, if the assumed
model is correct, the asymptotic null distribution turns out to be $\chi
_{p}^{2}$. In this context, by following
\citet{Satter}%
, our proposal consists of using $\bar{c}_{\beta,g}(\boldsymbol{\theta}%
_{0})\chi_{p}^{2}$, with%
\begin{equation}
\bar{c}_{\beta,g}(\boldsymbol{\theta}_{0})=\frac{1}{p}\sum_{i=1}^{p}%
c_{i,\beta,g}(\boldsymbol{\theta}_{0})=\frac{1}{p}\mathrm{trace}\left(
\boldsymbol{\Sigma}_{\beta}^{-1}(\boldsymbol{\theta}_{0})\boldsymbol{\Sigma
}_{\beta,g}(\boldsymbol{\theta}_{0})\right)  , \label{inflation}%
\end{equation}
to approximate $\sum_{i=1}^{p}c_{i,\beta,g}(\boldsymbol{\theta}_{0})Z_{i}^{2}%
$. This factor is called Chi-Square Inflation Factor (CSIF) and its value is
equal to unity if only if $\boldsymbol{\Sigma}_{\beta,g}(\boldsymbol{\theta
}_{0})=\boldsymbol{\Sigma}_{\beta}(\boldsymbol{\theta}_{0})$. Since a value
close to unity indicates strong robustness towards the model assumption of the
Wald-type test statistic, $\bar{c}_{\beta,g}(\boldsymbol{\theta}_{0})$ is
useful as a measure of robustness.
\citet{2014arXiv1404.5126G}
used this approach to illustrate the stability of the tests based on the
$S$-divergence when $p=1$. When $p=1$ the CSIF becomes
\[
\bar{c}_{\beta,g}(\boldsymbol{\theta}_{0})=c_{1,\beta,g}(\boldsymbol{\theta
}_{0})=\frac{J_{\beta}^{2}(\theta_{0})}{K_{\beta}(\theta_{0})}\frac
{K_{\beta,g}(\theta_{0})}{J_{\beta,g}^{2}(\theta_{0})}.
\]
In this case, the asymptotic null distribution of the Wald-type test statistic
is exactly (not approximately) $\bar{c}_{\beta,g}(\boldsymbol{\theta}_{0}%
)\chi_{1}^{2}$.

We shall now illustrate the effect of outliers in CSIF. Let us consider the
following fixed point contaminated density
\[
f_{\varepsilon,\boldsymbol{y}}(\cdot)=(1-\varepsilon)f_{\boldsymbol{\theta
}_{0}}(\cdot)+\varepsilon\Delta_{\boldsymbol{y}},
\]
where $\varepsilon\in(0,1)$ is the contamination proportion, and
$\boldsymbol{y}$ is the outlying point. Let us denote $\bar{c}_{\beta
,\varepsilon,\boldsymbol{y}}(\boldsymbol{\theta}_{0})$, in the place of
$\bar{c}_{\beta,g}(\boldsymbol{\theta}_{0})$ with $g=f_{\varepsilon
,\boldsymbol{y}}$. Note that, the rate of change in $\bar{c}_{\beta
,\varepsilon,\boldsymbol{y}}(\boldsymbol{\theta}_{0})$ with respect to
$\varepsilon$ at the origin gives us the effect of infinitesimal contamination
on the test statistic. Similar interpretation as the influence function
analysis may be drawn in this case; and the boundedness of the above mentioned
quantity will indicate robustness towards the assumed model. So $\frac
{\partial}{\partial\varepsilon}\left.  \bar{c}_{\beta,\varepsilon
,\boldsymbol{y}}(\boldsymbol{\theta}_{0})\right\vert _{\varepsilon=0}$ may be
regarded as another robustness measure in this context. Our next theorem gives
the explicit form of this index.

\begin{theorem}
\label{THM:7Chi_infl_fact_slope} Assume that $\boldsymbol{K}_{\beta
}(\boldsymbol{\theta}_{0})$ is a full rank matrix. If $g=f_{\varepsilon
,\boldsymbol{y}}$, then the infinitesimal change in the CSIF of the Wald-type
test statistic is given by%
\begin{align}
\frac{\partial}{\partial\varepsilon}\left.  \bar{c}_{\beta,\varepsilon
,\boldsymbol{y}}(\boldsymbol{\theta}_{0})\right\vert _{\varepsilon=0}  &
=\frac{2}{p}\left(  \beta\boldsymbol{u}_{\boldsymbol{\theta}_{0}}%
^{T}(\boldsymbol{y})\boldsymbol{J}_{\beta}^{-1}(\boldsymbol{\theta}%
_{0})\boldsymbol{u}_{\boldsymbol{\theta}_{0}}(\boldsymbol{y}%
)-f_{\boldsymbol{\theta}_{0}}^{\beta}(\boldsymbol{y})\tau_{\boldsymbol{\theta
}_{0}}(\boldsymbol{y})-\int f_{\boldsymbol{\theta}_{0}}^{1+\beta
}(\boldsymbol{x})\tau_{\boldsymbol{\theta}_{0}}(\boldsymbol{x})d\boldsymbol{x}%
\right) \nonumber\\
&  -\left(  2\beta+1\right)  -\frac{1}{2p}\mathcal{IF}_{2}(\boldsymbol{y}%
,W_{\beta}^{0},F_{\boldsymbol{\theta}_{0}}), \label{res}%
\end{align}
where $\mathcal{IF}_{2}(\cdot,W_{\beta}^{0},F_{\boldsymbol{\theta}_{0}})$ is
(\ref{IFb}) and%
\[
\tau_{\boldsymbol{\theta}_{0}}(\cdot)=\mathrm{trace}\left(  \boldsymbol{I}%
_{\boldsymbol{\theta}_{0}}(\cdot)\boldsymbol{J}_{\beta}^{-1}%
(\boldsymbol{\theta}_{0})\right)  .
\]

\end{theorem}

\begin{proof}
See Appendix \ref{A6}.
\end{proof}

For the normal location-scale problem, if $\beta>0$, then $\frac{\partial
}{\partial\varepsilon}\left.  \bar{c}_{\beta,\varepsilon,\boldsymbol{y}%
}(\boldsymbol{\theta}_{0})\right\vert _{\varepsilon=0}$ given in Theorem
\ref{THM:7Chi_infl_fact_slope}\ is bounded, implying the robustness of the
Wald-type test statistic towards the assumption on the model.

\begin{corollary}
If $g=f_{\varepsilon,\boldsymbol{y}}$ and the parameter $\theta$ is a scalar
($p=1$), then the infinitesimal change of CSIF\ is given by%
\[
2\frac{f_{\theta_{0}}^{\beta}(\boldsymbol{y})\left(  \beta u_{\theta_{0}}%
^{2}(\boldsymbol{y})-I_{\theta_{0}}(\boldsymbol{y})\right)  -\int
I_{\theta_{0}}(\boldsymbol{x})f_{\theta_{0}}^{1+\beta}(\boldsymbol{x}%
)d\boldsymbol{x}}{J_{\beta}(\theta_{0})}-\left(  2\beta+1\right)  -\frac{1}%
{2}\mathcal{IF}_{2}(\boldsymbol{y},W_{\beta}^{0},F_{\theta_{0}}),
\]
where%
\[
\mathcal{IF}_{2}(\cdot,W_{\beta}^{0},F_{\theta_{0}})=2\frac{\left(  \xi
_{\beta}(\theta_{0})-u_{\theta_{0}}(\cdot)f_{\theta_{0}}^{\beta}%
(\cdot)\right)  ^{2}}{K_{\beta}(\theta_{0})}.
\]

\end{corollary}

We shall now consider the Wald-type test statistic for the composite
hypothesis and derive the infinitesimal change in the CSIF. Let us define
$\boldsymbol{\Sigma}$$_{\beta}^{\ast}(\boldsymbol{\theta})=\boldsymbol{M}%
^{T}(\boldsymbol{\theta})\boldsymbol{\Sigma}_{\beta}(\boldsymbol{\theta
})\boldsymbol{M}(\boldsymbol{\theta})$ and $\boldsymbol{\Sigma}$$_{\beta
,g}^{\ast}(\boldsymbol{\theta})=\boldsymbol{M}^{T}(\boldsymbol{\theta
})\boldsymbol{\Sigma}_{\beta,g}(\boldsymbol{\theta})\boldsymbol{M}%
(\boldsymbol{\theta})$. Then the following theorem is the analogous to Theorem
\ref{THM:7Chi_infl_fact} for the composite hypothesis.

\begin{theorem}
\label{THM:7Chi_infl_fact_comp} Let $\widehat{\boldsymbol{\theta}}_{\beta
,g}=\boldsymbol{T}_{\beta}(G_{n})$ be the MDPDE with tuning parameter $\beta$.
Then under the composite hypothesis (\ref{EQ:7null_hyp_func}), the asymptotic
distribution of the Wald-type test statistic is given by
\[
W_{n}({\boldsymbol{\widehat{\boldsymbol{\theta}}}}_{\beta,g})=n\boldsymbol{m}%
^{T}({\boldsymbol{\widehat{\boldsymbol{\theta}}}}_{\beta,g})\left(
\boldsymbol{M}^{T}({\boldsymbol{\widehat{\boldsymbol{\theta}}}}_{\beta
,g})\boldsymbol{\Sigma}_{\beta}({\boldsymbol{\widehat{\boldsymbol{\theta}}}%
}_{\beta,g})\boldsymbol{M}({\boldsymbol{\widehat{\boldsymbol{\theta}}}}%
_{\beta,g})\right)  ^{-1}\boldsymbol{m}%
({\boldsymbol{\widehat{\boldsymbol{\theta}}}}_{\beta,g})\underset{n\rightarrow
\infty}{\overset{\mathcal{L}}{\longrightarrow}}\sum_{i=1}^{r}c_{i,\beta
,g}^{\ast}(\boldsymbol{\theta}_{0})Z_{i}^{2},
\]
where $\{Z_{i}\}_{i=1}^{r}$ are i.i.d. standard normal random variables and
$\{c_{i,\beta,g}^{\ast}(\boldsymbol{\theta}_{0})\}_{i=1}^{r}$ the set of
eigenvalues of\linebreak$\boldsymbol{\Sigma}_{\beta}^{\ast-1}%
(\boldsymbol{\theta}_{0})\boldsymbol{\Sigma}_{\beta,g}^{\ast}%
(\boldsymbol{\theta}_{0}).$
\end{theorem}

\begin{proof}
The proof of this theorem directly follows from (\ref{dist2}) using Corollary
2.2 of
\citet{DikGunst}%
.
\end{proof}

Theorem \ref{THM:7Chi_infl_fact_comp} shows that the asymptotic null
distribution of the Wald-type test statistic is a linear combination of $r$
independent variables with $\chi_{1}^{2}$ densities. On the other hand, if the
assumed model is correct, the asymptotic null distribution turns out to be
$\chi_{r}^{2}$. So the Chi-Square Inflation Factor of the Wald-type test
statistic for the composite hypothesis is defined by
\begin{equation}
\bar{c}_{\beta,g}^{\ast}(\boldsymbol{\theta}_{0})=\frac{1}{r}\sum_{i=1}%
^{r}c_{i,\beta,g}^{\ast}(\boldsymbol{\theta}_{0})=\frac{1}{p}\mathrm{trace}%
\left(  \boldsymbol{\Sigma}_{\beta}^{\ast-1}(\boldsymbol{\theta}%
_{0})\boldsymbol{\Sigma}_{\beta,g}^{\ast}(\boldsymbol{\theta}_{0})\right)  .
\label{inflation_comp}%
\end{equation}
The following theorem gives the expression for the infinitesimal change in the
CSIF of the Wald-type test statistic at the model. Let us denote $\bar
{c}_{\beta,\varepsilon,\boldsymbol{y}}^{\ast}(\boldsymbol{\theta}_{0})$, in
the place of $\bar{c}_{\beta,g}^{\ast}(\boldsymbol{\theta}_{0})$ with
$g=f_{\varepsilon,\boldsymbol{y}}$.

\begin{theorem}
\label{THM:7Chi_infl_fact_slope_comp} Consider the composite null hypothesis
$H_{0}:m\left(  \boldsymbol{T}_{\beta}(G)\right)  =0$. If $g=f_{\varepsilon
,\boldsymbol{y}}$, then the infinitesimal change in the CSIF of the Wald-type
test statistic at the model is given by
\begin{align*}
&  \frac{\partial}{\partial\varepsilon}\left.  \bar{c}_{\beta,\varepsilon
,\boldsymbol{y}}^{\ast}(\boldsymbol{\theta}_{0})\right\vert _{\varepsilon
=0}=\frac{2}{r}\left(  \beta\boldsymbol{u}_{\boldsymbol{\theta}_{0}}%
^{T}(\boldsymbol{y})\boldsymbol{\Sigma}_{\beta}(\boldsymbol{\theta}%
_{0})\boldsymbol{M}(\boldsymbol{\theta}_{0})\boldsymbol{\Sigma}_{\beta}%
^{\ast-1}(\boldsymbol{\theta}_{0})\boldsymbol{M}_{\beta}^{T}%
(\boldsymbol{\theta}_{0})\boldsymbol{J}_{\beta}^{-1}(\boldsymbol{\theta}%
_{0})\boldsymbol{u}_{\boldsymbol{\theta}_{0}}(\boldsymbol{y}%
)-f_{\boldsymbol{\theta}_{0}}^{\beta}(\boldsymbol{y})\tau_{\boldsymbol{\theta
}_{0}}^{\ast}(\boldsymbol{y})-\int f_{\boldsymbol{\theta}_{0}}^{1+\beta
}(\boldsymbol{x})\tau_{\boldsymbol{\theta}_{0}}^{\ast}(\boldsymbol{x}%
)d\boldsymbol{x}\right) \\
&  -\left(  2\beta+1\right)  -\frac{1}{2r}\mathcal{IF}_{2}(\boldsymbol{y}%
,W_{\beta},F_{\boldsymbol{\theta}_{0}}),
\end{align*}
where $\mathcal{IF}_{2}(\cdot,W_{\beta},F_{\boldsymbol{\theta}_{0}})$ is
(\ref{IF2}) and%
\[
\tau_{\boldsymbol{\theta}_{0}}^{\ast}(\cdot)=\mathrm{trace}\left(
\boldsymbol{I}_{\boldsymbol{\theta}_{0}}(\cdot)\boldsymbol{\Sigma}_{\beta
}(\boldsymbol{\theta}_{0})\boldsymbol{M}(\boldsymbol{\theta}_{0}%
)\boldsymbol{\Sigma}_{\beta}^{\ast-1}(\boldsymbol{\theta}_{0})\boldsymbol{M}%
_{\beta}^{T}(\boldsymbol{\theta}_{0})\boldsymbol{J}_{\beta}^{-1}%
(\boldsymbol{\theta}_{0})\right)  .
\]

\end{theorem}

\begin{proof}
See Appendix \ref{A7}.
\end{proof}

\section{Examples\label{sec6}}

For the location-scale parameters of a normal model it is easy to verify the
robustness properties of the Wald-type tests using the theoretical results
derived in this paper. In this section we have presented two other examples,
and justified the stability of the levels and powers of the Wald-type tests in
presence of outliers. On the other hand, it is shown that the classical Wald
tests break down as their power influence functions are unbounded.

\subsection{Test for Exponentiality against Weibull Alternatives}

Our first example considers an interesting problem from quality control and
examine the performance of the proposed MDPDE based Wald-type test for solving
it. Suppose we have $n$ independent sample observations $X_{1},\ldots,X_{n}$
from a lifetime distribution having density $f(x)$. We want to test the null
hypothesis that the underlying lifetime (random variable) follows an
exponential distribution against the alternative of Weibull distribution. In
other words, we want to test the hypothesis
\begin{align}
&  H_{0}:f(x)=f_{\mathrm{Exp},\sigma}(x)=\frac{1}{\sigma}e^{-\frac{x}{\sigma}%
},\quad x>0,\nonumber\\
&  \text{against}\nonumber\\
&  H_{1}:f(x)=f_{\mathrm{Weib},\theta,\sigma}(x)=\frac{\theta}{\sigma}\left(
\frac{x}{\sigma}\right)  ^{\theta-1}e^{-\left(  \frac{x}{\sigma}\right)
^{\theta}},\quad x>0. \label{Ex1-null_hyp1}%
\end{align}
Here $\theta>0$ is the shape parameter of the lifetime distribution and
$\sigma>0$ is the scale parameter. Further note that without loss of
generality, we can assume that the data are properly scaled so that we can
take $\sigma=1$ (this fact can also be tested first by applying the same
Wald-type test; see Section 4.2 of \cite{2014arXiv1403.7616B}). Then, we
consider the model $\mathcal{F}=\{f_{\theta}(x)=\theta x^{\theta
-1}e^{-x^{\theta}}:x>0,\theta>0\}$ so that we have $n$ i.i.d. observations
$X_{1},\ldots,X_{n}$ from this family and the null hypothesis
(\ref{Ex1-null_hyp1}) simplifies to
\begin{equation}
H_{0}:\theta=1\quad\text{against}\quad H_{1}:\theta\neq1.
\label{Ex-1-null_hyp}%
\end{equation}
This problem is now exactly similar to the simple hypothesis testing problem
considered in this paper. So we can construct a robust Wald-type test using
the MDPDE $\widehat{\theta}_{\beta}$ of $\theta$.

Note that the MDPDE $\widehat{\theta}_{\beta}$ of $\theta$, in this particular
example, is to be obtained by minimizing the objective function
\[
\frac{\theta^{\beta}}{(1+\beta)^{1+\beta-\frac{\beta}{\theta}}}\Gamma\left(
1+\beta-\frac{\beta}{\theta}\right)  -\frac{(1+\beta)\theta^{\beta}}{\beta
n}\sum_{i=1}^{n}X_{i}^{\beta(\theta-1)}e^{-\beta X_{i}^{\theta}},
\]
with respect to $\theta>0$, where $\Gamma(\cdot)$ represents the gamma
function. As noted in Section \ref{sec2}, $\widehat{\theta}_{\beta}$ is
$\sqrt{n}$-consistent and asymptotically normal. A straightforward calculation
shows that, under $H_{0}:\theta_{0}=1$, its asymptotic variance is given by
$\frac{\eta_{2\beta}}{\eta_{\beta}^{2}}$, where
\[
\eta_{\beta}=\frac{1}{1+\beta}+\left(  C_{2,\beta}+2C_{1,\beta}\right)  ,
\]
with
\[
C_{\alpha,\beta}=\int\left(  (1-y)\log(y)\right)  ^{\alpha}e^{-(1+\beta)y}dy.
\]
Thus, the MDPDE based Wald-type test statistics for testing the simple
hypothesis (\ref{Ex-1-null_hyp}) is given by
\[
W_{n}^{0}(\widehat{\theta}_{\beta})=\frac{n\eta_{\beta}^{2}}{\eta_{2\beta}%
}\left(  \widehat{\theta}_{\beta}-1\right)  ^{2},
\]
which asymptotically follows a chi-square distribution with one degree of
freedom. Further, at the contiguous alternatives $H_{1,n}:\theta
_{n}=1+n^{-1/2}{d}$, this test statistic has an asymptotic non-central
chi-square distribution with one degree of freedom and non-centrality
parameter $\delta=\frac{d^{2}\eta_{\beta}^{2}}{\eta_{2\beta}}$. Note that, for
any fixed level of significance, the asymptotic power of the Wald-type test
statistic under the contiguous alternative decreases as the non-centrality
parameter $\delta$ decreases and for any fixed $d$ it happens as $\beta$
increases. Table \ref{TAB:ex-1-cont-power} represents the asymptotic power for
different values of $d$ and $\beta$. It is clear from the table that there is
no significant loss in contiguous power of this test for smaller positive
values of $\beta$.\hspace*{0.1cm}%

\begin{table}[htbp]  \tabcolsep2.8pt  \centering
\caption{Asymptotic Power of the Wald-type test of (\ref{Ex-1-null_hyp}) at $5\%$ level
of significance for different $d$ and $\beta$.\label{TAB:ex-1-cont-power}}%
\begin{tabular}
[c]{llccccccccc}\hline
\hspace*{0.1cm} &  & \hspace*{0.1cm} & \multicolumn{7}{c}{$\beta$} &
\hspace*{0.1cm}\\
& \multicolumn{1}{r}{$d$} &  & 0 & 0.01 & 0.1 & 0.3 & 0.5 & 0.7 & 1 & \\\hline
& \multicolumn{1}{r}{0} &  & 0.050 & 0.050 & 0.050 & 0.050 & 0.050 & 0.050 &
0.050 & \\
& \multicolumn{1}{r}{2} &  & 0.778 & 0.788 & 0.747 & 0.617 & 0.558 & 0.502 &
0.473 & \\
& \multicolumn{1}{r}{3} &  & 0.981 & 0.984 & 0.975 & 0.930 & 0.880 & 0.825 &
0.790 & \\
& \multicolumn{1}{r}{4} &  & 1.000 & 1.000 & 1.000 & 0.996 & 0.983 & 0.973 &
0.967 & \\
& \multicolumn{1}{r}{5} &  & 1.000 & 1.000 & 1.000 & 1.000 & 1.000 & 0.999 &
0.995 & \\
& \multicolumn{1}{r}{10} &  & 1.000 & 1.000 & 1.000 & 1.000 & 1.000 & 1.000 &
1.000 & \\\hline
\end{tabular}%
\end{table}%

Next consider the robustness of the proposed Wald-type test as derived above.
From the density of the model, it is easy to see that the score function is
given by
\[
u_{\theta}(x)=\frac{1}{\theta}+(1-x^{\theta})\log x,
\]
so that the influence function of the minimum DPD functional ${T}_{\beta}$
under the null hypothesis (\ref{Ex-1-null_hyp}) is given by
\[
\mathcal{IF}({x},{T}_{\beta},F_{{\theta}_{0}})=\frac{1}{\eta_{\beta}}\left(
{1}+(1-x)\log x\right)  e^{-\beta x}.
\]
Therefore, using the result derived in Section \ref{sec3}, the second order
influence function of the Wald-type test statistics $W_{\beta}^{0}$ becomes
\[
\mathcal{IF}_{2}({x},{W}_{\beta}^{0},F_{{\theta}_{0}})=\frac{2}{\eta_{2\beta}%
}\left(  {1}+(1-x)\log x\right)  ^{2}e^{-2\beta x}.
\]
Note that its first order influence function is always zero at the simple
null. Figure \ref{FIG:ex-1-IF-test} presents the second order influence
function for several $\beta$. The boundedness of this second order influence
function is quite clear from the figure implying the robustness of the
proposed Wald-type test. However, the influence function of the classical Wald
test at $\beta=0$ is unbounded implying its non-robustness.

\begin{figure}[ptbh]
\centering
\subfloat[Influence function of the test statistics]{
\includegraphics[width=0.5\textwidth]{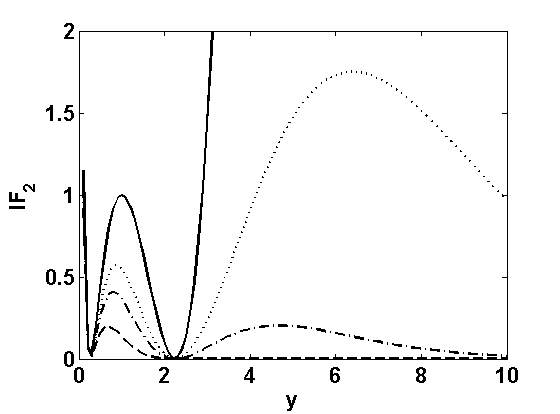}
\label{FIG:ex-1-IF-test}} ~
\subfloat[Power Influence function]{
\includegraphics[width=0.5\textwidth]{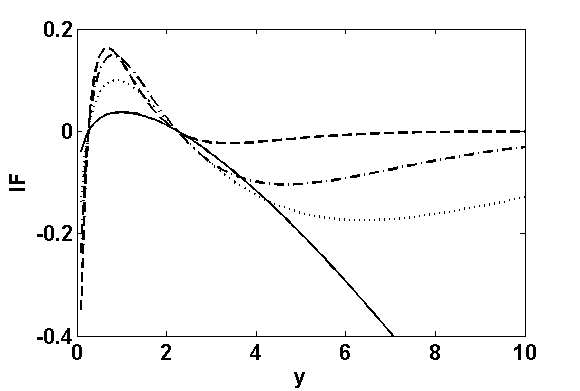}
\label{FIG:ex-1-IF-power}}
\caption{Influence functions of MDPDE based Wald-type test of
(\ref{Ex-1-null_hyp}) for different values of $\beta$ (solid line: $\beta=0$,
dotted line: $\beta=0.3$, dashed-dotted line: $\beta=0.5$, dashed line:
$\beta=1$).}%
\label{FIG:ex-1-IF}%
\end{figure}

Finally, let us examine the level and power stability of the proposed
Wald-type test. Following the results derived in Section \ref{Sec4}, the level
influence function of any order will be zero at the null implying the
robustness of its asymptotic level. Further, the power influence function of
the Wald-type test at the contiguous alternatives $\theta_{n}$ is given by
\begin{equation}
\mathcal{PIF}({x},W_{\beta}^{0},F_{{\theta}_{0}})\cong K_{1}^{\ast}\left(
\frac{d^{2}\eta_{\beta}^{2}}{\eta_{2\beta}}\right)  \frac{d\eta_{\beta}}%
{\eta_{2\beta}}\left(  {1}+(1-x)\log x\right)  e^{-\beta x},\nonumber
\end{equation}
where $K_{p}^{\ast}(s)$ is as defined in Theorem \ref{THM:PIF_composite}.
Figure \ref{FIG:ex-1-IF-power} shows the power influence function for some
particular $\beta$. Once again, the power robustness of the proposed test for
$\beta>0$ is clearly visible from the figure.

\newpage

\subsection{Test for Correlation in Bivariate Normal\label{Sec6.2}}

Let us now consider another interesting hypothesis testing problem involving
the correlation parameter of two normal populations with unknown means and
variances; this problem often arises in several real life applications when we
want to check for the association between any two sets of observation only
assuming the normality of those two populations. Consider the observations
$\boldsymbol{X}_{i}=(X_{i1},X_{i2})^{T}$, $i=1,\ldots,n$, from the bivariate
normal model $\{\mathcal{N}(\boldsymbol{\mu},\boldsymbol{\Sigma})\}$ where
$\boldsymbol{\mu}=(\mu_{1},\mu_{2})^{T}\in\mathbb{R}^{2}$ and
\[
\boldsymbol{\Sigma}=%
\begin{pmatrix}
\sigma_{1}^{2} & \rho\sigma_{1}\sigma_{2}\\
\rho\sigma_{1}\sigma_{2} & \sigma_{2}^{2}%
\end{pmatrix}
\]
belongs to the set of $2\times2$ positive definite matrices. Thus, our
parameter of interest is $\boldsymbol{\theta}=(\mu_{1},\mu_{2},\sigma
_{1},\sigma_{2},\rho)^{T}$ with the parameter space $\Theta=\mathbb{R}%
^{2}\times\mathbb{R}^{+}\times\mathbb{R}^{+}\times\lbrack-1,1]$. We want to
test for the composite hypothesis
\begin{equation}
H_{0}:\rho=0\quad\text{against}\quad H_{1}:\rho\neq0, \label{Ex-2-null_hyp}%
\end{equation}
with values of $\mu_{1}$, $\mu_{2}$, $\sigma_{1}$ and $\sigma_{2}$ being
unspecified. In terms of notations of Section \ref{sec2}, we have $r=1$
restrictions with $m(\boldsymbol{\theta})=\rho$ so that $\boldsymbol{M}%
(\boldsymbol{\theta})$ is a $5\times1$ matrix with the last entry $1$ and rest
$0$ and the null parameter space is $\Theta_{0}=\mathbb{R}^{2}\times
\mathbb{R}^{+}\times\mathbb{R}^{+}\times\{0\}$. We shall now develop the
Wald-type test statistic for this composite hypothesis along with its properties.

Using the form of the bivariate normal density, we can see that the MDPDE
${\boldsymbol{\widehat{\boldsymbol{\theta}}}}_{\beta}=(\widehat{\mu}_{1,\beta
},\widehat{\mu}_{2,\beta},\widehat{\sigma}_{1,\beta},\widehat{\sigma}%
_{2,\beta},\widehat{\rho}_{\beta})^{T}$ of $\boldsymbol{\theta}$ with
$\beta>0$ is the minimizer of
\[
\frac{1}{(2\pi)^{\beta}\sigma_{1}^{\beta}\sigma_{2}^{\beta}(1-\rho^{2}%
)^{\beta/2}}\left(  \frac{1}{\sqrt{1+\beta}}-\frac{1+\beta}{n\beta}\sum
_{i=1}^{n}e^{-\frac{\Upsilon(\boldsymbol{X}_{i},\boldsymbol{\theta})}{2}%
}\right)  ,
\]
with respect to $\boldsymbol{\theta}$, where $\Upsilon(\boldsymbol{x}%
,\boldsymbol{\theta})=(\boldsymbol{x}-\boldsymbol{\theta})^{T}%
\boldsymbol{\Sigma}^{-1}(\boldsymbol{x}-\boldsymbol{\theta})$. Take any
$\boldsymbol{\theta}_{0}=(\mu_{1,0},\mu_{2,0},\sigma_{1,0},\sigma_{2,0}%
,0)^{T}\in\Theta_{0}$. Then the asymptotic variance of the MDPDE
${\boldsymbol{\widehat{\boldsymbol{\theta}}}}_{\beta}$ under
$\boldsymbol{\theta}=\boldsymbol{\theta}_{0}$ is given by $\boldsymbol{\Sigma
}_{\beta}(\boldsymbol{\theta}_{0})=\boldsymbol{J}_{\beta}^{-1}%
(\boldsymbol{\theta}_{0})\boldsymbol{K}_{\beta}(\boldsymbol{\theta}%
_{0})\boldsymbol{J}_{\beta}^{-1}(\boldsymbol{\theta}_{0})$. A straightforward
but lengthy calculation shows that
\[
\boldsymbol{J}_{\beta}(\boldsymbol{\theta}_{0})=%
\begin{pmatrix}
\frac{C_{\beta}}{(1+\beta)^{3/2}\sigma_{1}^{2}} & 0 & 0 & 0 & 0\\
0 & \frac{C_{\beta}}{(1+\beta)^{3/2}\sigma_{2}^{2}} & 0 & 0 & 0\\
0 & 0 & \frac{(2+\beta^{2})C_{\beta}}{\sigma_{1}^{2}(1+\beta)^{5/2}} &
\frac{\beta^{2}C_{\beta}}{\sigma_{1}\sigma_{2}(1+\beta)^{5/2}} & 0\\
0 & 0 & \frac{\beta^{2}C_{\beta}}{\sigma_{1}\sigma_{2}(1+\beta)^{5/2}} &
\frac{(2+\beta^{2})C_{\beta}}{\sigma_{2}^{2}(1+\beta)^{5/2}} & 0\\
0 & 0 & 0 & 0 & \frac{C_{\beta}}{(1+\beta)^{5/2}}%
\end{pmatrix}
\]
and
\[
\boldsymbol{K}_{\beta}(\boldsymbol{\theta}_{0})=%
\begin{pmatrix}
\frac{C_{2\beta}}{(1+2\beta)^{3/2}\sigma_{1}^{2}} & 0 & 0 & 0 & 0\\
0 & \frac{C_{2\beta}}{(1+2\beta)^{3/2}\sigma_{2}^{2}} & 0 & 0 & 0\\
0 & 0 & \frac{(2+3\beta^{2})C_{2\beta}}{\sigma_{1}^{2}(1+2\beta)^{5/2}} &
\frac{3\beta^{2}C_{2\beta}}{\sigma_{1}\sigma_{2}(1+2\beta)^{5/2}} & 0\\
0 & 0 & \frac{3\beta^{2}C_{2\beta}}{\sigma_{1}\sigma_{2}(1+2\beta)^{5/2}} &
\frac{(2+3\beta^{2})C_{2\beta}}{\sigma_{2}^{2}(1+2\beta)^{5/2}} & 0\\
0 & 0 & 0 & 0 & \frac{C_{2\beta}}{(1+2\beta)^{5/2}}%
\end{pmatrix}
\]
where $C_{\beta}=(2\pi)^{-\beta}\sigma_{1}^{-\beta}\sigma_{2}^{-\beta}$ and
$C_{\beta}^{\ast}=4C_{2\beta}-C_{\beta}^{2}$. Hence,
\[
\boldsymbol{\Sigma}_{\beta}(\boldsymbol{\theta}_{0})=%
\begin{pmatrix}
\zeta_{\beta}^{3/2}\sigma_{1}^{2} & 0 & 0 & 0 & 0\\
0 & \zeta_{\beta}^{3/2}\sigma_{2}^{2} & 0 & 0 & 0\\
0 & 0 & \zeta_{\beta}^{5/2}\kappa_{\beta}^{1}\sigma_{1}^{2} & \zeta_{\beta
}^{5/2}\kappa_{\beta}^{2}\sigma_{1}\sigma_{2} & 0\\
0 & 0 & \zeta_{\beta}^{5/2}\kappa_{\beta}^{2}\sigma_{1}\sigma_{2} &
\zeta_{\beta}^{5/2}\kappa_{\beta}^{1}\sigma_{2}^{2} & 0\\
0 & 0 & 0 & 0 & \zeta_{\beta}^{5/2}%
\end{pmatrix}
\]
with
\[
\zeta_{\beta}=1+\frac{\beta^{2}}{1+2\beta}\text{,}\quad\kappa_{\beta}%
^{1}=\frac{(\beta^{4}+5\beta^{2}+2)}{(1+\beta^{2})^{2}}\quad\text{and}%
\quad\kappa_{\beta}^{2}=\frac{\beta^{2}(1-\beta^{2})}{(1+\beta^{2})^{2}}.
\]
Interestingly, note that whenever the null hypothesis $\rho=0$ is true the
MDPDE of $\mu_{1}$, $\mu_{2}$ and $\rho$ are asymptotically independent of
each other and also of the MDPDE of $\sigma_{1}$ and $\sigma_{2}$.

Now the robust Wald-type test statistic (\ref{22.17}) for testing the null
hypothesis (\ref{Ex-2-null_hyp}) is given by
\begin{equation}
W_{n}({\boldsymbol{\widehat{\boldsymbol{\theta}}}}_{\beta})=n\frac
{\widehat{\rho}_{\beta}^{2}}{\zeta_{\beta}^{5/2}}, \label{Ex-2.17}%
\end{equation}
which asymptotically follows a chi-square distribution with one degree of
freedom under the null hypothesis. Note that, at $\beta=0$, $\widehat{\rho
}_{\beta}$ coincides with the maximum likelihood estimator of $\rho$ and hence
the proposed test $W_{n}$ coincides with the classical Wald test for the
present problem. Further, under the contiguous alternatives $H_{1,n}^{\ast
}:\rho_{n}=n^{-1/2}d$, the asymptotic distribution of $W_{n}%
({\boldsymbol{\widehat{\boldsymbol{\theta}}}}_{\beta})$ is a non-central
chi-square distribution with one degree of freedom and non-centrality
parameter $\zeta_{\beta}^{-5/2}d^{2}$. Note that, for any fixed level of
significance, the asymptotic power of the Wald-type test under the contiguous
alternative hypotheses decreases as the non-centrality parameter decreases and
for any fixed $d$ it happens as $\beta$ increases. However, as we can see from
Table \ref{TAB:ex-2-cont-power}, the loss in contiguous power of the Wald-type
test is not very significant for smaller positive values of $\beta$.%

\begin{table}[htbp]  \tabcolsep2.8pt  \centering
\caption{Asymptotic Power of the MDPDE based Wald-type test of (\ref{Ex-2-null_hyp}) $5\%$ level
of significance for different $\delta$ and $\beta$.\label{TAB:ex-2-cont-power}}%
\begin{tabular}
[c]{llccccccccc}\hline
\hspace*{0.1cm} &  & \hspace*{0.1cm} & \multicolumn{7}{c}{$\beta$} &
\hspace*{0.1cm}\\
& \multicolumn{1}{r}{$d$} &  & 0 & 0.01 & 0.1 & 0.3 & 0.5 & 0.7 & 1 & \\\hline
& \multicolumn{1}{r}{0} &  & 0.050 & 0.050 & 0.050 & 0.050 & 0.050 & 0.050 &
0.050 & \\
& \multicolumn{1}{r}{2} &  & 0.516 & 0.516 & 0.508 & 0.463 & 0.408 & 0.354 &
0.287 & \\
& \multicolumn{1}{r}{3} &  & 0.851 & 0.851 & 0.844 & 0.800 & 0.735 & 0.662 &
0.553 & \\
& \multicolumn{1}{r}{4} &  & 0.979 & 0.979 & 0.977 & 0.962 & 0.932 & 0.887 &
0.797 & \\
& \multicolumn{1}{r}{5} &  & 0.999 & 0.999 & 0.999 & 0.997 & 0.991 & 0.978 &
0.937 & \\
& \multicolumn{1}{r}{10} &  & 1.000 & 1.000 & 1.000 & 1.000 & 1.000 & 1.000 &
1.000 & \\\hline
\end{tabular}%
\end{table}%

Now let us examine the robustness of this Wald-type test based on the results
derived in the present paper. Note that the influence function of the minimum
DPD functional $\boldsymbol{T}_{\beta}$ here under the null
$\boldsymbol{\theta}=\boldsymbol{\theta}_{0}$ is given by
\[
\mathcal{IF}(\boldsymbol{x},\boldsymbol{T}_{\beta},F_{\boldsymbol{\theta}_{0}%
})=%
\begin{pmatrix}
(1+\beta)^{3/2}(x_{1}-\mu_{1})\\
(1+\beta)^{3/2}(x_{2}-\mu_{2})\\
\frac{(1+\beta)^{5/2}\sigma_{1}}{(1+\beta^{2})}\left(  \frac{(2+\beta
^{2})(x_{1}-\mu_{1})^{2}}{\sigma_{1}^{2}}-\frac{\beta^{2}(x_{2}-\mu_{2})^{2}%
}{\sigma_{2}^{2}}-2\right) \\
\frac{(1+\beta)^{5/2}\sigma_{2}}{(1+\beta^{2})}\left(  \frac{(2+\beta
^{2})(x_{2}-\mu_{2})^{2}}{\sigma_{2}^{2}}-\frac{\beta^{2}(x_{1}-\mu_{1})^{2}%
}{\sigma_{1}^{2}}-2\right) \\
(1+\beta)^{3/2}\frac{(x_{1}-\mu_{1})(x_{2}-\mu_{2})}{\sigma_{1}\sigma_{2}}%
\end{pmatrix}
e^{-\frac{\beta}{2}\left(  \frac{(x_{1}-\mu_{1})^{2}}{\sigma_{1}^{2}}%
+\frac{(x_{2}-\mu_{2})^{2}}{\sigma_{2}^{2}}\right)  }-%
\begin{pmatrix}
0\\
0\\
\frac{\beta(1+\beta)^{2}}{(1+\beta^{2})}\sigma_{1}\\
\frac{\beta(1+\beta)^{2}}{(1+\beta^{2})}\sigma_{2}\\
0
\end{pmatrix}
.
\]
Using the result derived in Section \ref{sec3}, the first order influence
function of the Wald-type test statistic $W_{\beta}$ is zero at the null and
its second order influence function at the null is given by
\[
\mathcal{IF}_{2}(\boldsymbol{x},W_{\beta},F_{\boldsymbol{\theta}_{0}}%
)=\frac{2(1+2\beta)^{5/2}}{(1+\beta)^{2}\sigma_{1}^{2}\sigma_{2}^{2}}%
(x_{1}-\mu_{1})^{2}(x_{2}-\mu_{2})^{2}e^{-\beta\left(  \frac{(x_{1}-\mu
_{1})^{2}}{\sigma_{1}^{2}}+\frac{(x_{2}-\mu_{2})^{2}}{\sigma_{2}^{2}}\right)
}.
\]
Clearly, this influence function is unbounded at $\beta=0$, but whenever
$\beta>0$ it is bounded implying the robustness of the corresponding test
statistics. Figure \ref{FIG:ex-2-IF-test} shows the plot of this influence
function for some particular $\beta$. It is clear from the figures that the
extend of the influence function over the contamination point $\boldsymbol{x}%
=(x_{1},~x_{2})^{T}$ decreases as $\beta$ increases. this fact can also bee
seen by looking at the \textit{gross-error sensitivity} of the test statistics
given by
\[
\gamma_{\beta}^{\ast}=\left\{
\begin{array}
[c]{lc}%
\frac{2n(1+2\beta)^{5/2}}{\sqrt{\beta}(1+\beta)^{2}}e^{-\sqrt{\beta}}, &
\text{if }\beta>0,\\
\infty, & \text{if }\beta=0.
\end{array}
\right.
\]
Clearly $\gamma_{\beta}^{\ast}$ decreases as $\beta$ increases implying that
the extent of robustness of the MDPDE based Wald-type test statistics increases.

\begin{figure}[ptbh]
\centering
\subfloat[$\beta=0$]{
\includegraphics[width=0.5\textwidth]{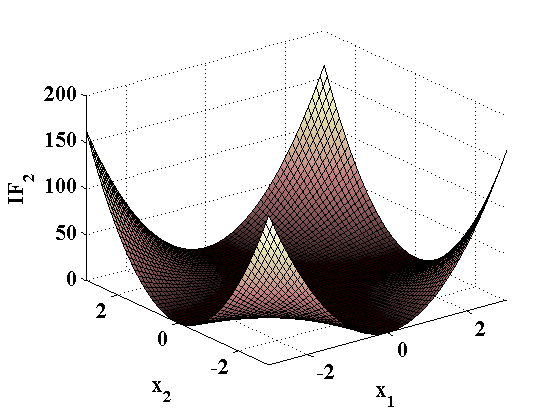}
\label{FIG:1a}} ~
\subfloat[$\beta=0.1$]{
\includegraphics[width=0.5\textwidth]{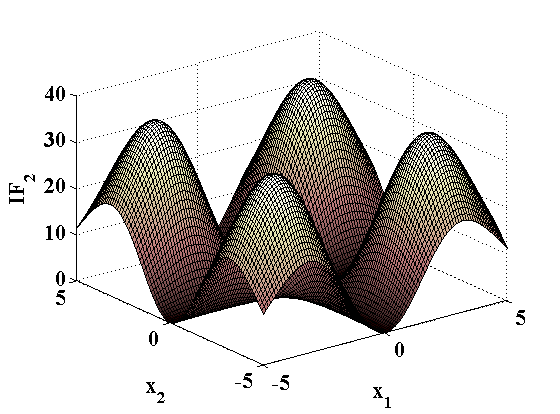}
\label{FIG:1b}} \newline\subfloat[$\beta=0.3$]{
\includegraphics[width=0.5\textwidth]{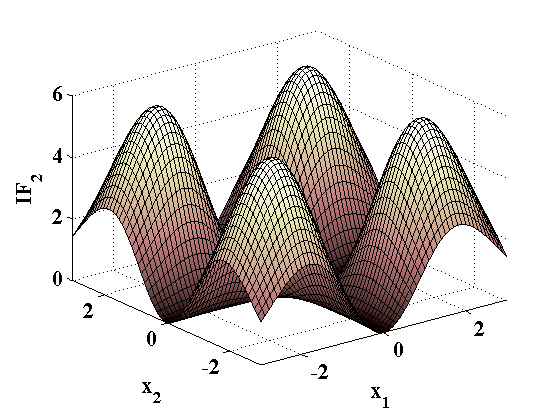}
\label{FIG:1c}} ~
\subfloat[$\beta=1$]{
\includegraphics[width=0.5\textwidth]{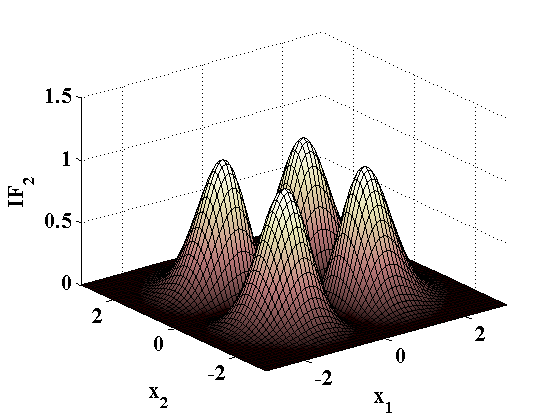}
\label{FIG:1d}}
\caption{Influence function of Wald-type test statistics for testing of
(\ref{Ex-2-null_hyp}) at the null for different values of $\beta$ (Here we
have taken $\mu_{1}=\mu_{2}=0$ and $\sigma_{1} = \sigma_{2} =1$).}%
\label{FIG:ex-2-IF-test}%
\end{figure}

Next, we shall consider the level and power stability of the present test. As
shown in Section \ref{Sec4.2}, the level influence function of any order will
be zero at the null hypothesis. Hence the level of the Wald-type test,
constructed using asymptotic distribution, will be robust under infinitesimal
contamination. On the other hand, if we consider the contamination proportion
and the difference of alternatives $\rho_{n}$ from null converges to zero at
the same rate of $n^{-1/2}$ ($\rho_{n}=n^{-1/2}d$), the power influence
function of this test is given by
\begin{equation}
\mathcal{PIF}(\boldsymbol{x},W_{\beta},F_{\boldsymbol{\theta}_{0}})K_{1}%
^{\ast}\left(  \zeta_{\beta}^{-5/2}d\right)  \frac{(1+\beta)^{3/2}\zeta
_{\beta}^{-5/2}d}{\sigma_{1}\sigma_{2}}(x_{1}-\mu_{1})(x_{2}-\mu_{2}%
)e^{-\frac{\beta}{2}\left(  \frac{(x_{1}-\mu_{1})^{2}}{\sigma_{1}^{2}}%
+\frac{(x_{2}-\mu_{2})^{2}}{\sigma_{2}^{2}}\right)  },\nonumber
\end{equation}
where $K_{p}^{\ast}(s)$ is as defined in \ref{THM:7asymp_power_one}.

Again, it is clear that the above power influence function of the MDPDE based
test statistic is bounded for all $\beta>0$ and unbounded at $\beta=0$ (see
Figure \ref{FIG:ex-2-IF-power}). This justifies the power robustness of the
proposed MDPDE based Wald-types tests with $\beta>0$ over the usual Wald test
at $\beta=0$.

\begin{figure}[ptbh]
\centering
\subfloat[$\beta=0$]{
\includegraphics[width=0.5\textwidth]{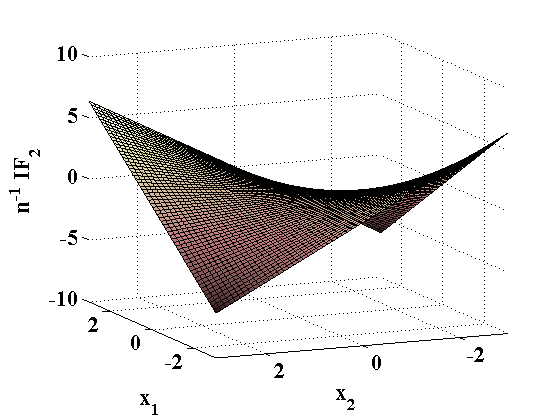}
\label{FIG:2a}} ~
\subfloat[$\beta=0.1$]{
\includegraphics[width=0.5\textwidth]{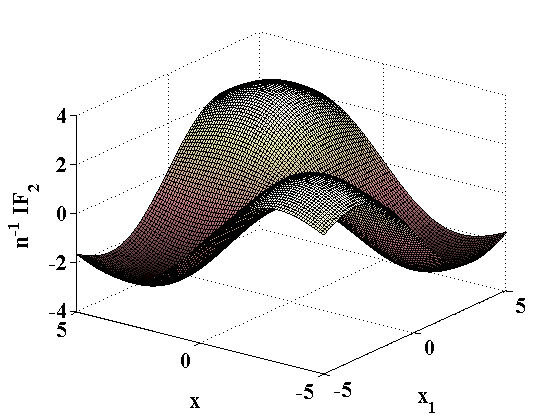}
\label{FIG:2b}} \newline\subfloat[$\beta=0.3$]{
\includegraphics[width=0.5\textwidth]{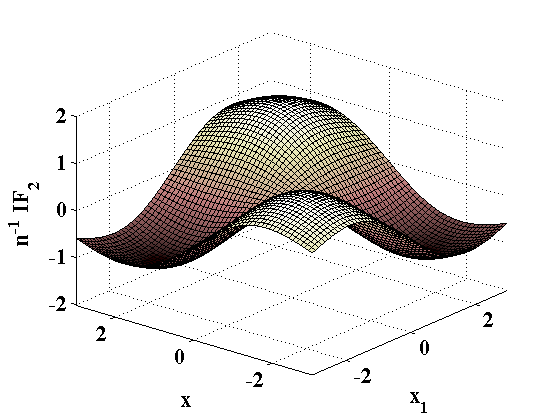}
\label{FIG:2c}} ~
\subfloat[$\beta=1$]{
\includegraphics[width=0.5\textwidth]{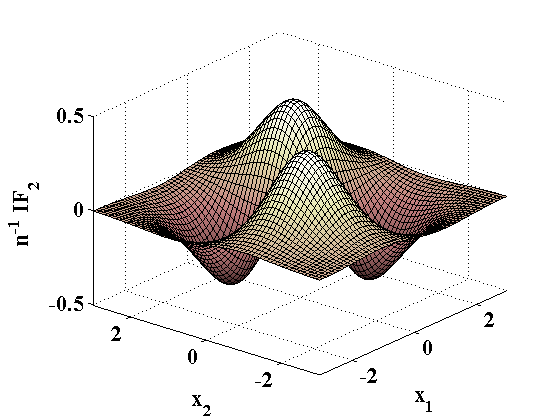}
\label{FIG:2d}}
\caption{Power influence function of Wald-type test of (\ref{Ex-1-null_hyp})
at $5\%$ level of significance and $d=3$ for different values of $\beta$.}%
\label{FIG:ex-2-IF-power}%
\end{figure}

\subsection{Test for the General Linear Hypothesis in Fixed-design Linear
Regression Models}

The robust minimum DPD estimators under the fixed-design Linear Regression
Models are considered in
\citet{ghosh2013}%
, who have also derived their asymptotic and robustness properties in great
detail (also see
\citet{ghosh2015a}%
). Indeed,
\citet{ghosh2013}
considered a general class of models based on the non-homogeneous data and
developed the theory of the MDPDE under that general set-up; the linear
regression with pre-fixed (given) covariates comes as a special case of the
general set-up. Under the same general set-up of independent but
non-homogeneous data,
\citet{ghosh2015b}
have developed the divergence based tests of different kind of statistical
hypothesis and discussed their properties and application in the fixed-design
linear regression model. A nice study about robust M-type testing procedures
for linear models can be seen in
\citet{MarkatouStahelRochetti}%
. Here, we briefly mention the corresponding Wald type test for only the class
of general linear hypothesis and discuss their influence robustness following
the theory developed in this paper.

Suppose we are given a fixed $n\times p$ design matrix, where the $i$-th value
of the $p$ covariates are denoted as $\boldsymbol{x}_{i}=(x_{i1},\ldots
,x_{ip})^{T}$ for $i=1,\ldots,n$. Consider the fixed-design linear regression
model
\begin{equation}
y_{i}=\boldsymbol{x}_{i}^{T}{{\boldsymbol{\vartheta}}}+\epsilon_{i}%
,~~i=1,\ldots,n,
\end{equation}
where the error $\epsilon_{i}$'s are assumed to be i.i.d.~normal with mean
zero and variance $\sigma^{2}$ and ${{\boldsymbol{\vartheta}}}=(\vartheta
_{1},\ldots,\vartheta_{p})^{T}$ denote the vector of regression coefficients.
Then, for each $i$, $y_{i}\sim\mathcal{N}(\boldsymbol{x}_{i}^{T}%
{{\boldsymbol{\vartheta}}},\sigma^{2})$ which are clearly independent but not
identically distributed.

Following
\citet{ghosh2013}%
, we can derive the $\sqrt{n}$-consistent MDPDE $\widehat{{\boldsymbol{\theta
}}}_{\beta}=(\widehat{{\boldsymbol{\vartheta}}}_{\beta}^{T},~\widehat{\sigma
}_{\beta}^{2})^{T}$ of the parameters ${{\boldsymbol{\theta}}}%
=({{\boldsymbol{\vartheta}}}^{T},~\sigma^{2})^{T}$ with tuning parameter
$\beta$, which are asymptotically independent normally distributed under
conditions (R1)--(R2) of
\citet{ghosh2013}%
. In particular, if ${{\boldsymbol{\vartheta}}}_{0}$ and $\sigma_{0}^{2}$ are
the true values of the parameters then we have
\begin{align}
&  \sqrt{n}(\boldsymbol{X}^{T}\boldsymbol{X})^{1/2}%
(\widehat{{\boldsymbol{\vartheta}}}_{\beta}-{{\boldsymbol{\vartheta}}}%
_{0})\underset{n\rightarrow\infty}{\overset{\mathcal{L}}{\longrightarrow}%
}\mathcal{N}_{p}\left(  0,\zeta_{\beta}^{3/2}\sigma_{0}^{2}\right)  ,\\
&  \sqrt{n}(\widehat{\sigma}_{\beta}^{2}-\sigma_{0}^{2})\underset{n\rightarrow
\infty}{\overset{\mathcal{L}}{\longrightarrow}}\mathcal{N}(0,4\zeta_{\beta
}^{5/2}\kappa_{\beta}^{1}\sigma_{0}^{4}),
\end{align}
where $\zeta_{\beta}$ and $\kappa_{\beta}^{1}$ are as defined Section
\ref{Sec6.2} and $\boldsymbol{X}=[\boldsymbol{x}_{1}~\boldsymbol{x}_{2}%
~\cdots~\boldsymbol{x}_{n}]^{T}$.

Now, let us consider the class of general linear hypothesis on
$\boldsymbol{\vartheta}$ with unspecified $\sigma$ as given by
\begin{align}
H_{0}:\boldsymbol{L}^{T}\boldsymbol{\vartheta}=\boldsymbol{l_{0}}%
\qquad\text{against}\qquad H_{1}:\boldsymbol{L}^{T}\boldsymbol{\vartheta}%
\neq\boldsymbol{l_{0}}, \label{EQ:9Gen_lin_hypothesis}%
\end{align}
where the $p\times r$ matrix $\boldsymbol{L}$ is known with rank $r(\leq p)$
and $l_{0}$ is a known $r$-vector. Due to full row rank of the matrix
$\boldsymbol{L}$, there exists a true parameter value $\boldsymbol{\vartheta
}_{0}$ satisfying the null hypothesis $\boldsymbol{L}^{T}\boldsymbol{\vartheta
}_{0}=\boldsymbol{l_{0}}$. In particular, this general class of linear
hypothesis consider the popular problem of testing the significance of the
model $H_{0}:\boldsymbol{\vartheta}=\boldsymbol{\vartheta_{0}}$ where $r=p$,
$\boldsymbol{l_{0}}=\boldsymbol{\vartheta_{0}}$ (usually a zero vector) and
$\boldsymbol{L}=\boldsymbol{I}_{p}$, the identity matrix of order $p$. Also
the test of significance of any one regression component $H_{0}:\vartheta
_{j}=\vartheta_{0j}$ belongs to the class of hypothesis
(\ref{EQ:9Gen_lin_hypothesis}) with $r=1$, $\boldsymbol{l_{0}}=\vartheta_{0j}$
and $\boldsymbol{L}$ is $p$-vector of zeros except the $j$-th component which
is $1$.

In the notation of Section \ref{Sec2.2}, here we have $\boldsymbol{m}%
^{T}\left(  {\boldsymbol{\theta}}\right)  = \boldsymbol{m}^{T}\left(
{\boldsymbol{\vartheta}},{\sigma^{2}}\right)  =\boldsymbol{L}^{T}%
\boldsymbol{\vartheta}-\boldsymbol{l_{0}}$ and $\boldsymbol{M}\left(
{\boldsymbol{\theta}}\right)  = \boldsymbol{M}\left(  {\boldsymbol{\vartheta}%
},{\sigma^{2}}\right)  =%
\begin{pmatrix}
\boldsymbol{L}^{T} & \boldsymbol{0}_{r}\\
\boldsymbol{0}_{p}^{T} & 0
\end{pmatrix}
$. Hence, the Wald-type test for this general linear hypothesis in
(\ref{EQ:9Gen_lin_hypothesis}) is given by
\begin{align}
W_{n}(\widehat{{\boldsymbol{\vartheta}}}_{\beta},\widehat{\sigma}_{\beta}%
^{2})=\frac{n}{\zeta_{\beta}^{3/2}\widehat{\sigma}_{\beta}^{2}}(\boldsymbol{L}%
^{T}\widehat{{\boldsymbol{\vartheta}}}_{\beta}-\boldsymbol{l_{0}})^{T}\left(
\boldsymbol{L}^{T}(\boldsymbol{X}^{T}\boldsymbol{X})^{-1}\boldsymbol{L}%
\right)  ^{-1}(\boldsymbol{L}^{T}\widehat{{\boldsymbol{\vartheta}}}_{\beta
}-\boldsymbol{l_{0}}), \label{EQ:Wald-test_LRM}%
\end{align}
which asymptotically follows $\chi_{r}^{2}$ distribution under the null
hypothesis.
Also, under the contiguous alternative $H_{1n}^{\ast}$ in (\ref{22.19}), given
by $\boldsymbol{L}^{T}{\boldsymbol{\vartheta}}=\boldsymbol{l_{0}}%
+n^{-1/2}\boldsymbol{\delta}$, the asymptotic distribution of the test
statistics $W_{n}(\widehat{{\boldsymbol{\theta}}}_{\beta},\widehat{\sigma
}_{\beta}^{2})$ is non-central chi-square with the non-centrality parameter
$\omega_{\beta}$ defined as
\[
\omega_{\beta}=\zeta_{\beta}^{-3/2}\sigma_{0}^{-2}\boldsymbol{\delta}%
^{T}\left(  \boldsymbol{L}^{T}(\boldsymbol{X}^{T}\boldsymbol{X})^{-1}%
\boldsymbol{L}\right)  ^{-1}\boldsymbol{\delta}.
\]

Now, let us derive the influence functions of the above Wald-type test
statistics. However, as noted in
\citet{ghosh2013,ghosh2015b}%
, in this case of non-homogeneous observations, the corresponding statistical
functional and the influence functions will depend on the sample size through
the given values of covariates $\boldsymbol{x}_{i}$'s. In particular, we need
to assume that the true distributions of each $y_{i}$ are (possibly)
different, say $H_{i}$ ($i=1,\ldots,n$), depending on the given values of
$x_{i}$. Then, the statistical functional corresponding to the Wald-type test
(\ref{EQ:Wald-test_LRM}) is given by
\[
W_{\beta}\left(  H_{1},\ldots,H_{n}\right)  =\zeta_{\beta}^{-3/2}\left(
\boldsymbol{L}^{T}\boldsymbol{T}_{\beta}^{\vartheta}(H_{1},\ldots
,H_{n})-\boldsymbol{l_{0}}\right)  ^{T}\frac{\left(  \boldsymbol{L}%
^{T}(\boldsymbol{X}^{T}\boldsymbol{X})^{-1}\boldsymbol{L}\right)  ^{-1}%
}{T_{\beta}^{\sigma}(H_{1},\ldots,H_{n})}\left(  \boldsymbol{L}^{T}%
\boldsymbol{T}_{\beta}^{\vartheta}(H_{1},\ldots,H_{n})-\boldsymbol{l_{0}%
}\right)  ,
\]
where $\boldsymbol{T}_{\beta}^{\vartheta}$ and $T_{\beta}^{\sigma}$ are the
statistical functionals corresponding to the MDPDEs
$\widehat{{\boldsymbol{\vartheta}}}_{\beta}$ and $\widehat{\sigma}_{\beta}%
^{2}$, as defined in
\citet{ghosh2013}%
. Since there are $n$ many different distributions, we can assume the
contamination in any one of these distributions or in all the distributions.
Corresponding influence functions of the MDPDEs are derived in
\citet{ghosh2013}%
. Using them and following the arguments used to proof Theorem \ref{Theorem1},
we get the influence functions of the proposed Wald type test. In particular,
at the null hypothesis, the first order influence function is zero for any
kind of contamination and the second order influence function at the null is
given by
\[
\mathcal{IF}_{2}(t_{i};W_{\beta},F_{\boldsymbol{\theta}_{0}})=2(1+\beta
)^{3}\zeta_{\beta}^{-3/2}\sigma_{0}^{-2}(t_{i}-\boldsymbol{x}_{i}%
^{T}\boldsymbol{\vartheta}_{0})^{2}\boldsymbol{x}_{i}^{T}\boldsymbol{D}%
\boldsymbol{x}_{i}e^{-\frac{\beta(t_{i}-\boldsymbol{x}_{i}^{T}%
\boldsymbol{\vartheta}_{0})^{2}}{\sigma_{0}^{2}}},~~~~~\boldsymbol{\theta}%
_{0}=(\boldsymbol{\vartheta}_{0}^{T},\sigma_{0}^{2})^{T},
\]
if the contamination is only in $i$-th direction at the point $t_{i}$, and
\[
\mathcal{IF}_{2}(t_{1},\ldots,t_{n};W_{\beta},F_{\boldsymbol{\theta}_{0}%
})=2(1+\beta)^{3}\zeta_{\beta}^{-3/2}\sigma_{0}^{-2}\sum_{i=1}^{n}%
(t_{i}-\boldsymbol{x}_{i}^{T}\boldsymbol{\vartheta}_{0})^{2}\boldsymbol{x}%
_{i}^{T}\boldsymbol{D}\boldsymbol{x}_{i}e^{-\frac{\beta(t_{i}-\boldsymbol{x}%
_{i}^{T}\boldsymbol{\vartheta}_{0})^{2}}{\sigma_{0}^{2}}},
\]
if there is contamination in all the directions at the points $t_{i}$'s. Here
\[
\boldsymbol{D}=(\boldsymbol{X}^{T}\boldsymbol{X})^{-1}\boldsymbol{L}\left(
\boldsymbol{L}^{T}(\boldsymbol{X}^{T}\boldsymbol{X})^{-1}\boldsymbol{L}%
\right)  ^{-1}\boldsymbol{L}^{T}(\boldsymbol{X}^{T}\boldsymbol{X})^{-1}.
\]

Next we consider the level and power influence functions of the proposed
Wald-type test. As in Section \ref{Sec4.2}, it follows that the level
influence function is always zero implying the level robustness of the
proposal. For power influence function, we again consider the alternatives
$H_{1n}^{\ast}:\boldsymbol{L}^{T}{\boldsymbol{\vartheta}}=\boldsymbol{l_{0}%
}+n^{-1/2}\boldsymbol{\delta}$ and proceed as in Section \ref{Sec4.2} to
obtain the PIF for different types of contamination. In particular, for
contamination only in the $i$-th direction at the point $t_{i}$ we get
\[
\mathcal{PIF}(t_{i};W_{\beta},F_{\boldsymbol{\theta}_{0}})=\frac{K_{r}^{\ast
}(\omega_{\beta})(1+\beta)^{3/2}}{\zeta_{\beta}^{3/2}\sigma_{0}^{2}}\left[
\boldsymbol{\delta}^{T}\boldsymbol{D}_{P}\boldsymbol{x_{i}}\right]
(t_{i}-\boldsymbol{x}_{i}^{T}\boldsymbol{\vartheta_{0}})e^{-\frac{\beta
(t_{i}-\boldsymbol{x}_{i}^{T}\boldsymbol{\vartheta_{0}})^{2}}{2\sigma_{0}^{2}%
}},
\]
where
\[
\boldsymbol{D}_{P}= \left(  \boldsymbol{L}^{T}(\boldsymbol{X}^{T}%
\boldsymbol{X})^{-1}\boldsymbol{L}\right)  ^{-1}\boldsymbol{L}^{T}%
(\boldsymbol{X}^{T}\boldsymbol{X})^{-1}
\]
Similarly, if the contamination is assumed to be in all the directions at the
points $t_{i}$s ($i=1,\ldots,n$), the corresponding power influence function
is given by
\[
\mathcal{PIF}(t_{1},\ldots,t_{n}; W_{\beta},F_{\boldsymbol{\theta}_{0}}%
)=\frac{K_{r}^{\ast}(\omega_{\beta})(1+\beta)^{3/2}}{\zeta_{\beta}^{3/2}%
\sigma_{0}^{2}}\boldsymbol{\delta}^{T}\boldsymbol{D}_{P}\sum_{i=1}%
^{n}\boldsymbol{x_{i}}(t_{i}-\boldsymbol{x}_{i}^{T}\boldsymbol{\vartheta_{0}%
})e^{-\frac{\beta(t_{i}-\boldsymbol{x}_{i} ^{T}\boldsymbol{\vartheta_{0}}%
)^{2}}{2\sigma_{0}^{2}}}.
\]

Clearly, the power influence function is bounded for all $\beta>0$ implying
robustness and unbounded at $\beta=0$ implying the non-robust nature of the
classical Wald test.

\begin{remark}
For the testing of significance of regression model ($H_{0}%
:\boldsymbol{\vartheta}=\boldsymbol{0}_{p}$) we have $r=p$, $\boldsymbol{l}%
_{0}= \boldsymbol{0}_{p}$ and $L=I_{p}$, the identity matrix of order $p$. In
this case the Wald-Type test statistic (\ref{EQ:Wald-test_LRM}) simplifies to
\[
W_{n}(\widehat{{\boldsymbol{\vartheta}}}_{\beta},\widehat{\sigma}_{\beta}%
^{2})=\frac{n}{\zeta_{\beta}^{3/2}\widehat{\sigma}_{\beta}^{2}}%
\widehat{{\boldsymbol{\vartheta}}}_{\beta}^{T}(\boldsymbol{X}^{T}%
\boldsymbol{X}) \widehat{{\boldsymbol{\vartheta}}}_{\beta},
\]
which is asymptotically $\chi_{p}^{2}$ under the null hypothesis. Under the
contiguous alternatives $H_{1n}^{\ast}$, its asymptotic distribution becomes
the non-central chi-square with $p$ degrees of freedom and non-centrality
parameter $\omega_{\beta}=\zeta_{\beta}^{-3/2}\sigma_{0}^{-2}
\boldsymbol{\delta}^{T}(\boldsymbol{X}^{T}\boldsymbol{X})\boldsymbol{\delta}$.
Noting that the asymptotic distribution under the contiguous alternatives
depends on the tuning parameter $\beta$ only through the quantity
$zeta_{\beta}$ and examining its form, one can easily check that the
asymptotic contiguous power of the proposed Wald-type tests decreases only
slightly with increasing values of $\beta$ so that the power loss under pure
data is not significant at small positive values of $\beta$.

On the other hand, under contamination we gain high robustness with these
positive values of $\beta$. For illustrations, we have presented (Figure
\ref{FIG:ex-3-IF}) the form of the second order influence function of the
tests and the power influence function for various values of $\beta$ under
contamination in one direction (say $i$-th). In this special case, they have
the simplified form (with $\boldsymbol{\vartheta}_{0}=\boldsymbol{0}_{p}$)
\[
\mathcal{IF}_{2}(t_{i};W_{\beta},F_{\boldsymbol{\theta}_{0}})=2(1+\beta
)^{3}\zeta_{\beta}^{-3/2}\sigma_{0}^{-2}\left[  \boldsymbol{x}_{i}%
^{T}(\boldsymbol{X}^{T}\boldsymbol{X})^{-1}\boldsymbol{x}_{i}\right]
t_{i}^{2}e^{-\frac{\beta t_{i}^{2}}{\sigma_{0}^{2}}},
\]
and
\[
\mathcal{PIF}(t_{i};W_{\beta},F_{\boldsymbol{\theta}_{0}})=\frac{K_{r}^{\ast
}(\omega_{\beta})(1+\beta)^{3/2}}{\zeta_{\beta}^{3/2}\sigma_{0}^{2}}\left[
\boldsymbol{\delta}^{T}\boldsymbol{x_{i}}\right]  t_{i}e^{-\frac{\beta
t_{i}^{2}}{2\sigma_{0}^{2}}}.
\]

\end{remark}

It is clear from the figure that the influence functions are bounded for all
$\beta>0$ and their maximum values decreases as $\beta$ increases implying the
increasing robustness.

\begin{figure}[h]
\centering
\subfloat[IF of the test statistics with $\boldsymbol{x}_{i}^{T}(\boldsymbol{X}^{T}\boldsymbol{X})^{-1}\boldsymbol{x}_{i}=1$]{
		\includegraphics[width=0.5\textwidth]{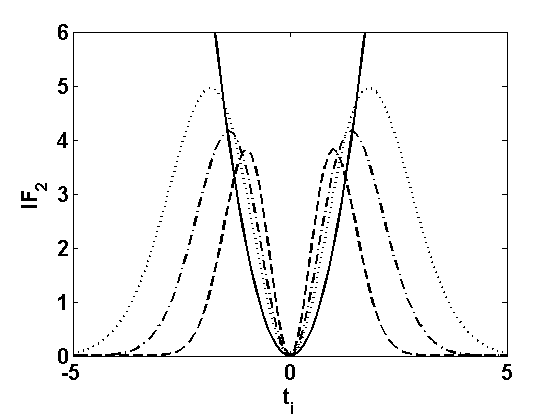}
		\label{FIG:ex-3-IF-test}} ~
\subfloat[PIF with $\boldsymbol{\delta}^{T}\boldsymbol{x_{i}}=1$ and $\boldsymbol{\delta}^{T}(\boldsymbol{X}^{T}\boldsymbol{X})\boldsymbol{\delta}=3$]{
		\includegraphics[width=0.5\textwidth]{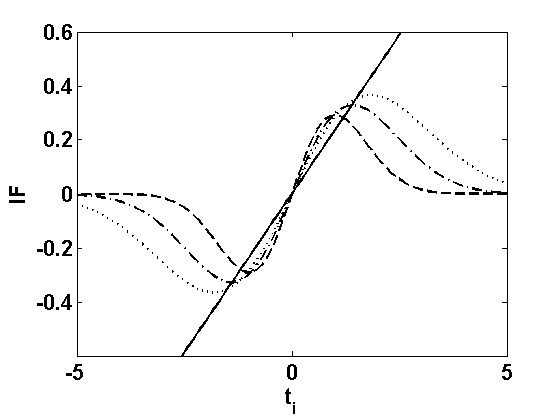}
		\label{FIG:ex-3-IF-power}}
\caption{Influence functions of MDPDE based Wald-type test of
(\ref{EQ:9Gen_lin_hypothesis}) with $\sigma_{0}=1$ for different values of
$\beta$ (solid line: $\beta=0$, dotted line: $\beta=0.3$, dashed-dotted line:
$\beta=0.5$, dashed line: $\beta=1$).}%
\label{FIG:ex-3-IF}%
\end{figure}

\pagebreak

\section{On the Choice of Tuning Parameter $\beta$\label{sec7}}

After deriving several important properties of the Wald-type test, a natural
question that arises from the point of view of a practitioner is what value of
the tuning parameter should be used for a particular dataset.
For the MDPDE the role of the tuning parameter $\beta$ has been well studied
in the literature, which indicates that robustness increases with $\beta$, but
efficiency decreases at the same time. So $\beta$ is selected that gives a
trade-off between robustness and efficiency of the estimator. However, a small
positive value of $\beta$ is generally recommended that provides enough
robustness with a slight loss in efficiency (see \citealp{MR1665873} and
\citealp{AyanBook}). \cite{broniatowski2012decomposable} have reported that
values of $\beta\in\lbrack0.1,0.25]$ are often reasonable choices. We largely
agree with this view, although tentative outliers and heavier contamination
may require a larger value of $\beta$ in some cases. Apart from a fixed choice
of the tuning parameter, one may dynamically select an optimum value of
$\beta$ based on the real data. \citet{Hong} and \citet{Warwick} have provided
some data driven choices of $\beta$ for the MDPDE. In case of hypothesis
testing the optimality criteria are different from the estimation case. Here
the asymptotic power against the contiguous alternative may be regarded as a
measure of efficiency of the test, which decreases with $\beta$. On the other
hand, the robustness of the test against contamination increases as $\beta$
increases. Therefore, our suggestion in this regard is to choose an optimum
value of $\beta$ that gives a suitable trade-off between the asymptotic power
against the contiguous alternative and a robustness measure, see
\cite{ghosh2015robust} for details. As the robustness of the Wald-type test
statistic depends primarily on the robustness of the estimators, another
simple criterion to choose an optimum value of $\beta$ is to focus on the same
optimum value for the estimator.

To avoid selecting a unique and specific tuning parameter, one may construct a
test combining a set of Wald-type tests corresponding to different $\beta$.
\cite{lavancier2014general} have derived a general procedure to combine a set
of estimators. This idea of constructing combined tests might be incorporated.

\section{Concluding Remarks\label{sec8}}%

\citet{2014arXiv1403.7616B}
have proposed the Wald-type test statistics based on the minimum density power
divergence estimators. They have observed strong robustness properties of the
tests by using extensive simulation results. In this paper we have given
proper theoretical foundations behind the robustness properties of the
Wald-type test statistics. The influence function analysis is carried out to
observe the effect of an infinitesimal contamination on the test statistics.
To justify the stability of the level and power under a contaminated
distribution we have studied the level and power influence functions. It is
shown that the level influence function of a Wald-type test statistic is zero,
so the level of the test remains unchanged in infinitesimal contamination. For
the contiguous alternative the power influence function is bounded whenever
the influence function of the MDPDE is bounded. Other than location-scale
parameters for the normal model we have shown some examples where the power
influence functions are bounded, and it gives the theoretical justification
behind the stability of the power function. On the other hand, the power
influence functions of the classical Wald tests are unbounded, and as a result
they exhibit poor power in contaminated data. We have also proposed the
chi-square inflation factor to measure the robustness property with respect to
the model assumption, and studied its infinitesimal change for the Wald-type
test statistics. On the whole, we hope that this research establishes that the
tests proposed by
\citet{2014arXiv1403.7616B}%
\ not only perform well in practise, but also have theoretically sound
robustness credentials.\bigskip\newline\textbf{Acknowledgements}: The authors
would like to acknowledge the comments of the three referess, since they
helped improving the paper.

\bibliographystyle{abbrvnat}
\bibliography{MandalRef,mypapers}

\appendix

\section{Appendix}

There is some overlap between the Lehmann and Basu et al. conditions. In the
following we present the consolidated set of conditions which are the useful
ones in our context.

\subsection{Lehmann and Basu et al. conditions}

\begin{itemize}
\item[(LB1)] The model distributions $F_{\boldsymbol{\theta}}$ of
$\boldsymbol{X}$ have common support, so that the set $\mathcal{X}%
=\{\boldsymbol{x}|f_{\boldsymbol{\theta}}(\boldsymbol{x})>0\}$ is independent
of $\boldsymbol{\theta}$. The true distribution $H$ is also supported on
$\mathcal{X}$, on which the corresponding density $h$ is greater than zero.

\item[(LB2)] There is an open subset of $\omega$ of the parameter space
$\Theta$, containing the best fitting parameter $\boldsymbol{\theta}_{0}$ such
that for almost all $\boldsymbol{x}\in\mathcal{X}$, and all
$\boldsymbol{\theta}\in\omega$, the density $f_{\boldsymbol{\theta}%
}(\boldsymbol{x})$ is three times differentiable with respect to
$\boldsymbol{\theta}$ and the third partial derivatives are continuous with
respect to $\boldsymbol{\theta}$.

\item[(LB3)] The integrals $\int f_{\boldsymbol{\theta}}^{1+\beta
}(\boldsymbol{x})d\boldsymbol{x}$ and $\int f_{\boldsymbol{\theta}}^{\beta
}(\boldsymbol{x})h(\boldsymbol{x})d\boldsymbol{x}$ can be differentiated three
times with respect to $\boldsymbol{\theta}$, and the derivatives can be taken
under the integral sign.

\item[(LB4)] The $p\times p$ matrix $\boldsymbol{J}_{\beta}(\boldsymbol{\theta
})$, defined in (\ref{22.9}), is positive definite.

\item[(LB5)] There exists a function $M_{jkl}(\boldsymbol{x})$ such that
$|\nabla_{jkl}V_{\boldsymbol{\theta}}(\boldsymbol{x})|\leq M_{jkl}%
(\boldsymbol{x})$ for all $\boldsymbol{\theta}\in\omega$, where $V_{\theta
}(\boldsymbol{x})=\int f_{\boldsymbol{\theta}}^{1+\beta}(\boldsymbol{y}%
)d\boldsymbol{y}-\left(  1+\tfrac{1}{\beta}\right)  f_{\boldsymbol{\theta}%
}^{\beta}(\boldsymbol{x})$ and $E_{h}[M_{jkl}(\boldsymbol{X})]=m_{jkl}<\infty$
for all $j$, $k$ and $l$.
\end{itemize}

\subsection{Proof of Theorem \textbf{\ref{Theorem1}\label{A1}}}

The second order influence function of $W_{\beta}^{0}(\cdot)$ is given by%
\[
\mathcal{IF}_{2}(\boldsymbol{x},W_{\beta}^{0},G)=\left.  \frac{\partial
^{2}W_{\beta}^{0}(G_{\varepsilon})}{\partial\varepsilon^{2}}\right\vert
_{\varepsilon=0},
\]
and%
\begin{align*}
\left.  \frac{\partial^{2}W_{\beta}^{0}(G_{\varepsilon})}{\partial
\varepsilon^{2}}\right\vert _{\varepsilon=0}  &  =2\mathcal{IF}^{T}%
(\boldsymbol{x},\boldsymbol{T}_{\beta},G)\boldsymbol{\Sigma}_{\beta}%
^{-1}(\boldsymbol{\theta}_{0})\mathcal{IF}(\boldsymbol{x},\boldsymbol{T}%
_{\beta},G)\\
&  +2(\boldsymbol{T}_{\beta}(G)-\boldsymbol{\theta}_{0})^{T}\boldsymbol{\Sigma
}_{\beta}^{-1}(\boldsymbol{\theta}_{0})\mathcal{IF}_{2}(\boldsymbol{x}%
,\boldsymbol{T}_{\beta},G).
\end{align*}
As $\boldsymbol{T}_{\beta}(F_{\boldsymbol{\theta}_{0}})=\boldsymbol{\theta
}_{0}$, we obtain%
\[
\mathcal{IF}_{2}(\boldsymbol{x},W_{\beta}^{0},F_{\boldsymbol{\theta}_{0}%
})=2\mathcal{IF}^{T}(\boldsymbol{x},\boldsymbol{T}_{\beta}%
,F_{\boldsymbol{\theta}_{0}})\boldsymbol{\Sigma}_{\beta}^{-1}%
(\boldsymbol{\theta}_{0})\mathcal{IF}(\boldsymbol{x},\boldsymbol{T}_{\beta
},F_{\boldsymbol{\theta}_{0}}).
\]
The second order influence function of (\ref{WF}) is given by%
\[
\mathcal{IF}_{2}(\boldsymbol{x},W_{\beta},G)=\left.  \frac{\partial
^{2}W_{\beta}(G_{\varepsilon})}{\partial\varepsilon^{2}}\right\vert
_{\varepsilon=0},
\]
and%
\begin{align*}
\left.  \frac{\partial^{2}W_{\beta}(G_{\varepsilon})}{\partial\varepsilon^{2}%
}\right\vert _{\varepsilon=0}  &  =2\mathcal{IF}^{T}(\boldsymbol{x}%
,\boldsymbol{T}_{\beta},G)\boldsymbol{M}(\boldsymbol{T}_{\beta}(G))\left(
\boldsymbol{M}^{T}(\boldsymbol{T}_{\beta}(G))\boldsymbol{\Sigma}_{\beta
}(\boldsymbol{T}_{\beta}(G))\boldsymbol{M}(\boldsymbol{T}_{\beta}(G))\right)
^{-1}\boldsymbol{M}^{T}(\boldsymbol{T}_{\beta}(G))\mathcal{IF}(\boldsymbol{x}%
,\boldsymbol{T}_{\beta},G)\\
&  +2\boldsymbol{m}^{T}(\boldsymbol{T}_{\beta}(G))\left.  \frac{\partial
}{\partial\varepsilon}\left(  \left(  \boldsymbol{M}^{T}(\boldsymbol{T}%
_{\beta}(G_{\varepsilon}))\boldsymbol{\Sigma}_{\beta}(\boldsymbol{T}_{\beta
}(G_{\varepsilon}))\boldsymbol{M}(\boldsymbol{T}_{\beta}(G_{\varepsilon
}))\right)  ^{-1}\boldsymbol{M}^{T}(\boldsymbol{T}_{\beta}(G_{\varepsilon
}))\mathcal{IF}(\boldsymbol{x},\boldsymbol{T}_{\beta},G_{\varepsilon})\right)
\right\vert _{\varepsilon=0}\\
&  +2\boldsymbol{m}^{T}(\boldsymbol{T}_{\beta}(G))\left.  \frac{\partial
}{\partial\varepsilon}\left(  \left(  \boldsymbol{M}^{T}(\boldsymbol{T}%
_{\beta}(G_{\varepsilon}))\boldsymbol{\Sigma}_{\beta}(\boldsymbol{T}_{\beta
}(G_{\varepsilon}))\boldsymbol{M}(\boldsymbol{T}_{\beta}(G_{\varepsilon
}))\right)  ^{-1}\right)  \right\vert _{\varepsilon=0}\boldsymbol{M}%
^{T}(\boldsymbol{T}_{\beta}(G))\mathcal{IF}(\boldsymbol{x},\boldsymbol{T}%
_{\beta},G)\\
&  +\boldsymbol{m}^{T}(\boldsymbol{T}_{\beta}(G))\left.  \frac{\partial^{2}%
}{\partial\varepsilon^{2}}\left(  \left(  \boldsymbol{M}^{T}(\boldsymbol{T}%
_{\beta}(G_{\varepsilon}))\boldsymbol{\Sigma}_{\beta}(\boldsymbol{T}_{\beta
}(G_{\varepsilon}))\boldsymbol{M}(\boldsymbol{T}_{\beta}(G_{\varepsilon
}))\right)  ^{-1}\right)  \right\vert _{\varepsilon=0}\boldsymbol{m}%
(\boldsymbol{T}_{\beta}(G))\\
&  =2\mathcal{IF}^{T}(\boldsymbol{x},\boldsymbol{T}_{\beta},G)\boldsymbol{M}%
(\boldsymbol{T}_{\beta}(G))\left[  \boldsymbol{M}^{T}(\boldsymbol{T}_{\beta
}(G))\boldsymbol{\Sigma}_{\beta}(\boldsymbol{T}_{\beta}(G))\boldsymbol{M}%
(\boldsymbol{T}_{\beta}(G))\right]  ^{-1}\boldsymbol{M}^{T}(\boldsymbol{T}%
_{\beta}(G))\mathcal{IF}(\boldsymbol{x},\boldsymbol{T}_{\beta},G)\\
&  =2\left(  \boldsymbol{u}_{\boldsymbol{\theta}}\left(  \boldsymbol{x}%
\right)  f_{\boldsymbol{\theta}_{0}}^{\beta}(\boldsymbol{x})-\boldsymbol{\xi
}\left(  \boldsymbol{\theta}_{0}\right)  \right)  ^{T}\boldsymbol{J}_{\beta
}^{-1}(\boldsymbol{\theta}_{0})\boldsymbol{M}(\boldsymbol{\theta}_{0})\left(
\boldsymbol{M}^{T}(\boldsymbol{\theta}_{0})\boldsymbol{\Sigma}_{\beta
}(\boldsymbol{\theta}_{0})\boldsymbol{M}(\boldsymbol{\theta}_{0})\right)
^{-1}\\
&  \times\boldsymbol{M}^{T}(\boldsymbol{\theta}_{0})\boldsymbol{J}_{\beta
}^{-1}(\boldsymbol{\theta}_{0})\left(  \boldsymbol{u}_{\boldsymbol{\theta}%
}\left(  \boldsymbol{x}\right)  f_{\boldsymbol{\theta}_{0}}^{\beta
}(\boldsymbol{x})-\boldsymbol{\xi}\left(  \boldsymbol{\theta}_{0}\right)
\right)  ,
\end{align*}
As $\boldsymbol{T}_{\beta}(F_{\boldsymbol{\theta}_{0}})=\boldsymbol{\theta}%
$$_{0}$, we obtain%
\begin{align*}
\mathcal{IF}_{2}(\boldsymbol{x},W_{\beta},F_{\boldsymbol{\theta}_{0}})  &
=2\mathcal{IF}^{T}(\boldsymbol{x},\boldsymbol{T}_{\beta},F_{\boldsymbol{\theta
}_{0}})\boldsymbol{M}(\boldsymbol{\theta}_{0})\left(  \boldsymbol{M}%
^{T}(\boldsymbol{\theta}_{0})\boldsymbol{\Sigma}_{\beta}(\boldsymbol{\theta
}_{0})\boldsymbol{M}(\boldsymbol{\theta}_{0})\right)  ^{-1}\\
&  \times\boldsymbol{M}^{T}(\boldsymbol{\theta}_{0})\mathcal{IF}%
(\boldsymbol{x},\boldsymbol{T}_{\beta},F_{\boldsymbol{\theta}_{0}}).
\end{align*}

\subsection{Proof of Theorem \textbf{\ref{THM:7asymp_power_one}\label{A2}}}

Let us denote the quadratic form of a symmetric matrix $\boldsymbol{A}%
_{p\times p}$ as $q_{\boldsymbol{A}}(\boldsymbol{z})=\boldsymbol{z}%
^{T}\boldsymbol{A}\boldsymbol{z}$. We shall frequently use the following
result that
\begin{equation}
q_{\boldsymbol{A}}(\boldsymbol{z}+\boldsymbol{h})=q_{\boldsymbol{A}%
}(\boldsymbol{z})+2\boldsymbol{h}^{T}\boldsymbol{A}\boldsymbol{z}%
+q_{\boldsymbol{A}}(\boldsymbol{h}), \label{lemma}%
\end{equation}
where $\boldsymbol{z}$ and $\boldsymbol{h}$ are two vectors in $\mathbb{R}%
^{p}$. Using $\boldsymbol{\theta}$$_{n}^{\ast}=\boldsymbol{T}_{\beta
}(F_{n,\varepsilon,\boldsymbol{x}}^{P})$ and equation (\ref{lemma}), with
$\boldsymbol{z=}\widehat{\boldsymbol{\theta}}_{\beta}-$$\boldsymbol{\theta}%
$$_{n}^{\ast}$ and $\boldsymbol{h=}$$\boldsymbol{\theta}$$_{n}^{\ast}%
-$$\boldsymbol{\theta}$$_{0}$, we get%
\begin{align*}
W_{n}^{0}(\widehat{\boldsymbol{\theta}}_{\beta})  &  =q_{n\boldsymbol{\Sigma
}_{\beta}^{-1}(\boldsymbol{\theta}_{0})}(\widehat{\boldsymbol{\theta}}_{\beta
}-\boldsymbol{\theta}_{0})=q_{n\boldsymbol{\Sigma}_{\beta}^{-1}%
(\boldsymbol{\theta}_{0})}\left(  (\widehat{\boldsymbol{\theta}}_{\beta
}-\boldsymbol{\theta}_{n}^{\ast})+(\boldsymbol{\theta}_{n}^{\ast
}-\boldsymbol{\theta}_{0})\right) \\
&  =q_{n\boldsymbol{\Sigma}_{\beta}^{-1}(\boldsymbol{\theta}_{0}%
)}(\widehat{\boldsymbol{\theta}}_{\beta}-\boldsymbol{\theta}_{n}^{\ast
})+2n(\widehat{\boldsymbol{\theta}}_{\beta}-\boldsymbol{\theta}_{n}^{\ast
})^{T}\boldsymbol{\Sigma}_{\beta}^{-1}(\boldsymbol{\theta}_{0}%
)(\boldsymbol{\theta}_{n}^{\ast}-\boldsymbol{\theta}_{0}%
)+q_{n\boldsymbol{\Sigma}_{\beta}^{-1}(\boldsymbol{\theta}_{0})}%
(\boldsymbol{\theta}_{n}^{\ast}-\boldsymbol{\theta}_{0}),
\end{align*}
i.e.,%
\begin{equation}
W_{n}^{0}(\widehat{\boldsymbol{\theta}}_{\beta})=W_{n}^{0}(\boldsymbol{\theta
}_{n}^{\ast})+q_{n\boldsymbol{\Sigma}_{\beta}^{-1}(\boldsymbol{\theta}_{0}%
)}(\widehat{\boldsymbol{\theta}}_{\beta}-\boldsymbol{\theta}_{n}^{\ast
})+2n(\widehat{\boldsymbol{\theta}}_{\beta}-\boldsymbol{\theta}_{n}^{\ast
})^{T}\boldsymbol{\Sigma}_{\beta}^{-1}(\boldsymbol{\theta}_{0}%
)(\boldsymbol{\theta}_{n}^{\ast}-\boldsymbol{\theta}_{0}). \label{W0}%
\end{equation}
Let us consider $\boldsymbol{\theta}$$_{n}^{\ast}$ as a function of
$\varepsilon_{n}=\varepsilon/\sqrt{n}$, i.e. $\boldsymbol{\theta}$$_{n}^{\ast
}=f(\varepsilon_{n})$. A Taylor series expansion of $f(\varepsilon_{n})$ at
$\varepsilon_{n}=0$ gives%
\begin{align*}
f(\varepsilon_{n})  &  =\sum_{k=0}^{\infty}\frac{1}{k!}\frac{\varepsilon^{k}%
}{n^{\frac{k}{2}}}\left.  \frac{\partial^{k}f(\varepsilon_{n})}{\partial
\varepsilon_{n}^{k}}\right\vert _{\varepsilon_{n}=0}\\
&  =\boldsymbol{\theta}_{n}+\tfrac{\varepsilon}{\sqrt{n}}\mathcal{IF}\left(
\boldsymbol{x},\boldsymbol{T}_{\beta},F_{\boldsymbol{\theta}_{n}}\right)
+\sum_{k=2}^{\infty}\tfrac{1}{k!}\left(  \tfrac{\varepsilon}{\sqrt{n}}\right)
^{k}\mathcal{IF}_{k}\left(  \boldsymbol{x},\boldsymbol{T}_{\beta
},F_{\boldsymbol{\theta}_{n}}\right)  .
\end{align*}
Therefore, we get%
\begin{align*}
\sqrt{n}(\boldsymbol{\theta}_{n}^{\ast}-\boldsymbol{\theta}_{n})  &
=\varepsilon\mathcal{IF}\left(  \boldsymbol{x},\boldsymbol{T}_{\beta
},F_{\boldsymbol{\theta}_{0}}\right)  +o_{p}(\boldsymbol{1}_{p}),\\
\sqrt{n}(\boldsymbol{\theta}_{n}^{\ast}-\boldsymbol{\theta}_{0}-n^{-1/2}%
\boldsymbol{d})  &  =\varepsilon\mathcal{IF}\left(  \boldsymbol{x}%
,\boldsymbol{T}_{\beta},F_{\boldsymbol{\theta}_{0}}\right)  +o_{p}%
(\boldsymbol{1}_{p}),
\end{align*}
and thus%
\begin{align}
\sqrt{n}(\boldsymbol{\theta}_{n}^{\ast}-\boldsymbol{\theta}_{0})  &
=\boldsymbol{d}+\varepsilon\mathcal{IF}\left(  \boldsymbol{x},\boldsymbol{T}%
_{\beta},F_{\boldsymbol{\theta}_{0}}\right)  +o_{p}(\boldsymbol{1}%
_{p})\nonumber\\
&  =\widetilde{\boldsymbol{d}}_{\varepsilon,\boldsymbol{x},\beta
}(\boldsymbol{\theta}_{0})+o_{p}(\boldsymbol{1}_{p}). \label{DT}%
\end{align}
So, in (\ref{W0}), both summands are given by%
\begin{align*}
W_{n}^{0}(\boldsymbol{\theta}_{n}^{\ast})  &  =\widetilde{\boldsymbol{d}%
}_{\varepsilon,\boldsymbol{x},\beta}^{T}(\boldsymbol{\theta}_{0}%
)\boldsymbol{\Sigma}_{\beta}^{-1}(\boldsymbol{\theta}_{0}%
)\widetilde{\boldsymbol{d}}_{\varepsilon,\boldsymbol{x},\beta}%
(\boldsymbol{\theta}_{0})+o_{p}(1),\\
2\sqrt{n}(\widehat{\boldsymbol{\theta}}_{\beta}-\boldsymbol{\theta}_{n}^{\ast
})^{T}\boldsymbol{\Sigma}_{\beta}^{-1}(\boldsymbol{\theta}_{0})\sqrt
{n}(\boldsymbol{\theta}_{n}^{\ast}-\boldsymbol{\theta}_{0})  &  =2\sqrt
{n}(\widehat{\boldsymbol{\theta}}_{\beta}-\boldsymbol{\theta}_{n}^{\ast}%
)^{T}\boldsymbol{\Sigma}_{\beta}^{-1}(\boldsymbol{\theta}_{0})\left(
\widetilde{\boldsymbol{d}}_{\varepsilon,\boldsymbol{x},\beta}%
(\boldsymbol{\theta}_{0})+o_{p}(\boldsymbol{1}_{p})\right)  .
\end{align*}
and hence according to the shape of (\ref{lemma}), (\ref{W0}) is equal to%
\[
W_{n}^{0}(\widehat{\boldsymbol{\theta}}_{\beta})=q_{\boldsymbol{\Sigma}%
_{\beta}^{-1}(\boldsymbol{\theta}_{0})}\left(  \sqrt{n}%
(\widehat{\boldsymbol{\theta}}_{\beta}-\boldsymbol{\theta}_{n}^{\ast
})+\widetilde{\boldsymbol{d}}_{\varepsilon,\boldsymbol{x},\beta}%
(\boldsymbol{\theta}_{0})\right)  +o_{p}(1).
\]
As%
\begin{equation}
\sqrt{n}(\widehat{\boldsymbol{\theta}}_{\beta}-\boldsymbol{\theta}_{n}^{\ast
})\underset{n\rightarrow\infty}{\overset{\mathcal{L}}{\longrightarrow}%
}\mathcal{N}(\boldsymbol{0}_{p},\boldsymbol{\Sigma}_{\beta}(\boldsymbol{\theta
}_{0})), \label{norm}%
\end{equation}
we get
\[
W_{n}^{0}(\widehat{\boldsymbol{\theta}}_{\beta})\underset{n\rightarrow
\infty}{\overset{\mathcal{L}}{\longrightarrow}}\chi_{p}^{2}\left(
\delta\right)  .
\]
with $\delta=\widetilde{\boldsymbol{d}}_{\varepsilon,\boldsymbol{x},\beta}%
^{T}(\boldsymbol{\theta}_{0})\boldsymbol{\Sigma}_{\beta}^{-1}%
(\boldsymbol{\theta}_{0})\widetilde{\boldsymbol{d}}_{\varepsilon
,\boldsymbol{x},\beta}(\boldsymbol{\theta}_{0}).$ This proves the first part
of the theorem.

Finally, the second part of the theorem follows from th infinite series
expansion of the non-central distribution function (and density) in terms of
that of the central chi-square variables;
\begin{align}
\beta_{W_{n}^{0}}(\boldsymbol{\theta}_{n},\varepsilon,\boldsymbol{x})  &
=\lim_{n\rightarrow\infty}P_{F_{n,\varepsilon,\boldsymbol{x}}^{P}}(W_{n}%
^{0}({\widehat{\boldsymbol{\theta}}_{\beta}})>\chi_{p,\alpha}^{2})\nonumber\\
&  \cong P(\chi_{p}^{2}\left(  \delta\right)  >\chi_{p,\alpha}^{2}%
)=1-F_{\chi_{p}^{2}\left(  \delta\right)  }\left(  \chi_{p,\alpha}^{2}\right)
\nonumber\\
&  =\sum\limits_{v=0}^{\infty}C_{v}\left(  \widetilde{\boldsymbol{d}%
}_{\varepsilon,\boldsymbol{x},\beta}(\boldsymbol{\theta}_{0}%
),\boldsymbol{\Sigma}_{\beta}^{-1}(\boldsymbol{\theta}_{0})\right)  P\left(
\chi_{p+2v}^{2}>\chi_{p,\alpha}^{2}\right)  .\nonumber
\end{align}

\subsection{Proof of Theorem \textbf{\ref{Theorem10}\label{A3}}}

Let us consider the expression of $\beta_{W_{n}^{0}}($$\boldsymbol{\theta}%
$$_{n},\varepsilon,\boldsymbol{x})$ as obtained in Theorem
\ref{THM:7asymp_power_one}. Note that, by definition
\begin{align*}
\mathcal{PIF}(\boldsymbol{x},W_{\beta}^{0},F_{\boldsymbol{\theta}_{0}})  &
=\frac{\partial}{\partial\varepsilon}\left.  \beta_{W_{n}^{0}}%
(\boldsymbol{\theta}_{n},\varepsilon,\boldsymbol{x})\right\vert _{\varepsilon
=0}\\
&  \cong\sum\limits_{v=0}^{\infty}\frac{\partial}{\partial\varepsilon}\left.
C_{v}\left(  \widetilde{\boldsymbol{d}}_{\varepsilon,\boldsymbol{x},\beta
}(\boldsymbol{\theta}_{0}),\boldsymbol{\Sigma}_{\beta}^{-1}(\boldsymbol{\theta
}_{0})\right)  \right\vert _{\varepsilon=0}P\left(  \chi_{p+2v}^{2}%
>\chi_{p,\alpha}^{2}\right) \\
&  \cong\sum\limits_{v=0}^{\infty}\left\{  \frac{\partial}{\partial\mathbf{t}%
}\left.  C_{v}\left(  \mathbf{t},\boldsymbol{\Sigma}_{\beta}^{-1}%
(\boldsymbol{\theta}_{0})\right)  \right\vert _{\mathbf{t}%
=\widetilde{\boldsymbol{d}}_{0,\boldsymbol{x},\beta}(\boldsymbol{\theta}_{0}%
)}\right\}  ^{T}\left\{  \frac{\partial}{\partial\varepsilon}\left.
\widetilde{\boldsymbol{d}}_{\varepsilon,\boldsymbol{x},\beta}%
(\boldsymbol{\theta}_{0})\right\vert _{\varepsilon=0}\right\}  P\left(
\chi_{p+2v}^{2}>\chi_{p,\alpha}^{2}\right)  ,
\end{align*}
where the last step follows from the chain rule. But
$\widetilde{\boldsymbol{d}}_{0,\boldsymbol{x},\beta}(\boldsymbol{\theta}%
_{0})=\boldsymbol{d}$ and routine differentiations yield
\[
\frac{\partial}{\partial\varepsilon}\widetilde{\boldsymbol{d}}_{\varepsilon
,\boldsymbol{x},\beta}(\boldsymbol{\theta}_{0})=\mathcal{IF}(\boldsymbol{x}%
,\boldsymbol{T}_{\beta},F_{\boldsymbol{\theta}_{0}}),
\]
and
\[
\frac{\partial}{\partial\mathbf{t}}C_{v}\left(  \mathbf{t},\mathbf{A}\right)
=\frac{\left(  \mathbf{t}^{T}\mathbf{A}\mathbf{t}\right)  ^{v-1}}{v!2^{v}%
}\left(  2v-\mathbf{t}^{T}\mathbf{A}\mathbf{t}\right)  \mathbf{A}%
\mathbf{t}e^{-\frac{1}{2}\mathbf{t}^{T}\mathbf{A}\mathbf{t}}.
\]
Combining these and simplifying, we get the theorem.

\subsection{Proof of Theorem \textbf{\ref{THM:7asymp_power_composite}%
\label{A4}}}

Let us denote $\boldsymbol{\theta}$$_{n}^{\ast}=\boldsymbol{T}_{\beta
}(F_{n,\varepsilon,\boldsymbol{x}}^{P})$. Using equation (\ref{lemma}), with
$\boldsymbol{z=m}(\widehat{\boldsymbol{\theta}}_{\beta})-\boldsymbol{m}%
($$\boldsymbol{\theta}$$_{n}^{\ast})$ and $\boldsymbol{h=m}($%
$\boldsymbol{\theta}$$_{n}^{\ast})$, we get%
\begin{align*}
W_{n}(\widehat{\boldsymbol{\theta}}_{\beta})  &  =q_{\boldsymbol{\Sigma
}_{\beta}^{\ast-1}(\widehat{\boldsymbol{\theta}}_{\beta})}(\sqrt
{n}\boldsymbol{m}(\widehat{\boldsymbol{\theta}}_{\beta}%
))=q_{\boldsymbol{\Sigma}_{\beta}^{\ast-1}(\widehat{\boldsymbol{\theta}%
}_{\beta})}\left(  \sqrt{n}(\boldsymbol{m}(\widehat{\boldsymbol{\theta}%
}_{\beta})-\boldsymbol{m}(\boldsymbol{\theta}_{n}^{\ast}))+\sqrt
{n}\boldsymbol{m}(\boldsymbol{\theta}_{n}^{\ast})\right) \\
&  =q_{\boldsymbol{\Sigma}_{\beta}^{\ast-1}(\widehat{\boldsymbol{\theta}%
}_{\beta})}\left(  \sqrt{n}(\boldsymbol{m}(\widehat{\boldsymbol{\theta}%
}_{\beta})-\boldsymbol{m}(\boldsymbol{\theta}_{n}^{\ast}))\right)  +2n\left(
\boldsymbol{m}(\widehat{\boldsymbol{\theta}}_{\beta})-\boldsymbol{m}%
(\boldsymbol{\theta}_{n}^{\ast})\right)  ^{T}\boldsymbol{\Sigma}_{\beta}%
^{\ast-1}(\widehat{\boldsymbol{\theta}}_{\beta})\boldsymbol{m}%
(\boldsymbol{\theta}_{n}^{\ast})+q_{\boldsymbol{\Sigma}_{\beta}^{\ast
-1}(\widehat{\boldsymbol{\theta}}_{\beta})}\left(  \sqrt{n}\boldsymbol{m}%
(\boldsymbol{\theta}_{n}^{\ast})\right)  ,
\end{align*}
where $\boldsymbol{\Sigma}_{\beta}^{\ast}(\boldsymbol{\theta}_{0}%
)=\boldsymbol{M}^{T}(\boldsymbol{\theta}_{0})\boldsymbol{\Sigma}_{\beta
}(\boldsymbol{\theta}_{0})\boldsymbol{M}(\boldsymbol{\theta}_{0})$., i.e.,%
\begin{equation}
W_{n}(\widehat{\boldsymbol{\theta}}_{\beta})=W_{n}(\boldsymbol{\theta}%
_{n}^{\ast})+q_{\boldsymbol{\Sigma}_{\beta}^{\ast-1}%
(\widehat{\boldsymbol{\theta}}_{\beta})}\left(  \sqrt{n}(\boldsymbol{m}%
(\widehat{\boldsymbol{\theta}}_{\beta})-\boldsymbol{m}(\boldsymbol{\theta}%
_{n}^{\ast}))\right)  +2n\left(  \boldsymbol{m}(\widehat{\boldsymbol{\theta}%
}_{\beta})-\boldsymbol{m}(\boldsymbol{\theta}_{n}^{\ast})\right)
^{T}\boldsymbol{\Sigma}_{\beta}^{\ast-1}(\boldsymbol{\theta}_{0}%
)\boldsymbol{m}(\boldsymbol{\theta}_{n}^{\ast}). \label{Ceq2}%
\end{equation}
Now, as in the proof of Theorem \ref{THM:7asymp_power_one}, we can show that
\begin{align}
\sqrt{n}(\boldsymbol{\theta}_{n}^{\ast}-\boldsymbol{\theta}_{0})  &
=\boldsymbol{d}+\varepsilon\mathcal{IF}\left(  \boldsymbol{x},\boldsymbol{T}%
_{\beta},F_{\boldsymbol{\theta}_{0}}\right)  +o_{p}(\boldsymbol{1}%
_{p})\nonumber\\
&  =\widetilde{\boldsymbol{d}}_{\varepsilon,\boldsymbol{x},\beta
}(\boldsymbol{\theta}_{0})+o_{p}(\boldsymbol{1}_{p}). \label{CDT}%
\end{align}
Using a Taylor series expansion, we get
\begin{equation}
\boldsymbol{m}(\boldsymbol{\theta}_{n}^{\ast})=\boldsymbol{m}%
(\boldsymbol{\theta}_{0})+\boldsymbol{M}^{T}(\boldsymbol{\theta}_{0})\left(
\boldsymbol{\theta}_{n}^{\ast}-\boldsymbol{\theta}_{0}\right)  +o\left(
||\boldsymbol{\theta}_{n}^{\ast}-\boldsymbol{\theta}_{0}||\right)  .
\label{CTylor}%
\end{equation}
As $\boldsymbol{m}(\boldsymbol{\theta}$$_{0})=\boldsymbol{0}_{r}$, from
(\ref{CDT}) it follows that%
\[
\sqrt{n}\boldsymbol{m}(\boldsymbol{\theta}_{n}^{\ast})=\boldsymbol{M}%
^{T}(\boldsymbol{\theta}_{0})\widetilde{\boldsymbol{d}}_{\varepsilon
,\boldsymbol{x},\beta}(\boldsymbol{\theta}_{0})+o_{p}(\boldsymbol{1}_{r}).
\]
Further, since (\ref{norm}) holds, a similar Taylor series expansion of
(\ref{CTylor}) yields
\begin{equation}
\sqrt{n}\left(  \boldsymbol{m}(\widehat{\boldsymbol{\theta}}_{\beta
})-\boldsymbol{m}(\boldsymbol{\theta}_{n}^{\ast})\right)
\underset{n\rightarrow\infty}{\overset{\mathcal{L}}{\longrightarrow}%
}\mathcal{N}(\boldsymbol{0}_{r},\boldsymbol{\Sigma}_{\beta}^{\ast
}(\boldsymbol{\theta}_{0})) \label{Cnorm2}%
\end{equation}
and%
\[
\sqrt{n}\boldsymbol{\Sigma}_{\beta}^{\ast-\frac{1}{2}}%
(\widehat{\boldsymbol{\theta}}_{\beta})\left(  \boldsymbol{m}%
(\widehat{\boldsymbol{\theta}}_{\beta})-\boldsymbol{m}(\boldsymbol{\theta}%
_{n}^{\ast})\right)  \underset{n\rightarrow\infty}{\overset{\mathcal{L}%
}{\longrightarrow}}\mathcal{N}(\boldsymbol{0}_{r},\boldsymbol{I}_{p}).
\]
Thus, we get
\[
q_{\boldsymbol{\Sigma}_{\beta}^{\ast-1}(\widehat{\boldsymbol{\theta}}_{\beta
})}\left(  \sqrt{n}(\boldsymbol{m}(\widehat{\boldsymbol{\theta}}_{\beta
})-\boldsymbol{m}(\boldsymbol{\theta}_{n}^{\ast}))\right)
\underset{n\rightarrow\infty}{\overset{\mathcal{L}}{\longrightarrow}}\chi
_{r}^{2}.
\]
Also, from (\ref{CDT}) we have%
\begin{align*}
W_{n}(\boldsymbol{\theta}_{n}^{\ast})  &  =\widetilde{\boldsymbol{d}%
}_{\varepsilon,\boldsymbol{x},\beta}^{T}(\boldsymbol{\theta}_{0}%
)\boldsymbol{M}(\boldsymbol{\theta}_{0})\boldsymbol{\Sigma}_{\beta}^{\ast
-1}(\widehat{\boldsymbol{\theta}}_{\beta})\boldsymbol{M}^{T}%
(\boldsymbol{\theta}_{0})\widetilde{\boldsymbol{d}}_{\varepsilon
,\boldsymbol{x},\beta}(\boldsymbol{\theta}_{0})+o_{p}(1)\\
&  =\widetilde{\boldsymbol{d}}_{\varepsilon,\boldsymbol{x},\beta}%
^{T}(\boldsymbol{\theta}_{0})\boldsymbol{M}(\boldsymbol{\theta}_{0}%
)\boldsymbol{\Sigma}_{\beta}^{\ast-1}(\boldsymbol{\theta}_{0})\boldsymbol{M}%
^{T}(\boldsymbol{\theta}_{0})\widetilde{\boldsymbol{d}}_{\varepsilon
,\boldsymbol{x},\beta}(\boldsymbol{\theta}_{0})+o_{p}(1),\\
2\sqrt{n}\left(  \boldsymbol{m}(\widehat{\boldsymbol{\theta}}_{\beta
})-\boldsymbol{m}(\boldsymbol{\theta}_{n}^{\ast})\right)  ^{T}%
\boldsymbol{\Sigma}_{\beta}^{\ast-1}(\widehat{\boldsymbol{\theta}}_{\beta
})\sqrt{n}\boldsymbol{m}(\boldsymbol{\theta}_{n}^{\ast})  &  =2\sqrt{n}\left(
\boldsymbol{m}(\widehat{\boldsymbol{\theta}}_{\beta})-\boldsymbol{m}%
(\boldsymbol{\theta}_{n}^{\ast})\right)  ^{T}\boldsymbol{\Sigma}_{\beta}%
^{\ast-1}(\widehat{\boldsymbol{\theta}}_{\beta})\boldsymbol{M}^{T}%
(\boldsymbol{\theta}_{0})\left(  \widetilde{\boldsymbol{d}}_{\varepsilon
,\boldsymbol{x},\beta}(\boldsymbol{\theta}_{0})+o_{p}(\boldsymbol{1}%
_{r})\right) \\
&  =2\sqrt{n}\left(  \boldsymbol{m}(\widehat{\boldsymbol{\theta}}_{\beta
})-\boldsymbol{m}(\boldsymbol{\theta}_{n}^{\ast})\right)  ^{T}%
\boldsymbol{\Sigma}_{\beta}^{\ast-1}(\boldsymbol{\theta}_{0})\boldsymbol{M}%
^{T}(\boldsymbol{\theta}_{0})\widetilde{\boldsymbol{d}}_{\varepsilon
,\boldsymbol{x},\beta}(\boldsymbol{\theta}_{0})+o_{p}(\boldsymbol{1}_{p}).
\end{align*}
Hence
\[
W_{n}(\widehat{\boldsymbol{\theta}}_{\beta})=q_{n\boldsymbol{\Sigma}_{\beta
}^{\ast-1}(\boldsymbol{\theta}_{0})}\left(  \left[  \boldsymbol{m}%
(\widehat{\boldsymbol{\theta}}_{\beta})-\boldsymbol{m}(\boldsymbol{\theta}%
_{n}^{\ast})\right]  +\frac{1}{\sqrt{n}}\boldsymbol{M}^{T}(\boldsymbol{\theta
}_{0})\widetilde{\boldsymbol{d}}_{\varepsilon,\boldsymbol{x},\beta
}(\boldsymbol{\theta}_{0})\right)  +o_{p}(1).
\]
As it holds (\ref{norm}), we get
\[
W_{n}(\widehat{\boldsymbol{\theta}}_{\beta})\underset{n\rightarrow
\infty}{\overset{\mathcal{L}}{\longrightarrow}}\chi_{r}^{2}(\delta),
\]
the non-central chi-square distribution with degrees of freedom $r$ and
non-centrality parameter $\delta=\widetilde{\boldsymbol{d}}_{\varepsilon
,\boldsymbol{x},\beta}^{T}(\boldsymbol{\theta}_{0})\boldsymbol{M}%
(\boldsymbol{\theta}_{0})\boldsymbol{\Sigma}_{\beta}^{\ast-1}%
(\boldsymbol{\theta}_{0})\boldsymbol{M}^{T}(\boldsymbol{\theta}_{0}%
)\widetilde{\boldsymbol{d}}_{\varepsilon,\boldsymbol{x},\beta}%
(\boldsymbol{\theta}_{0})$. This proves the first part of the theorem.

Second part of the theorem follows from above using the infinite series
expansion of the non-central distribution function (and density) in terms of
that of the central chi-square variables:
\begin{align}
\beta_{W_{n}}(\boldsymbol{\theta}_{n},\varepsilon,\boldsymbol{x})  &
=\lim_{n\rightarrow\infty}P_{F_{n,\varepsilon,\boldsymbol{x}}^{P}}(
W_{n}({\widehat{\boldsymbol{\theta}}_{\beta}})>\chi_{r,\alpha}^{2})\nonumber\\
&  \cong P(\chi_{r, \delta}^{2} >\chi_{r,\alpha}^{2})= 1 - F_{\chi_{r}
^{2}(\delta)}(\chi_{r,\alpha}^{2})\nonumber\\
&  =\sum\limits_{v=0}^{\infty}C_{v}\left(  \boldsymbol{M}^{T}
(\boldsymbol{\theta}_{0}) \widetilde{\boldsymbol{d}}_{\varepsilon
,\boldsymbol{x},\beta}(\boldsymbol{\theta}_{0}), \boldsymbol{\Sigma}_{\beta
}^{\ast-1}(\boldsymbol{\theta}_{0})\right)  P\left(  \chi_{r+2v}^{2} >
\chi_{r,\alpha}^{2}\right)  .\nonumber
\end{align}

\subsection{Proof of Theorem \textbf{\ref{THM:PIF_composite}\label{A5}}}

The proof is similar to that of Theorem \ref{Theorem10}, considering the
expression of $\beta_{W_{n}}($$\boldsymbol{\theta}$$_{n},\varepsilon
,\boldsymbol{x})$ from Theorem \ref{THM:7asymp_power_composite}. We omit the
detailed calculation for brevity.

\subsection{Proof of Theorem \textbf{\ref{THM:7Chi_infl_fact_slope}\label{A6}%
}}

Let us denote $\boldsymbol{J}_{\beta,g}(\boldsymbol{\theta})$, $\boldsymbol{K}%
_{\beta,g}(\boldsymbol{\theta})$, $\boldsymbol{\xi}_{\beta,g}%
({{\boldsymbol{\theta}}})$, $\boldsymbol{\Sigma}_{\beta,g}(\boldsymbol{\theta
})$ as $\boldsymbol{J}_{\beta,\varepsilon,\boldsymbol{y}}(\boldsymbol{\theta
})$, $\boldsymbol{K}_{\beta,\varepsilon,\boldsymbol{y}}(\boldsymbol{\theta})$,
$\boldsymbol{\xi}_{\beta,\varepsilon,\boldsymbol{y}}({{\boldsymbol{\theta}}}%
)$, $\boldsymbol{\Sigma}_{\beta,\varepsilon,\boldsymbol{y}}(\boldsymbol{\theta
})$ respectively, when $g=f_{\varepsilon,\boldsymbol{y}}$. The infinitesimal
change in the CSIF at the model is given by%
\[
\frac{\partial}{\partial\varepsilon}\bar{c}_{\beta,\varepsilon,\boldsymbol{y}%
}(\boldsymbol{\theta})=\frac{1}{p}\mathrm{trace}\left(  \boldsymbol{\Sigma
}_{\beta}^{-1}(\boldsymbol{\theta})\frac{\partial}{\partial\varepsilon}\left.
\boldsymbol{\Sigma}_{\beta,\varepsilon,\boldsymbol{y}}(\boldsymbol{\theta
})\right\vert _{\varepsilon=0}\right)  .
\]
Now%
\begin{align}
\frac{\partial}{\partial\varepsilon}\boldsymbol{\Sigma}_{\beta,\varepsilon
,\boldsymbol{y}}(\boldsymbol{\theta})  &  =\frac{\partial}{\partial
\varepsilon}\boldsymbol{J}_{\beta,\varepsilon,\boldsymbol{y}}^{-1}%
(\boldsymbol{\theta})\boldsymbol{K}_{\beta,\varepsilon,\boldsymbol{y}%
}(\boldsymbol{\theta})\boldsymbol{J}_{\beta,\varepsilon,\boldsymbol{y}}%
^{-1}(\boldsymbol{\theta})+\boldsymbol{J}_{\beta,\varepsilon,\boldsymbol{y}%
}^{-1}(\boldsymbol{\theta})\frac{\partial}{\partial\varepsilon}\boldsymbol{K}%
_{\beta,\varepsilon,\boldsymbol{y}}(\boldsymbol{\theta})\boldsymbol{J}%
_{\beta,\varepsilon,\boldsymbol{y}}^{-1}(\boldsymbol{\theta})\nonumber\\
&  +\boldsymbol{J}_{\beta,\varepsilon,\boldsymbol{y}}^{-1}(\boldsymbol{\theta
})\boldsymbol{K}_{\beta,\varepsilon,\boldsymbol{y}}(\boldsymbol{\theta}%
)\frac{\partial}{\partial\varepsilon}\boldsymbol{J}_{\beta,\varepsilon
,\boldsymbol{y}}^{-1}(\boldsymbol{\theta})\nonumber\\
&  =-\boldsymbol{J}_{\beta,\varepsilon,\boldsymbol{y}}^{-1}(\boldsymbol{\theta
})\frac{\partial}{\partial\varepsilon}\boldsymbol{J}_{\beta,\varepsilon
,\boldsymbol{y}}(\boldsymbol{\theta})\boldsymbol{\Sigma}_{\beta,\varepsilon
,\boldsymbol{y}}(\boldsymbol{\theta})+\boldsymbol{J}_{\beta,\varepsilon
,\boldsymbol{y}}^{-1}(\boldsymbol{\theta})\frac{\partial}{\partial\varepsilon
}\boldsymbol{K}_{\beta,\varepsilon,\boldsymbol{y}}(\boldsymbol{\theta
})\boldsymbol{J}_{\beta,\varepsilon,\boldsymbol{y}}^{-1}(\boldsymbol{\theta
})\nonumber\\
&  -\left(  \boldsymbol{J}_{\beta,\varepsilon,\boldsymbol{y}}^{-1}%
(\boldsymbol{\theta})\frac{\partial}{\partial\varepsilon}\boldsymbol{J}%
_{\beta,\varepsilon,\boldsymbol{y}}(\boldsymbol{\theta})\boldsymbol{\Sigma
}_{\beta,\varepsilon,\boldsymbol{y}}(\boldsymbol{\theta})\right)  ^{T},
\label{sigma}%
\end{align}
where%
\begin{align}
\frac{\partial}{\partial\varepsilon}\boldsymbol{J}_{\beta,\varepsilon
,\boldsymbol{y}}(\boldsymbol{\theta})  &  =\int\left(  \boldsymbol{I}%
_{\boldsymbol{\theta}}(\boldsymbol{x})-\beta\boldsymbol{u}%
_{{\boldsymbol{\theta}}}(\boldsymbol{x})\boldsymbol{u}_{{\boldsymbol{\theta}}%
}^{T}(\boldsymbol{x})\right)  \left(  \Delta_{\boldsymbol{y}}%
-f_{\boldsymbol{\theta}}(\boldsymbol{x})\right)  f_{{\boldsymbol{\theta}}%
}^{\beta}(\boldsymbol{x})d\boldsymbol{x}\nonumber\\
&  =f_{\boldsymbol{\theta}_{0}}^{\beta}(\boldsymbol{y})\left(  \boldsymbol{I}%
_{\boldsymbol{\theta}}(\boldsymbol{y})-\beta\boldsymbol{u}%
_{{\boldsymbol{\theta}}}(\boldsymbol{y})\boldsymbol{u}_{{\boldsymbol{\theta}}%
}^{T}(\boldsymbol{y})\right)  -\int\left(  \boldsymbol{I}_{\boldsymbol{\theta
}}(\boldsymbol{x})-\beta\boldsymbol{u}_{{\boldsymbol{\theta}}}(\boldsymbol{x}%
)\boldsymbol{u}_{{\boldsymbol{\theta}}}^{T}(\boldsymbol{x})\right)
f_{{\boldsymbol{\theta}}}^{1+\beta}(\boldsymbol{x})d\boldsymbol{x}\nonumber\\
&  =\beta\boldsymbol{J}_{\beta}(\boldsymbol{\theta})+f_{\boldsymbol{\theta
}_{0}}^{\beta}(\boldsymbol{y})\left(  \boldsymbol{I}_{\boldsymbol{\theta}%
}(\boldsymbol{y})-\beta\boldsymbol{u}_{{\boldsymbol{\theta}}}(\boldsymbol{y}%
)\boldsymbol{u}_{{\boldsymbol{\theta}}}^{T}(\boldsymbol{y})\right)
-\int\boldsymbol{I}_{\boldsymbol{\theta}}(\boldsymbol{x}%
)f_{{\boldsymbol{\theta}}}^{1+\beta}(\boldsymbol{x})d\boldsymbol{x}, \label{j}%
\end{align}
and%
\begin{align}
\frac{\partial}{\partial\varepsilon}\boldsymbol{K}_{\beta,\varepsilon
,\boldsymbol{y}}(\boldsymbol{\theta})  &  =\int\boldsymbol{u}%
_{{\boldsymbol{\theta}}}(\boldsymbol{x})\boldsymbol{u}_{{\boldsymbol{\theta}}%
}^{T}(\boldsymbol{x})f_{{\boldsymbol{\theta}}}^{2\beta}(\boldsymbol{x})\left(
\Delta_{\boldsymbol{y}}-f_{\boldsymbol{\theta}}(\boldsymbol{x})\right)
d\boldsymbol{x}-\frac{\partial}{\partial\varepsilon}\boldsymbol{\xi}%
_{\beta,\varepsilon,\boldsymbol{y}}({{\boldsymbol{\theta}}})\boldsymbol{\xi
}_{\beta,\varepsilon,\boldsymbol{y}}^{T}({{\boldsymbol{\theta}}}%
)-\boldsymbol{\xi}_{\beta,\varepsilon,\boldsymbol{y}}({{\boldsymbol{\theta}}%
})\frac{\partial}{\partial\varepsilon}\boldsymbol{\xi}_{\beta,\varepsilon
,\boldsymbol{y}}^{T}({{\boldsymbol{\theta}}})\nonumber\\
&  =\boldsymbol{u}_{{\boldsymbol{\theta}}}(\boldsymbol{y})\boldsymbol{u}%
_{{\boldsymbol{\theta}}}^{T}(\boldsymbol{y})f_{{\boldsymbol{\theta}}}^{2\beta
}(\boldsymbol{y})-\int\boldsymbol{u}_{{\boldsymbol{\theta}}}(\boldsymbol{x}%
)\boldsymbol{u}_{{\boldsymbol{\theta}}}^{T}(\boldsymbol{x}%
)f_{{\boldsymbol{\theta}}}^{2\beta+1}(\boldsymbol{x})d\boldsymbol{x}%
-\boldsymbol{\xi}_{\beta,\varepsilon,\boldsymbol{y}}({{\boldsymbol{\theta}}%
})\frac{\partial}{\partial\varepsilon}\boldsymbol{\xi}_{\beta,\varepsilon
,\boldsymbol{y}}^{T}({{\boldsymbol{\theta}}})-\left(  \boldsymbol{\xi}%
_{\beta,\varepsilon,\boldsymbol{y}}({{\boldsymbol{\theta}}})\frac{\partial
}{\partial\varepsilon}\boldsymbol{\xi}_{\beta,\varepsilon,\boldsymbol{y}}%
^{T}({{\boldsymbol{\theta}}})\right)  ^{T}. \label{k}%
\end{align}
Since%
\begin{align*}
\frac{\partial}{\partial\varepsilon}\boldsymbol{\xi}_{\beta,\varepsilon
,\boldsymbol{y}}({{\boldsymbol{\theta}}})  &  =\int\boldsymbol{u}%
_{{\boldsymbol{\theta}}}(\boldsymbol{x})f_{{\boldsymbol{\theta}}}^{\beta
}(\boldsymbol{x})\left(  \Delta_{\boldsymbol{y}}-f_{\boldsymbol{\theta}%
}(\boldsymbol{x})\right)  d\boldsymbol{x}=\boldsymbol{u}_{{\boldsymbol{\theta
}}}(\boldsymbol{y})f_{{\boldsymbol{\theta}}}^{\beta}(\boldsymbol{y}%
)-\int\boldsymbol{u}_{{\boldsymbol{\theta}}}(\boldsymbol{x}%
)f_{{\boldsymbol{\theta}}}^{1+\beta}(\boldsymbol{x})d\boldsymbol{x}\\
&  =\boldsymbol{u}_{{\boldsymbol{\theta}}}(\boldsymbol{y}%
)f_{{\boldsymbol{\theta}}}^{\beta}(\boldsymbol{y})-\boldsymbol{\xi}_{\beta
}({{\boldsymbol{\theta}}}),\\
\boldsymbol{\xi}_{\beta,\varepsilon,\boldsymbol{y}}({{\boldsymbol{\theta}}%
})\frac{\partial}{\partial\varepsilon}\boldsymbol{\xi}_{\beta,\varepsilon
,\boldsymbol{y}}^{T}({{\boldsymbol{\theta}}})  &  =\int\boldsymbol{u}%
_{{\boldsymbol{\theta}}}(\boldsymbol{x})f_{{\boldsymbol{\theta}}}^{\beta
}(\boldsymbol{x})\left(  (1-\varepsilon)f_{\theta_{0}}(\boldsymbol{x}%
)+\epsilon\Delta_{\boldsymbol{y}}\right)  d\boldsymbol{x}\left(
\boldsymbol{u}_{{\boldsymbol{\theta}}}(\boldsymbol{y})f_{{\boldsymbol{\theta}%
}}^{\beta}(\boldsymbol{y})-\boldsymbol{\xi}_{\beta}({{\boldsymbol{\theta}}%
})\right)  ^{T},\\
\boldsymbol{\xi}_{\beta,0,\boldsymbol{y}}({{\boldsymbol{\theta}}}%
)\frac{\partial}{\partial\varepsilon}\left.  \boldsymbol{\xi}_{\beta
,\varepsilon,\boldsymbol{y}}^{T}({{\boldsymbol{\theta}}})\right\vert
_{\varepsilon=0}  &  =\boldsymbol{\xi}_{\beta}({{\boldsymbol{\theta}}})\left(
\boldsymbol{u}_{{\boldsymbol{\theta}}}(\boldsymbol{y})f_{{\boldsymbol{\theta}%
}}^{\beta}(\boldsymbol{y})-\boldsymbol{\xi}_{\beta}({{\boldsymbol{\theta}}%
})\right)  ^{T}=\boldsymbol{\xi}_{\beta}({{\boldsymbol{\theta}}}%
)\boldsymbol{u}_{{\boldsymbol{\theta}}}^{T}(\boldsymbol{y}%
)f_{{\boldsymbol{\theta}}}^{\beta}(\boldsymbol{y})-\boldsymbol{\xi}_{\beta
}({{\boldsymbol{\theta}}})\boldsymbol{\xi}_{\beta}^{T}({{\boldsymbol{\theta}}%
}),
\end{align*}
we get from equation (\ref{k})%
\begin{align}
\frac{\partial}{\partial\varepsilon}\left.  \boldsymbol{K}_{\beta
,\varepsilon,\boldsymbol{y}}(\boldsymbol{\theta})\right\vert _{\varepsilon=0}
&  =\boldsymbol{u}_{{\boldsymbol{\theta}}}(\boldsymbol{y})\boldsymbol{u}%
_{{\boldsymbol{\theta}}}^{T}(\boldsymbol{y})f_{{\boldsymbol{\theta}}}^{2\beta
}(\boldsymbol{y})-\int\boldsymbol{u}_{{\boldsymbol{\theta}}}(\boldsymbol{x}%
)\boldsymbol{u}_{{\boldsymbol{\theta}}}^{T}(\boldsymbol{x}%
)f_{{\boldsymbol{\theta}}}^{2\beta+1}(\boldsymbol{x})d\boldsymbol{x}%
\nonumber\\
&  -\boldsymbol{\xi}_{\beta,0,\boldsymbol{y}}({{\boldsymbol{\theta}}}%
)\frac{\partial}{\partial\varepsilon}\left.  \boldsymbol{\xi}_{\beta
,\varepsilon,\boldsymbol{y}}^{T}({{\boldsymbol{\theta}}})\right\vert
_{\varepsilon=0}-\left(  \boldsymbol{\xi}_{\beta,0,\boldsymbol{y}%
}({{\boldsymbol{\theta}}})\frac{\partial}{\partial\varepsilon}\left.
\boldsymbol{\xi}_{\beta,\varepsilon,\boldsymbol{y}}^{T}({{\boldsymbol{\theta}%
}})\right\vert _{\varepsilon=0}\right)  ^{T}\nonumber\\
&  =\boldsymbol{u}_{{\boldsymbol{\theta}}}(\boldsymbol{y})\boldsymbol{u}%
_{{\boldsymbol{\theta}}}^{T}(\boldsymbol{y})f_{{\boldsymbol{\theta}}}^{2\beta
}(\boldsymbol{y})-\boldsymbol{K}_{\beta}(\boldsymbol{\theta})-\boldsymbol{\xi
}_{\beta}({{\boldsymbol{\theta}}})\boldsymbol{u}_{{\boldsymbol{\theta}}}%
^{T}(\boldsymbol{y})f_{{\boldsymbol{\theta}}}^{\beta}(\boldsymbol{y}%
)-\boldsymbol{u}_{{\boldsymbol{\theta}}}(\boldsymbol{y})\boldsymbol{\xi
}_{\beta}^{T}({{\boldsymbol{\theta}}})f_{{\boldsymbol{\theta}}}^{\beta
}(\boldsymbol{y})+\boldsymbol{\xi}_{\beta}({{\boldsymbol{\theta}}%
})\boldsymbol{\xi}_{\beta}^{T}({{\boldsymbol{\theta}}})\nonumber\\
&  =-\boldsymbol{K}_{\beta}(\boldsymbol{\theta})-\boldsymbol{\xi}_{\beta
}({{\boldsymbol{\theta}}})\boldsymbol{u}_{{\boldsymbol{\theta}}}%
^{T}(\boldsymbol{y})f_{{\boldsymbol{\theta}}}^{\beta}(\boldsymbol{y}%
)-\boldsymbol{u}_{{\boldsymbol{\theta}}}(\boldsymbol{y})\boldsymbol{\xi
}_{\beta}^{T}({{\boldsymbol{\theta}}})f_{{\boldsymbol{\theta}}}^{\beta
}(\boldsymbol{y})+\boldsymbol{u}_{{\boldsymbol{\theta}}}(\boldsymbol{y}%
)\boldsymbol{u}_{{\boldsymbol{\theta}}}^{T}(\boldsymbol{y}%
)f_{{\boldsymbol{\theta}}}^{2\beta}(\boldsymbol{y})+\boldsymbol{\xi}_{\beta
}({{\boldsymbol{\theta}}})\boldsymbol{\xi}_{\beta}^{T}({{\boldsymbol{\theta}}%
})\nonumber\\
&  =-\boldsymbol{K}_{\beta}(\boldsymbol{\theta})-\left(  \boldsymbol{\xi
}_{\beta}({{\boldsymbol{\theta}}})-\boldsymbol{u}_{{\boldsymbol{\theta}}%
}(\boldsymbol{y})f_{{\boldsymbol{\theta}}}^{\beta}(\boldsymbol{y})\right)
\left(  \boldsymbol{\xi}_{\beta}({{\boldsymbol{\theta}}})-\boldsymbol{u}%
_{{\boldsymbol{\theta}}}(\boldsymbol{y})f_{{\boldsymbol{\theta}}}^{\beta
}(\boldsymbol{y})\right)  ^{T}. \label{k0}%
\end{align}
Using (\ref{j}) and (\ref{k0}), we get%
\begin{align}
\boldsymbol{J}_{\beta}^{-1}(\boldsymbol{\theta})\frac{\partial}{\partial
\varepsilon}\boldsymbol{J}_{\beta,\varepsilon,\boldsymbol{y}}%
(\boldsymbol{\theta})\boldsymbol{\Sigma}_{\beta}(\boldsymbol{\theta})  &
=\beta\boldsymbol{\Sigma}_{\beta}(\boldsymbol{\theta})+f_{\boldsymbol{\theta
}_{0}}^{\beta}(\boldsymbol{y})\boldsymbol{J}_{\beta}^{-1}(\boldsymbol{\theta
})\left(  \boldsymbol{I}_{\boldsymbol{\theta}}(\boldsymbol{y})-\beta
\boldsymbol{u}_{{\boldsymbol{\theta}}}(\boldsymbol{y})\boldsymbol{u}%
_{{\boldsymbol{\theta}}}^{T}(\boldsymbol{y})\right)  \boldsymbol{\Sigma
}_{\beta}(\boldsymbol{\theta})\nonumber\\
&  -\boldsymbol{J}_{\beta}^{-1}(\boldsymbol{\theta})\int\boldsymbol{I}%
_{\boldsymbol{\theta}}(\boldsymbol{x})f_{{\boldsymbol{\theta}}}^{1+\beta
}(\boldsymbol{x})d\boldsymbol{x\Sigma}_{\beta}(\boldsymbol{\theta}),
\label{d2}%
\end{align}
and%
\begin{align}
\boldsymbol{J}_{\beta}^{-1}(\boldsymbol{\theta})\frac{\partial}{\partial
\varepsilon}\left.  \boldsymbol{K}_{\beta,\varepsilon,\boldsymbol{y}%
}(\boldsymbol{\theta})\right\vert _{\varepsilon=0}\boldsymbol{J}_{\beta}%
^{-1}(\boldsymbol{\theta})  &  =-\boldsymbol{\Sigma}_{\beta}%
(\boldsymbol{\theta}_{0})\nonumber\\
&  -\boldsymbol{J}_{\beta}^{-1}(\boldsymbol{\theta})\left(  \boldsymbol{\xi
}_{\beta}({{\boldsymbol{\theta}}})-\boldsymbol{u}_{{\boldsymbol{\theta}}%
}(\boldsymbol{y})f_{{\boldsymbol{\theta}}}^{\beta}(\boldsymbol{y})\right)
\left(  \boldsymbol{\xi}_{\beta}({{\boldsymbol{\theta}}})-\boldsymbol{u}%
_{{\boldsymbol{\theta}}}(\boldsymbol{y})f_{{\boldsymbol{\theta}}}^{\beta
}(\boldsymbol{y})\right)  ^{T}\boldsymbol{J}_{\beta}^{-1}(\boldsymbol{\theta
}), \label{d3}%
\end{align}
respectively. Combining (\ref{sigma}), (\ref{d2}), (\ref{d3}) we get%
\begin{align*}
\boldsymbol{\Sigma}_{\beta}^{-1}(\boldsymbol{\theta})\frac{\partial}%
{\partial\varepsilon}\left.  \boldsymbol{\Sigma}_{\beta,\varepsilon
,\boldsymbol{y}}(\boldsymbol{\theta})\right\vert _{\varepsilon=0}  &
=-2\beta\boldsymbol{I}_{p}-\boldsymbol{J}_{\beta}(\boldsymbol{\theta}%
_{0})\boldsymbol{K}_{\beta}^{-1}(\boldsymbol{\theta}_{0})\left(
f_{\boldsymbol{\theta}_{0}}^{\beta}(\boldsymbol{y})\left(  \boldsymbol{I}%
_{\boldsymbol{\theta}}(\boldsymbol{y})-\beta\boldsymbol{u}%
_{{\boldsymbol{\theta}}}(\boldsymbol{y})\boldsymbol{u}_{{\boldsymbol{\theta}}%
}^{T}(\boldsymbol{y})\right)  +\int\boldsymbol{I}_{\boldsymbol{\theta}%
}(\boldsymbol{x})f_{{\boldsymbol{\theta}}}^{1+\beta}(\boldsymbol{x}%
)d\boldsymbol{x}\right)  \boldsymbol{\Sigma}_{\beta}(\boldsymbol{\theta})\\
&  -\boldsymbol{\Sigma}_{\beta}(\boldsymbol{\theta})\left(
f_{\boldsymbol{\theta}_{0}}^{\beta}(\boldsymbol{y})\left(  \boldsymbol{I}%
_{\boldsymbol{\theta}}(\boldsymbol{y})-\beta\boldsymbol{u}%
_{{\boldsymbol{\theta}}}(\boldsymbol{y})\boldsymbol{u}_{{\boldsymbol{\theta}}%
}^{T}(\boldsymbol{y})\right)  +\int\boldsymbol{I}_{\boldsymbol{\theta}%
}(\boldsymbol{x})f_{{\boldsymbol{\theta}}}^{1+\beta}(\boldsymbol{x}%
)d\boldsymbol{x}\right)  \boldsymbol{J}_{\beta}(\boldsymbol{\theta}%
_{0})\boldsymbol{K}_{\beta}^{-1}(\boldsymbol{\theta}_{0})\\
&  -\boldsymbol{I}_{p}-\boldsymbol{J}_{\beta}(\boldsymbol{\theta}%
_{0})\boldsymbol{K}_{\beta}^{-1}(\boldsymbol{\theta}_{0})\left(
\boldsymbol{\xi}_{\beta}({{\boldsymbol{\theta}}})-\boldsymbol{u}%
_{{\boldsymbol{\theta}}}(\boldsymbol{y})f_{{\boldsymbol{\theta}}}^{\beta
}(\boldsymbol{y})\right)  \left(  \boldsymbol{\xi}_{\beta}%
({{\boldsymbol{\theta}}})-\boldsymbol{u}_{{\boldsymbol{\theta}}}%
(\boldsymbol{y})f_{{\boldsymbol{\theta}}}^{\beta}(\boldsymbol{y})\right)
^{T}\boldsymbol{J}_{\beta}^{-1}(\boldsymbol{\theta}),
\end{align*}
and thus the theorem follows from%
\begin{align*}
\mathrm{trace}\left(  \boldsymbol{\Sigma}_{\beta}^{-1}(\boldsymbol{\theta
})\frac{\partial}{\partial\varepsilon}\left.  \boldsymbol{\Sigma}%
_{\beta,\varepsilon,\boldsymbol{y}}(\boldsymbol{\theta})\right\vert
_{\varepsilon=0}\right)   &  =-\left(  2\beta+1\right)  p-\mathrm{trace}%
\left(  \left(  \boldsymbol{\xi}_{\beta}({{\boldsymbol{\theta}}}%
)-\boldsymbol{u}_{{\boldsymbol{\theta}}}(\boldsymbol{y})f_{{\boldsymbol{\theta
}}}^{\beta}(\boldsymbol{y})\right)  \left(  \boldsymbol{\xi}_{\beta
}({{\boldsymbol{\theta}}})-\boldsymbol{u}_{{\boldsymbol{\theta}}%
}(\boldsymbol{y})f_{{\boldsymbol{\theta}}}^{\beta}(\boldsymbol{y})\right)
^{T}\boldsymbol{K}_{\beta}^{-1}(\boldsymbol{\theta})\right) \\
&  -2\mathrm{trace}\left(  \left(  f_{\boldsymbol{\theta}_{0}}^{\beta
}(\boldsymbol{y})\left(  \boldsymbol{I}_{\boldsymbol{\theta}}(\boldsymbol{y}%
)-\beta\boldsymbol{u}_{{\boldsymbol{\theta}}}(\boldsymbol{y})\boldsymbol{u}%
_{{\boldsymbol{\theta}}}^{T}(\boldsymbol{y})\right)  +\int\boldsymbol{I}%
_{\boldsymbol{\theta}}(\boldsymbol{x})f_{{\boldsymbol{\theta}}}^{1+\beta
}(\boldsymbol{x})d\boldsymbol{x}\right)  \boldsymbol{J}_{\beta}^{-1}%
(\boldsymbol{\theta})\right)
\end{align*}
and taking into account%
\[
\mathcal{IF}_{2}(\boldsymbol{y},W_{\beta}^{0},F_{\boldsymbol{\theta}_{0}%
})=\mathrm{trace}\left(  \left(  \boldsymbol{\xi}_{\beta}({{\boldsymbol{\theta
}}})-\boldsymbol{u}_{{\boldsymbol{\theta}}}(\boldsymbol{y}%
)f_{{\boldsymbol{\theta}}}^{\beta}(\boldsymbol{y})\right)  \left(
\boldsymbol{\xi}_{\beta}({{\boldsymbol{\theta}}})-\boldsymbol{u}%
_{{\boldsymbol{\theta}}}(\boldsymbol{y})f_{{\boldsymbol{\theta}}}^{\beta
}(\boldsymbol{y})\right)  ^{T}\boldsymbol{K}_{\beta}^{-1}(\boldsymbol{\theta
})\right)  .
\]

\subsection{Proof of Theorem \textbf{\ref{THM:7Chi_infl_fact_slope_comp}%
\label{A7}}}

From (\ref{sigma}), (\ref{d2}), (\ref{d3}) we get
\begin{align*}
&  \boldsymbol{\Sigma}_{\beta}^{\ast-1}(\boldsymbol{\theta}_{0})\frac
{\partial}{\partial\varepsilon}\left.  \boldsymbol{\Sigma}_{\beta
,\varepsilon,\boldsymbol{y}}^{\ast}(\boldsymbol{\theta}_{0})\right\vert
_{\varepsilon=0}=-\beta\boldsymbol{I}_{r}\\
&  -\boldsymbol{\Sigma}_{\beta}^{\ast-1}(\boldsymbol{\theta}_{0}%
)\boldsymbol{M}^{T}(\boldsymbol{\theta}_{0})\boldsymbol{J}_{\beta}%
^{-1}(\boldsymbol{\theta}_{0})\left(  f_{\boldsymbol{\theta}_{0}}^{\beta
}(\boldsymbol{y})\left(  \boldsymbol{I}_{\boldsymbol{\theta}_{0}%
}(\boldsymbol{y})-\beta\boldsymbol{u}_{{\boldsymbol{\theta}_{0}}%
}(\boldsymbol{y})\boldsymbol{u}_{{\boldsymbol{\theta}_{0}}}^{T}(\boldsymbol{y}%
)\right)  -\int\boldsymbol{I}_{\boldsymbol{\theta}_{0}}(\boldsymbol{x}%
)f_{{\boldsymbol{\theta}_{0}}}^{1+\beta}(\boldsymbol{x})d\boldsymbol{x}%
\right)  \boldsymbol{\Sigma}_{\beta}(\boldsymbol{\theta}_{0})\boldsymbol{M}%
(\boldsymbol{\theta}_{0})\\
&  -\boldsymbol{M}^{T}(\boldsymbol{\theta}_{0})\boldsymbol{\Sigma}_{\beta
}(\boldsymbol{\theta}_{0})\left(  f_{\boldsymbol{\theta}_{0}}^{\beta
}(\boldsymbol{y})\left(  \boldsymbol{I}_{\boldsymbol{\theta}_{0}%
}(\boldsymbol{y})-\beta\boldsymbol{u}_{{\boldsymbol{\theta}_{0}}%
}(\boldsymbol{y})\boldsymbol{u}_{{\boldsymbol{\theta}_{0}}}^{T}(\boldsymbol{y}%
)\right)  -\int\boldsymbol{I}_{\boldsymbol{\theta}_{0}}(\boldsymbol{x}%
)f_{{\boldsymbol{\theta}_{0}}}^{1+\beta}(\boldsymbol{x})d\boldsymbol{x}%
\right)  \boldsymbol{J}_{\beta}^{-1}(\boldsymbol{\theta}_{0})\boldsymbol{M}%
(\boldsymbol{\theta}_{0})\boldsymbol{\Sigma}_{\beta}^{\ast-1}%
(\boldsymbol{\theta}_{0})\\
&  -\boldsymbol{I}_{r}-\boldsymbol{\Sigma}_{\beta}^{\ast-1}(\boldsymbol{\theta
}_{0})\boldsymbol{M}_{\beta}^{T}(\boldsymbol{\theta}_{0})\boldsymbol{J}%
_{\beta}^{-1}(\boldsymbol{\theta}_{0})\left(  \boldsymbol{\xi}_{\beta
}(\boldsymbol{\theta}_{0})-\boldsymbol{u}_{\boldsymbol{\theta}_{0}%
}(\boldsymbol{y})f_{\boldsymbol{\theta}_{0}}^{\beta}(\boldsymbol{y})\right)
\left(  \boldsymbol{\xi}_{\beta}(\boldsymbol{\theta}_{0})-\boldsymbol{u}%
_{\boldsymbol{\theta}_{0}}(\boldsymbol{y})f_{\boldsymbol{\theta}_{0}}^{\beta
}(\boldsymbol{y})\right)  ^{T}\boldsymbol{J}_{\beta}^{-1}(\boldsymbol{\theta
}_{0})\boldsymbol{M}(\boldsymbol{\theta}_{0}),
\end{align*}
and thus the theorem follows from
\begin{align*}
&  \mathrm{trace}\left(  \boldsymbol{\Sigma}_{\beta}^{-1}(\boldsymbol{\theta
}_{0})\frac{\partial}{\partial\varepsilon}\left.  \boldsymbol{\Sigma}%
_{\beta,\varepsilon,\boldsymbol{y}}(\boldsymbol{\theta}_{0})\right\vert
_{\varepsilon=0}\right)  =-\left(  2\beta+1\right)  r\\
&  -\mathrm{trace}\left(  \left(  \boldsymbol{\xi}_{\beta}%
({{\boldsymbol{\theta}_{0}}})-\boldsymbol{u}_{{\boldsymbol{\theta}_{0}}%
}(\boldsymbol{y})f_{{\boldsymbol{\theta}_{0}}}^{\beta}(\boldsymbol{y})\right)
\left(  \boldsymbol{\xi}_{\beta}({{\boldsymbol{\theta}_{0}}})-\boldsymbol{u}%
_{{\boldsymbol{\theta}_{0}}}(\boldsymbol{y})f_{{\boldsymbol{\theta}_{0}}%
}^{\beta}(\boldsymbol{y})\right)  ^{T}\boldsymbol{J}_{\beta}^{-1}%
(\boldsymbol{\theta}_{0})\boldsymbol{M}(\boldsymbol{\theta}_{0}%
)\boldsymbol{\Sigma}_{\beta}^{\ast-1}(\boldsymbol{\theta}_{0})\boldsymbol{M}%
^{T}(\boldsymbol{\theta}_{0})\boldsymbol{J}_{\beta}^{-1}(\boldsymbol{\theta
}_{0})\right) \\
&  -2\mathrm{trace}\left(  \left(  f_{\boldsymbol{\theta}_{0}}^{\beta
}(\boldsymbol{y})\left(  \boldsymbol{I}_{\boldsymbol{\theta}_{0}%
}(\boldsymbol{y})-\beta\boldsymbol{u}_{{\boldsymbol{\theta}_{0}}%
}(\boldsymbol{y})\boldsymbol{u}_{{\boldsymbol{\theta}_{0}}}^{T}(\boldsymbol{y}%
)\right)  -\int\boldsymbol{I}_{\boldsymbol{\theta}_{0}}(\boldsymbol{x}%
)f_{{\boldsymbol{\theta}_{0}}}^{1+\beta}(\boldsymbol{x})d\boldsymbol{x}%
\right)  \boldsymbol{\Sigma}_{\beta}(\boldsymbol{\theta}_{0})\boldsymbol{M}%
(\boldsymbol{\theta}_{0})\boldsymbol{\Sigma}_{\beta}^{\ast-1}%
(\boldsymbol{\theta}_{0})\boldsymbol{M}_{\beta}^{T}(\boldsymbol{\theta}%
_{0})\boldsymbol{J}_{\beta}^{-1}(\boldsymbol{\theta}_{0})\right)
\end{align*}
and taking into account%
\begin{align*}
&  \mathcal{IF}_{2}(\boldsymbol{y},W_{\beta},F_{\boldsymbol{\theta}_{0}})\\
&  =\mathrm{trace}\left(  \left(  \boldsymbol{\xi}_{\beta}%
({{\boldsymbol{\theta}_{0}}})-\boldsymbol{u}_{{\boldsymbol{\theta}_{0}}%
}(\boldsymbol{y})f_{{\boldsymbol{\theta}_{0}}}^{\beta}(\boldsymbol{y})\right)
\left(  \boldsymbol{\xi}_{\beta}({{\boldsymbol{\theta}_{0}}})-\boldsymbol{u}%
_{{\boldsymbol{\theta}_{0}}}(\boldsymbol{y})f_{{\boldsymbol{\theta}_{0}}%
}^{\beta}(\boldsymbol{y})\right)  ^{T}\boldsymbol{J}_{\beta}^{-1}%
(\boldsymbol{\theta}_{0})\boldsymbol{M}(\boldsymbol{\theta}_{0}%
)\boldsymbol{\Sigma}_{\beta}^{\ast-1}(\boldsymbol{\theta}_{0})\boldsymbol{M}%
^{T}(\boldsymbol{\theta}_{0})\boldsymbol{J}_{\beta}^{-1}(\boldsymbol{\theta
}_{0})\right)  .
\end{align*}

\end{document}